\documentclass[10pt,reqno]{amsart}
\usepackage{wth}
\usepackage[foot]{amsaddr}
\usepackage{bbm}
\usepackage{tikz}
\usepackage[position=b,labelformat=simple,labelfont=rm]{subcaption}
\usepackage{microtype}

\title{Universality of Approximate Message Passing algorithms and
tensor networks}
\author{Tianhao Wang}\email{tianhao.wang@yale.edu}
\author{Xinyi Zhong}\email{xinyi.zhong@yale.edu}
\author{Zhou Fan}\email{zhou.fan@yale.edu}
\address{TW, XZ, ZF: Department of Statistics and Data Science, Yale University}
\date{\today}

\def\a{\alpha}
\def\b{\beta}

\def\s{\mathbf{s}}
\def\u{\mathbf{u}}
\def\v{\mathbf{v}}
\def\x{\mathbf{x}}
\def\y{\mathbf{y}}
\def\z{\mathbf{z}}
\def\f{\mathbf{f}}
\def\g{\mathbf{g}}
\def\W{\mathbf{W}}
\def\S{\mathbf{S}}
\def\G{\mathbf{G}}
\def\D{\mathbf{D}}
\def\A{\mathbf{A}}
\def\X{\mathbf{X}}
\def\Y{\mathbf{Y}}

\def\E{\mathbb{E}}
\def\R{\mathbb{R}}

\def\Var{\operatorname{Var}}

\def\cU{\mathcal{U}}
\def\cV{\mathcal{V}}
\def\cE{\mathcal{E}}
\def\diag{\operatorname{diag}}
\def\val{\operatorname{val}}
\def\limval{\operatorname{lim-val}}

\def\op{\mathrm{op}}
\def\Normal{\mathcal{N}}
\def\GOE{\operatorname{GOE}}

\def\toW{\overset{W}{\to}}
\def\toWp{\overset{W_p}{\to}}
\def\toWtwo{\overset{W_2}{\to}}

\def\1{\mathbf{1}}

\def\bW{{\mathbf{W}}}
\def\bbE{{\mathbb{E}}}

\def\bi{\mathbf{i}}
\def\bx{\mathbf{x}}
\def\bs{\mathbf{s}}

\def\calE{\mathcal{E}}
\def\calV{\mathcal{V}}

\def\bu{{\mathbf{u}}}
\def\bz{{\mathbf{z}}}
\def\bff{{\mathbf{f}}}

\def\eps{\varepsilon}
\def\Tr{\operatorname{Tr}}
\def\Haar{\operatorname{Haar}}
\def\OO{\mathbb{O}}

\newcommand{\degext}{\deg_\mathrm{ext}}
\newcommand{\ext}{\mathrm{ext}}
\newcommand{\tdiag}{\mathrm{diag}}

\begin{document}

\begin{abstract}
Approximate Message Passing (AMP) algorithms provide a valuable tool for
studying mean-field approximations and dynamics in a variety of applications.
Although these algorithms are often first derived for matrices having
independent Gaussian entries or satisfying rotational invariance in law, their
state evolution characterizations are expected to hold over larger universality
classes of random matrix ensembles.

We develop several new results on AMP universality. For AMP algorithms tailored
to independent Gaussian entries, we show that their state evolutions hold over
broadly defined generalized Wigner and white noise ensembles, including matrices
with heavy-tailed entries and heterogeneous entrywise variances that may arise
in data applications. For AMP algorithms tailored to rotational invariance in
law, we show that their state evolutions hold over
delocalized sign-and-permutation-invariant
matrix ensembles that have a limit distribution over the diagonal,
including sensing matrices composed of subsampled Hadamard or Fourier transforms
and diagonal operators.

We establish these results via a simplified moment-method proof, reducing AMP
universality to the study of products of random matrices and diagonal tensors
along a tensor network. As a by-product of our analyses, we show that the
aforementioned matrix ensembles satisfy a notion of asymptotic freeness with
respect to such tensor networks, which parallels usual definitions of freeness
for traces of matrix products.
\end{abstract}

\maketitle

\section{Introduction}\label{sec:intro}

Approximate Message Passing (AMP) algorithms are a general family of iterative
algorithms, driven by a random matrix $\W$,
whose iterates admit a simple distributional characterization in the asymptotic
limit of increasing dimensions. Their origins may be traced separately in the
engineering, statistics, and probability literatures
\cite{kabashima2003cdma,donoho2009message,bolthausen2014iterative}, where these
algorithms have since provided an important tool for studying mean-field
phenomena in many probabilistic models. Without seeking
to be exhaustive, we mention here their applications to analyses
of spin glass and perceptron models
\cite{bolthausen2018morita,ding2019capacity,fan2021replica,bolthausen2021gardner,fan2022tap},
recovery thresholds and asymptotic phenomena in high-dimensional statistical
models
\cite{donoho2009message,rangan2011generalized,rangan2012iterative,
donoho2013information,donoho2016high,dia2016mutual,bu2019algorithmic,sur2019likelihood,
montanari2021estimation,mondelli2021pca,zhong2021approximate,li2021minimum},
and mean-field dynamics of other first-order optimization
algorithms including discrete-time and continuous-time gradient descent
\cite{celentano2020estimation,celentano2021high}.
We refer readers to \cite{feng2022unifying} for a recent review. 

Asymptotic distributional characterizations of the AMP iterates, known as their
\emph{state evolutions}, are often first proved for orthogonally invariant
matrices $\W$ using an inductive
conditioning technique. For $\W$ with i.i.d.\ Gaussian entries, this
method was developed in \cite{bolthausen2014iterative,bayati2011dynamics}
and has been extended to analyze AMP algorithms of increasing generality in
\cite{rangan2011generalized,javanmard2013state,berthier2020state,montanari2021estimation,gerbelot2021graph}.
For $\W$ satisfying rotational invariance in law, a similar technique
has been applied to analyze various AMP algorithms in
\cite{rangan2019vector,schniter2016vector,ma2017orthogonal,takeuchi2017rigorous,
takeuchi2020convolutional,takeuchi2021bayes,liu2021memory,fan2022approximate},
with a parallel line of work
\cite{cakmak2014s,opper2016theory,ccakmak2019memory,ccakmak2020dynamical}
deriving related algorithms using non-rigorous methods of
dynamic functional theory.

It is expected---and in some settings known---that the state evolution
characterizations of AMP algorithms should extend beyond orthogonally 
invariant matrices, to describe also the limit distributions of iterates
when applied to broader universality classes of random matrix ensembles. 
For example, it was shown in \cite{bayati2015universality} that AMP algorithms
designed for i.i.d.\ Gaussian matrices and having
polynomial non-linearities admit state evolutions that are universal across
matrices with sub-Gaussian entries of common variance.
In \cite{chen2021universality}, universality over a similar matrix class
for AMP with Lipschitz non-linearities
was proven using a different Gaussian interpolation method, and extended
to spectrally initialized algorithms for spiked matrix models.
Moving beyond matrices with independent entries,
in \cite{dudeja2020universality} it was shown that the state evolution of a linear 
AMP algorithm for phase retrieval holds universally for
sub-sampled Hadamard matrices. Recently, results of
\cite{dudeja2022universality,dudeja2022spectral}---fruit of parallel research
efforts---showed universality for AMP
algorithms having divergence-free non-linearities over a broad
model of semi-random matrices with randomly signed rows/columns and delocalized
entries. The latter work \cite{dudeja2022spectral} also applied these results
to establish universality classes of matrices for more general first-order
iterative algorithms, including proximal gradient methods and general versions
of AMP. We discuss the relation of these results to our work in
more detail at the conclusion of the following section.

\subsection{Contributions}

Our current work has the two-fold goal of extending the scope of some of these
universality results of
\cite{bayati2015universality,chen2021universality,dudeja2022universality},
and of presenting a more direct and elementary proof for AMP universality. We summarize our
contributions as follows:
\begin{enumerate}
\item For AMP algorithms designed for i.i.d.\ Gaussian matrices, we show that their state
evolutions hold more broadly over generalized Wigner and white
noise ensembles, with entries having potentially heteroskedastic variances and
higher moments growing rapidly with the dimension $n$.
This includes standardized adjacency matrices
of sparse random graphs down to sparsity levels of $(\log n)/n$, as well as data
matrices arising in contexts of count-valued and missing observations after
applying practical row and column normalization schemes. We discuss two
motivating applications in Examples \ref{ex:SBM} and \ref{ex:biwhitenedcount}
of Section \ref{subsec:examples}. In the random matrix theory literature,
global spectral laws and spectral CLTs for related ensembles
were studied in \cite{anderson2006clt}, and universality of local
spectral statistics in \cite{erdHos2012bulk,erdHos2012rigidity}.
\item For AMP algorithms designed for rotationally invariant matrix
ensembles, we show that their state evolutions hold over universality classes of
``generalized invariant matrices''
that satisfy only invariances of permutation and sign and 
whose generated algebra over the diagonal, in the sense of
\cite{au2021freeness}, consists of matrices with delocalized entries and common
normalized trace.
Importantly, this includes matrices composed of subsampled Hadamard
or discrete Fourier transforms and diagonal operators, which admit fast
matrix-vector multiplication for signal processing applications. We discuss a
specific application to universality of the compressed sensing
phase transition for AMP \cite{donoho2009observed,bayati2015universality}
in Example \ref{ex:CS} of Section \ref{subsec:examples}. Related models of
permutation-and-sign-invariant matrices have been studied in the context of
asymptotic liberation in \cite{anderson2014asymptotically}.
\item We introduce a simplified two-step proof of AMP universality, in the 
first step reducing universality to the study of products of $\W$ with
diagonal tensors along a tensor network, and in the second step establishing
universality of the values of these matrix-tensor products. The second step
admits a simple combinatorial analysis for all of the preceding matrix ensembles. Our argument for the first step is general and holds irrespective of the specific matrix ensemble. 
We propose this two-step proof framework
in part to enable easier extensions of AMP universality to other random matrix 
models (e.g.\ having sufficiently weak or short-range correlation across
entries) as this need arises in applications.
\item For symmetric matrices $\W \in \R^{n \times n}$, our definition of a tensor
network is a natural generalization of expressions of the form
\[\frac{1}{n}\u^\top \W\bT_1\W\bT_2 \cdots \bT_k\W\v\]
for deterministic vectors $\u,\v$ and diagonal matrices
$\bT_1,\ldots,\bT_k$ to expressions involving higher-order diagonal tensors.
As a by-product of our analyses, we show for both the preceding classes of
generalized Wigner and generalized invariant matrices
$\W$ that they satisfy a notion of
asymptotic freeness with respect to such tensor networks, namely, that if all
diagonal tensors have asymptotically vanishing normalized trace, then 
evaluations of expressions of this form are also 0 in the asymptotic limit. This
is parallel to notions of asymptotic freeness \cite{voiculescu1992free}, usually
defined with respect to normalized traces of matrix products,
in settings of products with higher-order tensors. 
Our analysis of tensor
networks has also similarities to the analysis of graph observables in the
theory of traffic freeness developed in \cite{male2020traffic}.
\end{enumerate}

Our proofs use a moment-method and polynomial approximation strategy,
similar to \cite{bayati2015universality}. In heuristic derivations of AMP
algorithms from belief propagation for matrices in the Gaussian universality
class, the Onsager correction terms arise from the removal of
single-step-backtracking messages. The arguments of
\cite{bayati2015universality} showed a corresponding
equivalence between such AMP algorithms and a
tensorial unfolding of AMP using non-backtracking paths. To our knowledge, the
correction terms in the algorithms of \cite{fan2022approximate} for
rotationally-invariant ensembles do not have a similar combinatorial
interpretation, motivating us to analyze a
simpler tensorial unfolding without non-backtracking structure.
Our results for the Gaussian universality class may be obtained via either
approach; we take the opportunity to present unified proofs for both the
Gaussian and non-Gaussian universality classes using the same unfolding,
and to simplify
the polynomial approximation arguments of \cite{bayati2015universality}
using more recent state evolution results of
\cite{fan2022approximate} for AMP with non-Lipschitz functions.
We remark that, as in the AMP universality analysis
of \cite{chen2021universality} which developed a
different continuous interpolation argument, our method of proof applies also
to more general first-order iterative algorithms of the form studied in
\cite{celentano2020estimation} that are characterizable by an asymptotic state
evolution.

Our analyses for generalized invariant ensembles (Definitions
\ref{def:syminvariant} and \ref{def:rectinvariant}) are complementary to those of
the recent works \cite{dudeja2022universality,dudeja2022spectral}, which
studied an important family of Vector-AMP style
methods that have divergence-free non-linearities
\cite{schniter2016vector,ma2017orthogonal,takeuchi2017rigorous,ccakmak2019memory}.
As discussed in \cite{dudeja2022spectral},
the universality classes for these algorithms are broader than that of the more
general AMP algorithms we study here, for example containing matrices
with differing spectral distributions having common second moment.
\cite{dudeja2022universality,dudeja2022spectral} prove universality of these
algorithms for semi-random sign-invariant matrices and i.i.d.\ side
information vectors, by developing a Hermite-polynomial unfolding of the
AMP iterations and leveraging the vanishing of certain terms in this unfolding
due to the divergence-free form. The latter work \cite{dudeja2022spectral}
extends this result to also derive certain Spectral and Strongly Semi-Random
universality classes for first-order algorithms
that do not have this divergence-free structure. Our methods here establish
universality over a class of matrices that has similarities to, and is
partially inspired by, these latter classes studied in \cite{dudeja2022spectral}
(c.f.\ Proposition \ref{prop:syminvariant}(b)). We obtain these results via
an alternative
analysis of a simpler tensorial unfolding in the standard monomial basis.
As we discuss in Remark \ref{remark:sym_tn_universality}, our proofs also
establish the existence and universality of the limit empirical distribution
of iterates for first-order methods applied to matrices beyond
the orthogonally invariant universality class, suggesting the possible
development of new iterative algorithms with characterizable state evolutions
for such matrices.

\subsection{Notation}
We denote entries of $\x \in \R^n$ and $\W \in \R^{n \times n}$ as
$x[i]$ and $W[i,j]$. For
vectors $\x_1,\ldots,\x_k \in \R^n$ and a random vector $(X_1,\ldots,X_k)$,
we write
\[(\x_1,\ldots,\x_k) \toWp (X_1,\ldots,X_k) \text{ as } n \to \infty\]
for the Wasserstein-$p$ convergence of the empirical
distribution of rows of $(\x_1,\ldots,\x_k) \in \R^{n \times k}$ to the joint
law of $(X_1,\ldots,X_k)$. This means, for any
continuous function $f:\R^k \to \R$ satisfying
\begin{equation}\label{eq:polygrowth}
|f(x_1,\ldots,x_k)| \leq C(1+\|(x_1,\ldots,x_k)\|_2^p)
\text{ for a constant } C>0,
\end{equation}
we have as $n \to \infty$
\begin{equation}\label{eq:Wconvergence}
\frac{1}{n}\sum_{i=1}^n f(x_1[i],\ldots,x_k[i]) \to \E[f(X_1,\ldots,X_k)].
\end{equation}
We write
\[(\x_1,\ldots,\x_k) \toW (X_1,\ldots,X_k)\]
to mean that the above Wasserstein-$p$ convergence holds
for every order $p \geq 1$.

For a function $f:\R^k \to \R$ and vectors $\x_1,\ldots,\x_k \in
\R^n$, we denote by $f(\x_1,\ldots,\x_k) \in \R^n$ the evaluation of $f(\cdot)$
on each row of $(\x_1,\ldots,\x_k) \in \R^{n \times k}$.
We write $\langle \cdot \rangle$ for the
empirical average of the coordinates of a vector, and introduce the shorthand
$\x_{1:k}=(\x_1,\ldots,\x_k)$ and $X_{1:k}=(X_1,\ldots,X_k)$. Thus
(\ref{eq:Wconvergence}) may be expressed as
$\langle f(\x_{1:k}) \rangle \to \E[f(X_{1:k})]$.

For vectors $\x,\y \in \R^n$, $\x \odot \y \in \R^n$ is their entrywise
product. $\diag(\x) \in \R^{n \times n}$ or $\diag(\x) \in \R^{n \times \cdots
\times n}$ denotes the diagonal matrix or tensor with $\x$ along the main
diagonal, i.e.\ $\diag(\x)[i,\ldots,i] = \x[i]$ and $\diag(\x)$ has all
other entries equal to 0.
For $\x \in \R^{\min(m,n)}$, we write also $\diag(\x) \in
\R^{m \times n}$ for the rectangular diagonal matrix where each $(i,i)$ entry
is $x[i]$; we will indicate the dimensions if needed to disambiguate these
notations. $\|\W\|_\op$ is the $\ell_2 \to \ell_2$ operator norm of the matrix
$\W$. We denote $[n]=\{1,\ldots,n\}$, and reserve Roman letters
$i,j,\ldots$ for indices in $[n]$ and Greek letters
$\alpha,\beta,\ldots$ for indices in $[m]$.

\section{Main results}\label{sec:result}

\subsection{Universality of AMP algorithms for symmetric matrices}\label{sec:result_sym}

Let $\W \in \R^{n \times n}$ be a symmetric random matrix. Consider an
initialization $\u_1 \in \R^n$ and auxiliary ``side information'' vectors
$\f_1,\ldots,\f_k \in \R^n$, independent of $\W$. In applications, such side information
vectors may play the role of the external
field in spin glass models, the true signal vector in spiked matrix models, or
the signal and residual error vectors in regression models. We refer
to~\cite{bayati2011dynamics,rangan2011generalized,rangan2012iterative} for
several examples. Let $u_2,u_3,u_4,\ldots$ be a
sequence of non-linear functions, where $u_{t+1}:\R^{t+k} \to \R$.
We study a general form for an AMP algorithm with separable non-linearities
that computes, for $t=1,2,3,\ldots$
\begin{subequations}
\label{eq:AMP}
\begin{align}
\z_t&=\W\u_t-\sum_{s=1}^t b_{ts}\u_s \label{eq:AMPz}\\
\u_{t+1}&=u_{t+1}(\z_1,\ldots,\z_t,\f_1,\ldots,\f_k)\label{eq:AMPu} 
\end{align}
\end{subequations}
where $\{b_{ts}\}_{s \leq t}$ are deterministic scalar ``Onsager correction''
coefficients. We will characterize the iterates of this algorithm in the large
system limit as $n \to \infty$, for fixed $k \geq 0$.

We assume throughout the following conditions for $(\u_1,\f_1,\ldots,\f_k)$.

\begin{assumption}\label{assump:ufconvergence}
Almost surely as $n \to \infty$,
\begin{equation}\label{eq:ufconvergence}
(\u_1,\f_1,\ldots,\f_k) \toW (U_1,F_1,\ldots,F_k)
\end{equation}
for a joint limit law $(U_1,F_1,\ldots,F_k)$ having finite moments of all
orders, where $\E[U_1^2]>0$. Furthermore, multivariate polynomials are
dense in the real $L^2$-space of functions $f:\R^{k+1} \to \R$ with inner-product
\[(f,g) \mapsto \E[f(U_1,F_1,\ldots,F_k)g(U_1,F_1,\ldots,F_k)].\]
\end{assumption}

\begin{remark}
The convergence (\ref{eq:ufconvergence}) holds, for example,
if rows of $(\u_1,\f_1,\ldots,\f_k) \in \R^{n \times (k+1)}$ are i.i.d.\
and equal in law to $(U_1,F_1,\ldots,F_k)$.
The density of polynomials holds if $\|(U_1,F_1,\ldots,F_k)\|_2$ has
finite moment generating function in a neighborhood of 0; see
\cite[Section 14.1 and Corollary 14.24]{schmudgen2017moment}.
\end{remark}

In an AMP algorithm, the coefficients $\{b_{ts}\}$ of (\ref{eq:AMP}) are defined
so that the iterates $\{\z_t\}$ are described by a simple \emph{state evolution}
in the asymptotic limit as $n \to \infty$. 
For $\W \sim \GOE(n)$ (c.f.\ Definition \ref{def:Wigner}), this may be
done as follows: Set $\bSigma_1=\E[U_1^2] \in \R^{1 \times 1}$. Inductively,
having defined $\bSigma_t \in \R^{t \times t}$,
let $Z_{1:t} \sim \Normal(0,\bSigma_t)$ be
independent of $(U_1,F_{1:k})$, set
$U_{s+1}=u_{s+1}(Z_{1:s},F_{1:k})$ for each $s=1,\ldots,t$, and define
\begin{equation}\label{eq:GOEbSigma}
\bSigma_{t+1}=(\E[U_rU_s])_{r,s=1}^{t+1} \in \R^{(t+1) \times (t+1)}.
\end{equation}
Let $b_{tt} = 0$, and for each $s < t$, define the coefficient 
$b_{ts}$ as
\begin{equation}\label{eq:GOEb}
b_{ts}=\E[\partial_s u_t(Z_{1:t-1},F_{1:k})]
\end{equation}
where $\partial_s u_t$ is the partial derivative of $u_t(\cdot)$ in its
$s^\text{th}$ argument. We will call (\ref{eq:GOEbSigma}) and (\ref{eq:GOEb}) the
\emph{GOE prescriptions} for $\bSigma_t$ and $b_{ts}$.
Results of \cite{bayati2011dynamics,javanmard2013state}
(see also \cite[Proposition 2.1]{montanari2021optimization} for this form)
then imply that, for any Lipschitz functions $u_t(\cdot)$,
the iterates of~(\ref{eq:AMP}) satisfy the state
evolution, almost surely as $n \to \infty$ for any fixed $t \geq 1$,
\[(\u_1,\f_1,\ldots,\f_k,\z_1,\ldots,\z_t) \toW
(U_1,F_1,\ldots,F_k,Z_1,\ldots,Z_t).\]
We note that a variant of this algorithm may instead use the empirical average
$b_{ts}=\langle \partial_s u_t(\z_{1:t-1},\f_{1:k}) \rangle$,
for which the same state evolution continues to hold
(cf.\ Remark~\ref{remark:symempirical}).

In \cite{fan2022approximate}, building upon work of \cite{opper2016theory},
an extension of this result was proven for a
larger class of orthogonally invariant matrices and non-linear functions:
We say that $\W$ is \emph{orthogonally invariant} if
it has spectral decomposition $\W=\bO\D\bO^\top$ where
$\bO \sim \Haar(\OO(n))$ is Haar-distributed
on the orthogonal group and independent of $\D=\diag(\bd)$.
Suppose that $\bd \toW D$ as $n \to \infty$, where $D$ represents the limit
spectral law of $\W$.
Set $\bSigma_1=\Var[D] \cdot \E[U_1^2] \in \R^{1 \times 1}$. Having defined
$\bSigma_t \in \R^{t \times t}$, let $Z_{1:t} \sim \Normal(0,\bSigma_t)$ be
independent of $(U_1,F_{1:k})$, let
$U_{s+1}=u_{s+1}(Z_{1:s},F_{1:k})$ for each $s=1,\ldots,t$, and define
\begin{equation}\label{eq:symorthobSigma}
\bSigma_{t+1}=\bSigma_{t+1}\Big(\{\E[U_rU_s]\}_{1 \leq r,s \leq t+1},
\{\E[\partial_r u_{s+1}(Z_{1:s},F_{1:k})]\}_{1 \leq r \leq s \leq t}\Big)
\in \R^{(t+1) \times (t+1)}
\end{equation}
for a continuous function $\bSigma_{t+1}(\cdot)$ whose form depends only on the law of $D$.
For each $s \leq t$ and a continuous function $b_{ts}(\cdot)$ whose 
form also depends only on the law of $D$, define
\begin{equation}\label{eq:symorthob}
b_{ts}=b_{ts}\Big(\{\E[U_qU_r]\}_{1 \leq q,r \leq t},
\{\E[\partial_q u_{r+1}(Z_{1:r},F_{1:k})]\}_{1 \leq q \leq r<t}\Big).
\end{equation}
We will call (\ref{eq:symorthobSigma}) and (\ref{eq:symorthob}) the
\emph{orthogonally invariant prescriptions} for $\bSigma_t$ and $b_{ts}$.
We refer to \cite[Section 4]{fan2022approximate} for
their precise functional forms, which will not be important for our current
work. When $\W \sim \GOE(n)$ and $D$ has Wigner's semicircle
law on $[-2,2]$, these reduce to the previous GOE prescriptions
of (\ref{eq:GOEbSigma}) and (\ref{eq:GOEb}). 
It was shown in \cite{fan2022approximate} that
for weakly differentiable functions $u_t(\cdot)$ whose derivatives have
at most polynomial growth, the iterates of (\ref{eq:AMP}) again satisfy the
state evolution, almost surely as $n \to \infty$ for any fixed $t \geq 1$,
\[(\u_1,\f_1,\ldots,\f_k,\z_1,\ldots,\z_t) \toW
(U_1,F_1,\ldots,F_k,Z_1,\ldots,Z_t).\]

Our main results are universality statements that extend the state evolution
characterizations of these AMP algorithms to more general
random matrix ensembles. Corresponding to $\W \sim \GOE(n)$, we
study the following universality class of generalized Wigner matrices, having
possibly heteroskedastic entrywise variances and heavy-tailed entries.

\begin{definition}\label{def:Wigner}
$\W \in \R^{n \times n}$ is a \emph{generalized Wigner matrix} with
(deterministic) variance profile $\S \in \R^{n \times n}$ if
\begin{enumerate}[(a)]
\item $\W$ is symmetric, and entries on and above the diagonal
$(W[i,j]: 1 \leq i \leq j \leq n)$ are independent.
\label{assump:Wigner:Wind}
\item Each $W[i,j]$ has mean 0, variance $n^{-1}S[i,j]$, and higher
moments satisfying, for each integer $p \geq 3$,
\[\lim_{n \to \infty} n\cdot \max_{i,j=1}^n \E[|W[i,j]|^p]=0.\]
\label{assump:Wigner:Wmoment}
\item For a constant $C>0$ independent of $n$,
\[\max_{i,j=1}^n S[i,j] \leq C \quad \text{ and } \quad
\lim_{n \to \infty} \max_{i=1}^n \bigg|\frac{1}{n}\sum_{j=1}^n
S[i,j]-1\bigg|=0.\]
\label{assump:Wigner:S}
\end{enumerate}
We write $\W \sim \GOE(n)$ for the special case where
$W[i,j] \sim \Normal(0,1/n)$ and $S[i,j]=1$ for all $i<j$,
and $W[i,i] \sim \Normal(0,2/n)$ and $S[i,i]=2$ for all $i$.
\end{definition}

The moment assumption in condition (b) weakens a uniform sub-Gaussianity
condition for $\sqrt{n}\,W[i,j]$ that is assumed in the previous AMP universality
results of \cite{bayati2015universality,chen2021universality} and that would
require instead $\E[|W[i,j]|^p] \lesssim n^{-p/2}$ for all $p \geq 3$.
This condition (b) is weak enough to encompass centered
and normalized adjacency matrices of sparse random graphs with slowly growing
average vertex degree. Condition (c) allows general patterns of entrywise
variances whose rows and columns have approximately the same sum, where we also
require in (\ref{eq:Wignerunaligned}) of Theorem \ref{thm:Wigner} below that
these rows and columns are ``asymptotically unaligned'' with the
initialization and side information vectors $\u_1,\f_1,\ldots,\f_k$.
We discuss two applications in Examples \ref{ex:SBM} and
\ref{ex:biwhitenedcount} of Section \ref{subsec:examples}.

The following theorem shows that the state evolution of AMP algorithms
for GOE random matrices remains valid for matrices $\W$ in this generalized
Wigner universality class.

\begin{theorem}\label{thm:Wigner}
Let $\W \in \R^{n \times n}$ be a generalized Wigner matrix with variance
profile $\S$, and let $\u_1,\f_1,\ldots,\f_k$ be independent of $\W$
and satisfy Assumption \ref{assump:ufconvergence}. Suppose that
\begin{enumerate}
\item Each function $u_{t+1}:\R^{t+k} \to \R$ is continuous, satisfies the 
polynomial growth condition \eqref{eq:polygrowth} for some order $p \geq 1$, 
and is Lipschitz in its first $t$ arguments.
\item $\|\W\|_\op<C$ for a constant $C>0$ almost surely for all large $n$.
\item Let $\s_i$ be the $i^\text{th}$ row of $\S$. For any fixed
polynomial function $q:\R^{k+1} \to \R$, almost surely as $n \to \infty$,
\begin{equation}\label{eq:Wignerunaligned}
\max_{i=1}^n \Big|\langle q(\u_1,\f_1,\ldots,\f_k) \odot \s_i \rangle
-\langle q(\u_1,\f_1,\ldots,\f_k) \rangle \cdot \langle \s_i \rangle
\Big| \to 0.
\end{equation}
\end{enumerate}
Let $\{b_{ts}\}$ and $\{\bSigma_t\}$ be defined by the GOE prescriptions
(\ref{eq:GOEbSigma}) and (\ref{eq:GOEb}), where each matrix $\bSigma_t$ is
non-singular. Then for any fixed $t \geq 1$, almost surely as $n \to \infty$,
the iterates of (\ref{eq:AMP}) satisfy
\[(\u_1,\f_1,\ldots,\f_k,\z_1,\ldots,\z_t) \toWtwo
(U_1,F_1,\ldots,F_k,Z_1,\ldots,Z_t)\]
where $(Z_1,\ldots,Z_t) \sim \Normal(0,\bSigma_t)$ is independent of
$(U_1,F_1,\ldots,F_k)$, i.e.\ this limit has the same joint law as described 
by the AMP state evolution for $\W \sim \GOE(n)$.
\end{theorem}

Next, corresponding to orthogonal invariance,
we study universality classes of matrices that are permutation-and-sign-invariant
in law
and that have limit distributions over the diagonal, in the following sense
inspired by \cite{au2021freeness}:
Let $\Delta: \RR^{n\times n} \to \RR^{n\times n}$ be the diagonal map that
preserves only the entries on the diagonal, i.e.
\[\Delta(\bM)=\diag(M[1,1],\ldots,M[n,n]) \in \R^{n \times n}.\]
Let $\Delta \langle \x \rangle$ denote the set of all words in $\x$ and
$\Delta(\cdot)$, for example
\[\x\x, \quad \x\Delta(\x\x)\x, \quad \Delta(\x\x\Delta(\x))\x,
\quad \x\x\x\Delta(\Delta(\x))\Delta(\x\x).\]
We refer to $\Delta\langle \x \rangle$ as the set of \emph{diagonal monomials}
in $\x$. For $p(\bx) \in \Delta\langle \x \rangle$ and $\bM \in \R^{n \times n}$,
we write $p(\bM) \in \R^{n \times n}$ for its evaluation at $\x=\bM$.

\begin{definition}\label{def:diagdistr}
The \emph{distribution over the diagonal} of $\bM$ is the mapping\footnote{We
define the distribution over the diagonal by the values of
$\tfrac{1}{n} \Tr p(\bM) \in \R$ rather than $\Delta(p(\bM)) \in \R^{n \times
n}$ as might be more standard in operator-valued free probability.}
\begin{align*}
p(\bx) \in \Delta \langle \bx \rangle \mapsto \tfrac{1}{n} \Tr p(\bM).
\end{align*}
Matrices $\bM \in \R^{n \times n}$ \emph{converge in diagonal distribution a.s.}
if $\lim_{n \to \infty} \tfrac{1}{n} \Tr p(\bM)$ exists almost surely
(and is finite) for every fixed $p(\bx) \in \Delta \langle \x \rangle$.
The \emph{limit diagonal distribution} of $\bM$, which we will refer to as
$\cD_\tdiag$, is then the mapping
\begin{align*}
p(\bx) \in \Delta\langle\bx\rangle \mapsto
\lim_{n\to\infty} \tfrac{1}{n} \Tr p(\bM).
\end{align*}
\end{definition} 

We remark that $\cD_\tdiag$ specifies the limit of
$\tfrac{1}{n} \Tr \bM^\nu$ for each fixed integer $\nu \geq 1$,
and hence also the limit spectral distribution
of $\bM$ when this distribution has compact support.

We call $\bPi=\bP\bXi \in
\R^{n \times n}$ a \emph{uniformly random signed permutation matrix} if
$\bXi=\diag(\xi[1],\ldots,\xi[n]) \in \R^{n \times n}$ where each diagonal entry 
$\xi[i]$ is independently chosen from $\{+1,-1\}$ with equal probability, and 
$\bP \in \R^{n \times n}$ is
a uniformly random permutation matrix independent of $\bXi$.
Note that for any symmetric matrix $\bM$
and signed permutation matrix
$\bPi$, we have $\Delta(\bPi\bM\bPi^\top)=\bPi\Delta(\bM)\bPi^\top$, so also
$p(\bPi\bM\bPi^\top)=\bPi\,p(\bM)\,\bPi^\top$ for every diagonal monomial
$p(\bx) \in \Delta\langle \bx \rangle$. In particular, $\bM$ and
$\bPi\bM\bPi^\top$ have the same distributions over the diagonal.
The following then defines our universality class.

\begin{definition}\label{def:syminvariant}
$\W=\bPi \bM \bPi^\top\in \R^{n \times n}$ is a 
\emph{symmetric generalized invariant matrix}\footnote{More formally, these
definitions of generalized Wigner and generalized invariant matrices are
describing sequences of matrices $\bW \in \R^{n \times n}$ of increasing
dimensions $n \to \infty$, rather than a single matrix.
We will choose not make this terminological distinction in our work.}
with limit diagonal distribution $\cD_\tdiag$ if, as $n \to \infty$,
\begin{enumerate}[(a)]
\item $\Mb \in \R^{n \times n}$ converges in
diagonal distribution a.s.\ to a limit $\cD_\tdiag$.
\item For any $\eps>0$ and any fixed $p(\x)\in\Delta\langle\x\rangle$, almost
surely for all large $n$,
\begin{align*}
	\max_{i\neq j} |p(\Mb)[i,j]| < n^{-1/2+\eps}.
\end{align*}
\item $\bPi\in\RR^{n\times n}$ is a uniformly random signed permutation,
independent of $\Mb$.
\end{enumerate}
\end{definition}

Our result on AMP universality will pertain specifically
to such matrices $\W$ whose limit
diagonal distribution $\cD_\tdiag$ coincides with that of an orthogonally
invariant matrix. In this setting, the next proposition clarifies
that $\cD_\tdiag$ is determined uniquely by the limit spectral law of $\bW$,
and it also provides simpler conditions inspired by the Spectral Universality
Class in \cite{dudeja2022spectral} that imply Definition \ref{def:syminvariant}.
We have stated Definition \ref{def:syminvariant} for more general
limits $\cD_\tdiag$ because, as
discussed in Remark \ref{remark:sym_tn_universality} to follow, we will in fact prove
a general lemma showing the existence and universality of the limit empirical
distribution of iterates for first-order iterative algorithms
applied to any such matrix $\W$, even if $\cD_\tdiag$ does not correspond to an
orthogonally-invariant model.

\begin{proposition}\label{prop:syminvariant}
Let $\W \in \R^{n \times n}$ be a symmetric matrix with eigenvalues
$\bd \in \R^n$ satisfying $\bd \toW D$ almost surely as $n \to \infty$,
where $D$ has finite moments of all orders.
\begin{enumerate}[(a)]
\item If $\W$ is orthogonally invariant, then $\W$ is a symmetric generalized
invariant matrix in the sense of Definition \ref{def:syminvariant}, and
its limit diagonal distribution $\cD_\tdiag$ is determined uniquely by the law
of $D$.
\item Suppose that either
\begin{enumerate}[1.]
\item $\W=\bO\bD\bO^\top$ where $\bD=\diag(\bd)$ and
$\bO=\bPi_V\bH\bPi_E$, such that
$\bPi_V,\bPi_E \in \R^{n \times n}$ are uniformly random signed permutations
independent of each other and of $(\bD,\bH)$, and $\bH$ is an orthogonal matrix
with entries satisfying
\begin{equation}\label{eq:symdelocalization}
\max_{i,j \in [n]} |H[i,j]|<n^{-1/2+\eps}
\end{equation}
for any fixed $\eps>0$, almost surely for all large $n$.
\item $\W=\bPi\bM\bPi^\top$ such that $\bPi$ is a
uniformly random signed permutation independent of $\bM$ (which has
eigenvalues $\bd$),
and for each fixed integer $\nu \geq 1$, the matrix $\bM^\nu$ satisfies
\begin{equation}\label{eq:Mcondition}
\max_{i=1}^n \big|M^\nu[i,i] - \tfrac{1}{n}\Tr \bM^\nu \big| <
n^{-1/2+\eps},
\quad \max_{i\neq j} \big|M^\nu[i,j]\big| < n^{-1/2+\eps}
\end{equation}
for any fixed $\eps>0$, almost surely for all large $n$.
\end{enumerate}
Then $\W$ is a generalized invariant matrix in the sense of Definition
\ref{def:syminvariant}, and its limit diagonal distribution $\cD_\tdiag$ coincides
with that of the orthogonally invariant matrix in part (a).
\end{enumerate}
\end{proposition}

We prove Proposition \ref{prop:syminvariant} in Appendix
\ref{appendix:syminvariant}.
Important examples for applications are when $\W$ is
a composition of permutations, deterministic Hadamard/Fourier
matrices, and diagonal operators.
We discuss one such application to compressed sensing in Example \ref{ex:CS} of 
Section~\ref{subsec:examples}.

The following is our main theorem on AMP universality in this context,
showing that the state evolution of AMP algorithms for
orthogonally invariant matrices holds universally over the class of
generalized invariant matrices with matching limit diagonal distribution.

\begin{theorem}\label{thm:syminvariant}
Let $\W \in \R^{n \times n}$ be a symmetric generalized invariant
matrix whose limit diagonal distribution $\cD_\tdiag$ coincides with
that of an orthogonally invariant matrix $\G$.
Let $\u_1,\f_1,\ldots,\f_k$ be independent of $\W$ and satisfy 
Assumption~\ref{assump:ufconvergence}. 
Suppose that
\begin{enumerate}
\item Each function $u_{t+1}:\R^{t+k} \to \R$ is continuous, satisfies the 
polynomial growth condition \eqref{eq:polygrowth} for some order $p \geq 1$, and is
Lipschitz in its first $t$ arguments.
\item $\|\W\|_{\op}<C$ for a constant
$C>0$ almost surely for all large $n$.
\end{enumerate}
Let $\{b_{ts}\}$ and $\{\bSigma_t\}$ be defined by the orthogonally invariant
prescriptions (\ref{eq:symorthobSigma}) and (\ref{eq:symorthob}) for the
limit spectral distribution $D$ specified by $\cD_\tdiag$. Suppose that
$\Var[D]>0$ and each matrix $\bSigma_t$ is non-singular.
Then for any fixed $t \geq 1$, almost
surely as $n \to \infty$, the iterates of (\ref{eq:AMP}) satisfy
\[(\u_1,\f_1,\ldots,\f_k,\z_1,\ldots,\z_t) \toWtwo
(U_1,F_1,\ldots,F_k,Z_1,\ldots,Z_t)\]
where $(Z_1,\ldots,Z_t) \sim \Normal(0,\bSigma_t)$ is independent of
$(U_1,F_1,\ldots,F_k)$, i.e.\ this limit has the same joint law as
described by the AMP state evolution for $\G$.
\end{theorem}

\begin{remark}\label{remark:symempirical}
Theorems \ref{thm:Wigner} and \ref{thm:syminvariant} hold equally for AMP
algorithms where, in the prescriptions (\ref{eq:GOEb}) and (\ref{eq:symorthob})
for $b_{ts}$, the quantities
$\E[\partial_r u_{s+1}(Z_{1:s},F_{1:k})]$ and $\E[U_rU_s]$
are replaced by the empirical averages
\[\langle \partial_r u_{s+1}(\z_{1:s},\f_{1:k}) \rangle
=\frac{1}{n} \sum_{i=1}^n \partial_r u_{s+1}(z_{1:s}[i],f_{1:k}[i]),
\qquad \langle \u_r \odot \u_s \rangle=\frac{1}{n}\sum_{i=1}^n u_r[i]u_s[i].\]
For example, such an AMP algorithm for GOE matrices $\W$ and non-linearities
$u_{t+1}(z_{1:t},f_{1:k})=u_{t+1}(z_t)$ consists of the iterations
\[\z_t=\W\u_t-\langle u_t'(\z_{t-1}) \rangle \u_{t-1}, \qquad
\u_{t+1}=u_{t+1}(\z_t).\]

To see this, note that $b_{11}$ depends only on $\E[U_1^2]$, so these 
prescriptions for $b_{11}$ asymptotically coincide by Assumption
\ref{assump:ufconvergence}. Then the state evolution
holds for $\z_1$. 
Inductively, validity of the state evolution for 
$\z_{1:t}$ ensures that, almost surely as $n \to \infty$,
\begin{align*}
\langle \partial_r u_{s+1}(\z_{1:s},\f_{1:k}) \rangle &\to
\E[\partial_r u_{s+1}(Z_{1:s},F_{1:k})] \text{ for all } r \leq s \leq t,\\
\langle \u_r \odot \u_s \rangle &\to \E[U_rU_s] \text{ for all } r,s \leq t+1
\end{align*}
where the first statement follows from Wasserstein-2 convergence of
$(\z_{1:s},\f_{1:k})$ and Stein's lemma
(c.f.\ \cite[Proposition E.5]{fan2022approximate}).
Then the presciptions of (\ref{eq:GOEb}) and (\ref{eq:symorthob}) for
$\{b_{t+1,s}\}_{s \leq t+1}$
asymptotically coincide with their empirical versions defined by
$\langle \partial_r u_{s+1}(\z_{1:s},\f_{1:k}) \rangle$
and $\langle \u_r \odot \u_s \rangle$, which in turn implies validity of the
state evolution for $\z_{1:(t+1)}$.
\end{remark}

\begin{remark}
Theorems \ref{thm:Wigner} and \ref{thm:syminvariant} show universality
of AMP algorithms with an initialization $\bu_1$ that is independent of $\bW$.
For spiked matrix models with a low-rank signal component, alternative
AMP algorithms with spectral initializations have been studied for example in
\cite{montanari2021estimation,mondelli2021pca,zhong2021approximate}. Universality for such
algorithms may be shown using the preceding results, by approximating the
spectral initialization with a large number of linear AMP iterations starting
from an initialization $\bu_1$ that is independent of $\bW$ but
correlated with the true signal; we refer
to \cite[Section 8]{chen2021universality} and \cite[Section
A.2]{zhong2021approximate} for examples of this type of argument.

Since we allow the non-linearities $u_{t+1}(\cdot)$ to be functions of
all preceding iterates $\z_1,\ldots,\z_t$,
universality of AMP with matrix-valued iterates in $\R^{n \times J}$ for a
fixed dimension $J \geq 1$ may also be
deduced from the preceding results, by simulating each iteration of
any such algorithm using $J$ iterations of an algorithm with iterates
in $\R^n$. We leave the further study of these extensions to future work,
as the need arises in applications.
\end{remark}

\subsection{Tensor networks and strategy of proof}

We describe here our high-level strategy of proof for
Theorems \ref{thm:Wigner} and \ref{thm:syminvariant}.
The full proofs of these results are contained in Section \ref{sec:proofs}.

\begin{definition}\label{def:tensornetwork}
A \emph{diagonal tensor network} $T=(\cV,\cE,\{q_v\}_{v \in \cV})$ in $k$
variables is an undirected tree graph with vertices $\cV$ and edges
$\cE \subset \cV \times \cV$, each of whose vertices $v \in \cV$ is labeled by
a polynomial function $q_v:\R^k \to \R$.
The \emph{value} of $T$ on a symmetric matrix
$\W \in \R^{n \times n}$ and vectors $\x_1,\ldots,\x_k \in \R^n$ is
\[\val_T(\W;\x_1,\ldots,\x_k)=\frac{1}{n} \sum_{\bi \in [n]^\cV}
q_{\bi|T} \cdot W_{\bi|T}\]
where, for each index tuple $\bi=(i_v:v \in \cV) \in [n]^\cV$, we set
\[q_{\bi|T}=\prod_{v \in \cV} q_v(x_1[i_v],\ldots,x_k[i_v]),
\qquad W_{\bi|T}=\prod_{(u,v) \in \cE} W[i_u,i_v].\]
\end{definition}

This value may be understood as:
\begin{enumerate}
\item Associating to each vertex $v \in \cV$ a diagonal tensor
$\bT_v=\diag(q_v(\x_1,\ldots,\x_k)) \in \R^{n \times \cdots \times n}$,
where the order of this tensor equals the degree of $v$ in the tree.\footnote{
We remind readers our notation that
$q_v(\x_1,\ldots,\x_k) \in \R^n$ indicates the
application of $q_v:\R^k \to \R$ row-wise to $(\x_1,\ldots,\x_k) \in \R^{n
\times k}$, and $\bT_v$ is then a diagonal tensor with
$q_v(\x_1,\ldots,\x_k) \in \R^n$ along its main diagonal.}
\item Associating to each edge the symmetric matrix $\W$.
\item Iteratively contracting all tensor-matrix-tensor products represented by
the edges of the tree.
\end{enumerate}
For example, if $\cV=[w+1]$ and $T$ is the line graph
$1-2-\cdots-w-(w+1)$, then
$\bT_1,\bT_{w+1} \in \R^n$ are vectors, $\bT_v \in \R^{n \times n}$
is a diagonal matrix for each vertex $v \in \{2,\ldots,w\}$,
and the value (in usual matrix-vector product notation) is
\[\val_T(\W;\x_1,\ldots,\x_k)=\frac{1}{n}\,\bT_1^\top
\W\bT_2\W \cdots \W\bT_w\W\bT_{w+1}.\]
When each tensor $\bT_v$ has all 1's along the main diagonal, this
definition is an example of the graph sum used to show asymptotic freeness of
Wigner and diagonal matrices in \cite{mingo2012sharp,mingo2017free},
and it is also a specific case of a ``graph monomial'' in the notion of
traffic freeness in \cite{male2020traffic}.

We will show in Lemma
\ref{lemma:polynomial_tensor_network_decomposition}
that for any AMP algorithm (or more generally, any
first-order iterative algorithm of the form (\ref{eq:AMP}))
with polynomial non-linearities $u_2,u_3,u_4,\ldots$,
and for any polynomial test function $p(\cdot)$, the coordinate average
\[\langle p(\u_{1:t},\z_{1:t},\f_{1:k}) \rangle=
\frac{1}{n}\sum_{i=1}^n p(u_{1:t}[i],z_{1:t}[i],f_{1:k}[i])\]
of $p(\cdot)$ evaluated on the AMP iterates and side information vectors
is a linear combination of values of different tensor networks on $\W$ and
$\u_1,\f_1,\ldots,\f_k$. Then,
leveraging state evolution results of \cite{fan2022approximate}
to perform an inductive polynomial approximation argument, the proof reduces
the universality of AMP for Lipschitz non-linearities
to the universality of these tensor
network values. This reduction is encapsulated in the following lemma.

\begin{lemma}\label{lemma:Wignercompare}
Let $\u_1,\f_1,\ldots,\f_k \in \R^n$ satisfy Assumption
\ref{assump:ufconvergence}. Let $\W,\G \in \R^{n \times n}$ be symmetric
random matrices independent of $\u_1,\f_1,\ldots,\f_k$ such that
\begin{enumerate}
\item $\G=\Ob\Db\Ob^\top$ is an orthogonally invariant matrix, where
$\Db=\diag(\db)$ and $\db \toW D$ for a limit law $D$ with compact support and
$\Var[D]>0$.
\item $\|\W\|_\op<C$ for a constant $C>0$, almost surely for all large $n$.
\item For every diagonal tensor network $T$ in $k+1$ variables,
almost surely as $n \to \infty$,
\[\val_T(\W;\u_1,\f_1,\ldots,\f_k)-\val_T(\G;\u_1,\f_1,\ldots,\f_k) \to 0.\]
\end{enumerate}
Let $u_{t+1}:\R^{t+k} \to \R$ be continuous functions which satisfy the 
polynomial growth condition \eqref{eq:polygrowth} for some order $p \geq 1$,
and are Lipschitz in their first $t$ arguments. Let
$\{b_{ts}\}$ and $\{\bSigma_t\}$ be defined by the orthogonally invariant
prescriptions (\ref{eq:symorthobSigma}) and (\ref{eq:symorthob}) for the limit
law $D$, where each $\bSigma_t$ is non-singular. Then the iterates
(\ref{eq:AMP}) applied to $\W$ satisfy, almost surely as $n \to \infty$ for any
fixed $t \geq 1$,
\[(\u_1,\f_1,\ldots,\f_k,\z_1,\ldots,\z_t) \toWtwo
(U_1,F_1,\ldots,F_k,Z_1,\ldots,Z_t)\]
where this limit has the same joint law as described by the AMP state evolution
for $\G$.
\end{lemma}

This lemma applies also in the special case of $\G \sim \GOE(n)$,
where the definitions of $\{b_{ts}\}$ and $\{\bSigma_t\}$ reduce to the GOE 
prescriptions of (\ref{eq:GOEbSigma}) and (\ref{eq:GOEb}).
The lemma does not assume any particular matrix model for $\W$, and thus
may be used as a tool to establish AMP universality for matrix
models beyond the ones we consider in this work.

Theorems \ref{thm:Wigner} and \ref{thm:syminvariant} then follow 
from the next two lemmas, which verify the universality of tensor network
values for the classes of generalized Wigner matrices and symmetric
generalized invariant matrices.

\begin{lemma}\label{lemma:Wignermoments}
Let $\x_1,\ldots,\x_k \in \R^n$ be (random or deterministic) vectors and let
$(X_1,\ldots,X_k)$ have finite moments of all orders,
such that almost surely as $n \to \infty$,
\begin{equation}\label{eq:WignerMomentsWpConvergence}
(\x_1,\ldots,\x_k) \toW (X_1,\ldots,X_k).
\end{equation}
Let $\W \in \R^{n \times n}$ be a generalized Wigner matrix,
independent of $\x_1,\ldots,\x_k$, with variance profile matrix $\S$.
Let $\s_i$ be the $i^\text{th}$ row of $\S$,
and suppose for each fixed polynomial function $q:\R^k \to \R$ that
\begin{equation}\label{eq:WignerMomentsVP}
\max_{i=1}^n \Big|\langle q(\x_1,\ldots,\x_k) \odot \s_i \rangle
-\langle q(\x_1,\ldots,\x_k) \rangle \cdot \langle \s_i \rangle \Big| \to 0.
\end{equation}
Then for any diagonal tensor network $T$ in $k$ variables, there is a
deterministic value $\limval_T(X_1,\ldots,X_k)$ depending only
on $T$ and the joint law of $(X_1,\ldots,X_k)$ such that almost surely,
\[\lim_{n \to \infty} \val_T(\W;\x_1,\ldots,\x_k)
=\limval_T(X_1,\ldots,X_k).\]
In particular, this limit value is the same for $\W$ as for $\G \sim \GOE(n)$.
\end{lemma}

\begin{lemma}\label{lemma:syminvmoments}
Let $\x_1,\ldots,\x_k \in \R^n$ be (random or deterministic) vectors and let
$(X_1,\ldots,X_k)$ have finite moments of all orders,
such that almost surely as $n \to \infty$,
\begin{equation}
(\x_1,\ldots,\x_k) \toW (X_1,\ldots,X_k).
\end{equation}
Let $\W \in \R^{n \times n}$ be a symmetric generalized invariant
matrix, independent of $\x_1,\ldots,\x_k$,
with limit diagonal distribution $\cD_\tdiag$.
Then for any diagonal tensor
network $T$ in $k$ variables, there is a deterministic limit value
$\limval_T(X_1,\ldots,X_k,\cD_\tdiag)$ depending only on $T$, the joint law of
$(X_1,\ldots,X_k)$, and $\cD_\tdiag$ such that almost surely,
\[\lim_{n \to \infty} \val_T(\W;\x_1,\ldots,\x_k)
=\limval_T(X_1,\ldots,X_k,\cD_\tdiag).\]
In particular, if there exists an orthogonally invariant matrix $\G$ having 
the same limit diagonal distribution $\cD_\tdiag$, then this limit value is 
the same for $\W$ as for $\G$.
\end{lemma}

\begin{remark}\label{remark:sym_tn_universality}
Lemma~\ref{lemma:syminvmoments} applies to any class of 
symmetric generalized invariant matrices satisfying Definition
\ref{def:syminvariant}, where the limit diagonal distribution
$\cD_\tdiag$ does not necessarily coincide with that of an orthogonally
invariant model.

This has the following implication: Consider any first-order iterative algorithm
having the structure (\ref{eq:AMP}), where $b_{ts}$ are arbitrary
fixed constants and $u_{t+1}:\R^{t+k} \to \R$ are polynomial functions applied
entrywise. Then for any polynomial test function $p(\cdot)$, the value
\[\frac{1}{n}\sum_{i=1}^n p(u_{1:t}[i],z_{1:t}[i],f_{1:k}[i])\]
is a linear combination of tensor network values (c.f.\ Lemma
\ref{lemma:polynomial_tensor_network_decomposition}) and hence has a universal
limit as $n \to \infty$. Under mild moment assumptions, this implies that
there exists a
limit law for the empirical distribution of each iterate $\u_t$ and $\z_t$,
and this
law is universal across such matrices having the same limit diagonal
distribution $\cD_\tdiag$.

When $\cD_\tdiag$ is not described by an orthogonally invariant model, we
believe it may be an interesting open question to develop such an
algorithm that has a more succinct state-evolution
characterization of its iterates in terms of this limit diagonal law.
\end{remark}

The proofs of Lemmas \ref{lemma:Wignermoments} and \ref{lemma:syminvmoments}
result in forms for the limit tensor network values that are, in general,
combinatorially complex. However, a by-product of the proofs is that these forms
reduce to 0 when all diagonal
tensors of the tensor network have vanishing normalized trace.
This may be viewed as a version of asymptotic freeness for tensor networks, and 
we state the result here for independent interest.

\begin{proposition}\label{prop:freeness}
\begin{enumerate}[(a)]
\item In the setting of Lemma \ref{lemma:Wignermoments}, let $T$ be a diagonal
tensor network such that, for every vertex of $v$ of $T$, almost surely
\begin{equation}\label{eq:qnormalized}
\lim_{n \to \infty} \langle q_v(\x_1,\ldots,\x_k) \rangle
=\lim_{n \to \infty} \frac{1}{n}\sum_{i=1}^n q_v(x_1[i],\ldots,x_k[i])=0.
\end{equation}
Then $\limval_T(X_1,\ldots,X_k)=0$.
\item In the setting of Lemma \ref{lemma:syminvmoments}, suppose $T$ is a
diagonal tensor network for which (\ref{eq:qnormalized}) holds almost surely
for every vertex $v$.
Suppose also that there exists an orthogonally invariant matrix having the same 
limit diagonal distribution $\cD_\tdiag$ as $\bW$, and 
$\lim_{n \to \infty} \frac{1}{n} \Tr \W=0$.
Then $\limval_T(X_1,\ldots,X_k,\cD_\tdiag)=0$.
\end{enumerate}
\end{proposition}

\subsection{Universality of AMP algorithms for rectangular matrices}\label{sec:result_rec}

Let $\W \in \R^{m \times n}$ be a rectangular random matrix. Consider an
initialization $\u_1 \in \R^m$ and vectors of
side information $\f_1,\ldots,\f_k \in \R^m$ and $\g_1,\ldots,\g_\ell \in \R^n$,
all independent of $\W$. Let $v_1,v_2,v_3,\ldots$ and $u_2,u_3,u_4,\ldots$ be two sequences
of non-linear functions where $v_t:\R^{t+\ell} \to \R$ and
$u_{t+1}:\R^{t+k} \to \R$. We study an AMP algorithm that
computes, for $t=1,2,3,\ldots$
\begin{subequations}\label{eq:AMPrect}
\begin{align}
\z_t&=\W^\top \u_t-\sum_{s=1}^{t-1} b_{ts}\v_s \label{eq:AMPrectz}\\
\v_t&=v_t(\z_1,\ldots,\z_t,\g_1,\ldots,\g_\ell)\\
\y_t&=\W\v_t-\sum_{s=1}^t a_{ts} \u_s\\
\u_{t+1}&=u_{t+1}(\y_1,\ldots,\y_t,\f_1,\ldots,\f_k)\label{eq:AMPrectu}
\end{align}
\end{subequations}
where $\{b_{ts}\}_{s<t}$ and $\{a_{ts}\}_{s \leq t}$ are deterministic ``Onsager
correction'' coefficients. We will characterize the iterates of this algorithm
in the limit as $m,n \to \infty$ proportionally with
$m/n \to \gamma \in (0,\infty)$, for fixed $k,\ell \geq 0$. 
For Gaussian
and bi-orthogonally invariant matrices $\bW$ (see the definition after 
Definition~\ref{def:rectinvariant}), we review
the forms for these correction coefficients and the corresponding state
evolutions in 
(\ref{eq:whitenoisebSigma}-\ref{eq:whitenoiseb})
and (\ref{eq:rectorthobSigma}--\ref{eq:rectorthob}) of
Appendix~\ref{appendix:rect}.

We assume the following condition for $(\u_1,\f_1,\ldots,\f_k)$ and
$(\g_1,\ldots,\g_\ell)$, which is
analogous to Assumption~\ref{assump:ufconvergence}.

\begin{assumption}\label{assump:ufgconvergence}
Almost surely as $m,n \to \infty$,
\[(\u_1,\f_1,\ldots,\f_k) \toW (U_1,F_1,\ldots,F_k) \quad
\text{ and } \quad (\g_1,\ldots,\g_\ell) \toW (G_1,\ldots,G_\ell)\]
for joint limit laws $(U_1,F_1,\ldots,F_k)$ and $(G_1,\ldots,G_\ell)$ having
finite moments of all orders, where $\E[U_1^2]>0$. Multivariate polynomials
are dense in the real $L^2$-spaces of functions $f:\R^{k+1} \to \R$
and $g:\R^{\ell} \to \R$ with the inner-products
\begin{align*}
(f,\tilde f) &\mapsto \E[f(U_1,F_1,\ldots,F_k)\tilde f(U_1,F_1,\ldots,F_k)]\\
(g,\tilde g) &\mapsto 
\E[g(G_1,\ldots,G_\ell)\tilde g(G_1,\ldots,G_\ell)].
\end{align*}
\end{assumption}

Our main results show that the state evolution characterizations of AMP
algorithms for Gaussian and orthogonally invariant matrices are
universal across the following matrix ensembles, analogous to
Definitions \ref{def:Wigner} and \ref{def:syminvariant} in the symmetric
setting.

\begin{definition}\label{def:rect}
$\W \in \R^{m \times n}$ is a \emph{generalized white noise matrix} with
(deterministic) variance profile $\S \in \R^{m \times n}$ if
\begin{enumerate}[(a)]
\item All entries $W[\a,i]$ are independent.
\item Each entry $W[\a,i]$ has mean 0, variance $n^{-1}S[\a,i]$, and higher moments
satisfying, for each integer $p \geq 3$,
\[\lim_{m,n \to \infty} n \cdot \max_{\a=1}^m \max_{i=1}^n \E[|W[\a,i]|^p]=0.\]
\item For a constant $C>0$ independent of $m,n$,
\[\max_{\a=1}^m \max_{i=1}^n S[\a,i] \leq C, \quad
\lim_{m,n \to \infty} \max_{\a=1}^m \bigg|\frac{1}{n}\sum_{i=1}^n S[\a,i]-1\bigg|
=0,\]
\[
\lim_{m,n \to \infty} \max_{i=1}^n \bigg|\frac{1}{m}\sum_{\a=1}^m S[\a,i]-1\bigg|
=0.\]
\end{enumerate}
We call $\W$ a \emph{Gaussian white noise matrix} in the special case
where $W[\a,i] \sim \Normal(0,1/n)$ and $S[\a,i]=1$ for all 
$(\alpha,i) \in [m] \times [n]$.
\end{definition}

Next, we introduce a notion of diagonal distribution for rectangular
matrices, analogous to Definition~\ref{def:diagdistr}. 
Recall the diagonal map $\Delta(\cdot)$,
and let $\Delta\langle\bx, \bI_m, \bI_n\rangle$ be the set of all words in 
$\bx, \bI_m, \bI_n$ and $\Delta(\cdot)$, for example
\[\bx\bI_m, \quad \bx\Delta(\bI_n\bx)\bI_m,
\quad \Delta(\bx\bx\Delta(\bI_n))\bx,
\quad \bI_m\bx\bI_n\Delta(\Delta(\bx))\Delta(\bx\bI_m).\]
For $p(\bx)\in\Delta\langle\bx,\bI_m,\bI_n\rangle$ and $\bM\in\RR^{m\times n}$,
we write $p(\widetilde\bM) \in \RR^{(m+n)\times(m+n)}$ for its evaluation at
$\bx=\widetilde\bM$,
	\[\bI_m=\begin{pmatrix} \Id_m & 0 \\ 0 & 0 \end{pmatrix}
\in \RR^{(m+n)\times(m+n)}, \quad
\bI_n=\begin{pmatrix} 0 & 0 \\ 0 & \Id_n \end{pmatrix} \in
\RR^{(m+n)\times(m+n)},\]
where we define the symmetric embedding
\begin{equation}\label{eq:Membedding}
	\widetilde\bM = \begin{pmatrix}
		0 & \bM\\
		\bM^\top & 0
	\end{pmatrix} \in \RR^{(m+n)\times(m+n)}
\end{equation}
and the identity matrices $\Id_m \in \R^{m \times m}$ and $\Id_n \in \R^{n
\times n}$.

\begin{definition}\label{def:rec_diag}
The distribution over the diagonal of a rectangular matrix $\bM\in\RR^{m\times n}$ 
is the mapping
\begin{align*}
p(\bx) \in \Delta \langle \bx, \bI_m, \bI_n \rangle \mapsto 
\tfrac{1}{m+n} \Tr p(\widetilde\bM).
\end{align*}
Matrices $\bM \in \R^{m \times n}$ converge in diagonal distribution a.s.\ if
$\lim_{m,n \to \infty} \frac{1}{m+n} \Tr p(\widetilde\bM)$ exists almost
surely (and is finite) for every fixed $p(\bx) \in \Delta \langle \x, \bI_m,
\bI_n \rangle$, as $m,n \to \infty$ with $m/n\to\gamma \in (0,\infty)$.
The limit diagonal distribution of $\bM$, which we will refer to as
$\cD_\tdiag$, is then the mapping
\begin{align*}
p(\bx) \in \Delta\langle\bx,\bI_m,\bI_n\rangle \mapsto
\lim_{m,n\to\infty} \tfrac{1}{m+n} \Tr p(\widetilde\bM).
\end{align*}
\end{definition}

Note that $\cD_\tdiag$ and $\gamma$ specify the
limit of $\frac{1}{m}\Tr (\bM\bM^\top)^\nu$ for each fixed integer
$\nu \geq 1$, and hence also the limit
singular value distribution of $\bM$ when this distribution has compact support.
Note also that, similarly to the symmetric setting, $\bM$ and
$\bPi_U \bM\bPi_V^\top$ must have the same limit diagonal distribution
$\cD_{\tdiag}$ for any signed permutation matrices
$\bPi_U \in \R^{m \times m}$ and $\bPi_V \in \R^{n \times n}$.

\begin{definition}\label{def:rectinvariant}
$\W=\bPi_U\bM\bPi_V^\top \in \R^{m \times n}$ is a
\emph{rectangular generalized invariant matrix}
with limit diagonal distribution $\cD_\tdiag$ if,
as $m,n\to\infty$ with $m/n\to\gamma\in(0,\infty)$,
\begin{enumerate}[(a)]
\item $\bM$ converges in diagonal distribution a.s.\ to a limit
$\cD_\tdiag$.
\item For any $\eps>0$ and any fixed $p(\bx) \in \Delta\langle \bx,\bI_m,\bI_n \rangle$,
almost surely for all large $m,n$,
\[\max_{i \neq j} |p(\widetilde \bM)[i,j]|<n^{-1/2+\eps}\]
where $\widetilde{\bM}$ is the symmetric embedding (\ref{eq:Membedding}).
\item $\bPi_U\in\RR^{m\times m}$ and $\bPi_V\in\RR^{n\times n}$ are uniformly
random signed permutations independent of each other and of $\bM$.
\end{enumerate}
\end{definition}

We call $\bW \in \R^{m \times n}$ \emph{bi-orthogonally invariant}
if it has singular value decomposition $\bW=\bO\bD\bQ^\top$ where
$\bO \sim \Haar(\OO(m))$ and
$\bQ \sim \Haar(\OO(n))$ are Haar-distributed on the orthogonal groups
independently of each other and of $\bD=\diag(\bd) \in \R^{m \times n}$. We verify in Proposition
\ref{prop:rectinvariant} of Appendix \ref{appendix:rect} that such
bi-orthogonally invariant matrices
satisfy Definition \ref{def:rectinvariant}, where $\cD_\tdiag$ is determined
uniquely by $\gamma=\lim_{m,n \to \infty} m/n$ and the limit singular value
distribution of $\bD$.

The following theorems show that the state evolution of AMP algorithms for
Gaussian white noise matrices holds universally for generalized white noise
matrices as in Definition \ref{def:rect},
and the state evolution for
bi-orthogonally invariant matrices holds universally for rectangular generalized
invariant matrices as in Definition \ref{def:rectinvariant}.

\begin{theorem}\label{thm:rect}
Let $\W \in \R^{m \times n}$ be a generalized white noise matrix with variance
profile matrix $\S$, and let $\u_1,\f_1,\ldots,\f_k,\g_1,\ldots,\g_\ell$ be independent
of $\W$ and satisfy Assumption \ref{assump:ufgconvergence}. Suppose that
\begin{enumerate}
\item Each function $v_t:\R^{t+\ell} \to \R$ and
$u_{t+1}:\R^{t+k} \to \R$ is continuous, satisfies the 
polynomial growth condition \eqref{eq:polygrowth} for some order $p \geq 1$, 
and is Lipschitz in its first $t$ arguments.
\item $\|\W\|_\op<C$ for a constant $C>0$ almost surely for all large $m,n$.
\item Let $\s_\a$ be the $\a^\text{th}$ row of $\S$ and $\s^i$ be
the $i^\text{th}$ column of $\S$. 
For any fixed polynomial
functions $p:\R^{k+1} \to \R$ and $q:\R^\ell \to \R$, almost surely as $m,n
\to \infty$,
\begin{equation}\label{eq:rectprofileassumption}
\begin{aligned}
\max_{\a=1}^m \Big| \langle p(\u_1,\f_1,\ldots,\f_k) \odot \s_\alpha \rangle
-\langle p(\u_1,\f_1,\ldots,\f_k) \rangle \cdot \langle \s_\alpha \rangle\Big|
&\to 0,\\
\max_{i=1}^n \Big| \langle q(\g_1,\ldots,\g_\ell) \odot \s^i \rangle
-\langle q(\g_1,\ldots,\g_\ell) \rangle \cdot \langle \s^i \rangle\Big| &\to 0.
\end{aligned}
\end{equation}
\end{enumerate}
Let $\{a_{ts}\},\{b_{ts}\},\{\bOmega_t\},\{\bSigma_t\}$ be defined by the
white noise prescriptions
(\ref{eq:whitenoisebSigma}) and (\ref{eq:whitenoiseb}), where
each matrix $\bOmega_t$ and $\bSigma_t$ is non-singular.
Then for any fixed $t \geq 1$, almost surely as $m,n \to \infty$ with $m/n \to
\gamma \in (0,\infty)$, the iterates of (\ref{eq:AMPrect}) satisfy
\begin{align*}
(\u_1,\f_1,\ldots,\f_k,\y_1,\ldots,\y_t) &\toWtwo (U_1,F_1,\ldots,F_k,
Y_1,\ldots,Y_t),\\
(\g_1,\ldots,\g_\ell,\z_1,\ldots,\z_t) &\toWtwo
(G_1,\ldots,G_\ell,Z_1,\ldots,Z_t)
\end{align*}
where $(Z_1,\ldots,Z_t) \sim \Normal(0,\bOmega_t)$ and $(Y_1,\ldots,Y_t) \sim
\Normal(0,\bSigma_t)$ are independent of $(U_1,F_1,\ldots,F_k)$ and
$(G_1,\ldots,G_\ell)$, i.e.\ these limits have the same joint laws as
described by the AMP state evolution for a Gaussian white noise matrix $\W$.
\end{theorem}

\begin{theorem}\label{thm:rectinvariant}
Let $\W \in \R^{m \times n}$ be a rectangular generalized invariant matrix
whose limit diagonal distribution $\cD_\tdiag$ coincides with that of a
bi-orthogonally invariant matrix $\Gb$.
Let $\u_1,\f_1,\ldots,\f_k,\g_1,\ldots,\g_\ell$ be independent of $\W$ 
and satisfy Assumption \ref{assump:ufgconvergence}. Suppose that
\begin{enumerate}
\item Each function $v_t:\R^{t+\ell} \to \R$ and
$u_{t+1}:\R^{t+k} \to \R$ is continuous, satisfies the 
polynomial growth condition \eqref{eq:polygrowth} for some order $p \geq 1$, 
and is Lipschitz in its first $t$ arguments.
\item $\|\W\|_\op<C$ for a constant $C>0$ almost surely for all large $m,n$.
\end{enumerate}
Let $\{a_{ts}\},\{b_{ts}\},\{\bOmega_t\},\{\bSigma_t\}$ be defined by the
bi-orthogonally invariant prescriptions (\ref{eq:rectorthobSigma}) and
(\ref{eq:rectorthob}) for the limit singular value distribution $D$
specified by $\cD_{\tdiag}$ and $\gamma$.
Suppose that $\E[D^2]>0$ and each $\bOmega_t$ and $\bSigma_t$ is non-singular.
Then for any fixed $t \geq 1$, almost surely as $m,n \to \infty$ with
$m/n \to \gamma \in (0,\infty)$, the iterates of (\ref{eq:AMPrect}) satisfy
\begin{align*}
(\u_1,\f_1,\ldots,\f_k,\y_1,\ldots,\y_t) &\toWtwo (U_1,F_1,\ldots,F_k,
Y_1,\ldots,Y_t),\\
(\g_1,\ldots,\g_\ell,\z_1,\ldots,\z_t) &\toWtwo
(G_1,\ldots,G_\ell,Z_1,\ldots,Z_t)
\end{align*}
where $(Z_1,\ldots,Z_t) \sim \Normal(0,\bOmega_t)$ and $(Y_1,\ldots,Y_t) \sim
\Normal(0,\bSigma_t)$ are independent of $(U_1,F_1,\ldots,F_k)$ and
$(G_1,\ldots,G_\ell)$, i.e.\ these limits have the same joint laws as
described by the AMP state evolution for $\Gb$.
\end{theorem}

\begin{remark}\label{remark:rectempirical}
As in Remark \ref{remark:symempirical},
Theorems \ref{thm:rect} and \ref{thm:rectinvariant} hold equally for AMP
algorithms where, in the prescriptions (\ref{eq:whitenoiseb}) and
(\ref{eq:rectorthob}) for $a_{ts}$ and $b_{ts}$, the quantities
$\E[\partial_r v_s(Z_{1:s},G_{1:\ell})]$,
$\E[\partial_r u_{s+1}(Y_{1:s},F_{1:k})]$, $\E[U_rU_s]$, and $\E[V_rV_s]$ are
replaced by the empirical averages
\[\langle \partial_r v_s(\z_{1:s},\g_{1:\ell}) \rangle, \qquad
\langle \partial_r u_{s+1}(\y_{1:s},\f_{1:k}) \rangle, \qquad
\langle \u_r \odot \u_s \rangle, \qquad \langle \v_r \odot \v_s \rangle.\]
For example, such an AMP algorithm for Gaussian white noise matrices $\W$ and
non-linearities $v_t(z_{1:t},g_{1:\ell})=v(z_t)$ and
$u_{t+1}(y_{1:t},f_{1:k})=u(y_t)$ consists of the iterations
\[\z_t=\W^\top \u_t-\gamma \langle u'(\y_{t-1}) \rangle \v_{t-1}, 
\qquad \v_t=v(\z_t),\]
\[\y_t=\W\v_t-\langle v'(\z_t) \rangle \u_t, 
\qquad \u_{t+1}=u(\y_t).\]
\end{remark}

The proofs of Theorems \ref{thm:rect} and \ref{thm:rectinvariant} are similar
to those of Theorems \ref{thm:Wigner} and \ref{thm:syminvariant} for 
symmetric matrices, and we defer them to Appendix \ref{appendix:rect}.

\subsection{Applications}\label{subsec:examples}

\begin{example}\label{ex:SBM}
AMP algorithms for the Gaussian universality class
may be heuristically derived by approximating belief
propagation on dense graphical models
\cite{kabashima2003cdma,donoho2010messagea}. Our assumptions in Theorems
\ref{thm:Wigner} and \ref{thm:rect} are
sufficiently weak to show that their state evolutions remain
valid in sparse random graphs down to sparsity levels of $(\log n)/n$.

As a concrete example, consider the symmetric stochastic block model where $G$
is an undirected graph over $n$ vertices, divided into two communities $\cV_+$ and $\cV_-$ of equal sizes $n/2$. For two $n$-dependent probabilities $p_n>q_n$,
each pair of vertices $(i,j)$ in $G$ (including self-loops, for simplicity of
discussion) is independently connected with probability
\[\PP[i \text{ is connected to } j]=\begin{cases} p_n & \text{ if } i,j \in
\cV_+ \text{ or } i,j \in \cV_-, \\
q_n & \text{ if } i \in \cV_+ \text{ and } j \in \cV_- \text{ or if } i \in
\cV_- \text{ and } j \in \cV_+. \end{cases}\]
Let $\A \in \{0,1\}^{n \times n}$ be the adjacency matrix of $G$, and let
$\bar{p}_n=(p_n+q_n)/2$ be the mean connectivity. Then the centered and
normalized adjacency matrix takes the form
\begin{equation}\label{eq:SBM}
\frac{\A-\bar{p}_n}{\sqrt{n\bar{p}_n(1-\bar{p}_n)}}
=\frac{\sqrt{\lambda_n}}{n}\,\f\f^\top+\W
\end{equation}
where $\lambda_n=n(p_n-q_n)^2/[4\bar{p}_n(1-\bar{p}_n)]$
is a parameter representing the signal-to-noise ratio of the model,
$\f \in \{+1,-1\}^n$ is the binary indicator vector representing the membership
of the vertices, and $\W$ is a symmetric noise matrix with independent
entries. It may checked for each $(i,j) \in [n] \times [n]$ that
\[\E[W[i,j]]=0, \quad \E[W[i,j]^2] \in
\left\{\frac{p_n(1-p_n)}{n\bar{p}_n(1-\bar{p}_n)},
\frac{q_n(1-q_n)}{n\bar{p}_n(1-\bar{p}_n)}\right\},
\quad |W[i,j]| \leq \frac{1}{\sqrt{n\bar{p}_n(1-\bar{p}_n)}}.\]

In the asymptotic regime where $n\bar{p}_n(1-\bar{p}_n) \to \infty$ and
$\lambda_n \to \lambda$ a positive constant, we have that
$S[i,j]:=n \cdot \E[W[i,j]^2] \to 1$ uniformly over $(i,j) \in [n] \times [n]$, 
so that $\W$ is a generalized Wigner matrix in the sense of 
Definition~\ref{def:Wigner} (with variance profile $\S$ approximately constant 
in every entry).
Furthermore, under a slightly stronger assumption
\begin{equation}\label{eq:sparsity}
n\bar{p}_n(1-\bar{p}_n) \geq c\log n
\end{equation}
for any constant $c>0$, \cite[Theorem 2.7 and Eq.\ (2.4)]{benaych2020spectral}
implies that $\|\W\|_{\op}<C$ almost surely for all large
$n$. This encompasses the stochastic block model in regimes with sparsity
$\bar{p}_n \gtrsim (\log n)/n$.\footnote{Our universality result for AMP
with polynomial non-linearities does not require the operator norm bound
$\|\W\|_{\op}<C$ and hence holds for any sparsity $\bar{p}_n \gg 1/n$, c.f.\
Remark \ref{remark:polynomialuniversality}. We
believe that the operator norm requirement in condition 2 of Theorem
\ref{thm:Wigner} may be an artifact of our
polynomial approximation proof.

In contrast, we do not expect AMP universality
to hold for random graph models with sparsity $\bar{p}_n \asymp 1/n$, where
the belief propagation recursions on such graphs may not admit asymptotic
Gaussian approximations.}

It was shown in \cite{deshpande2017asymptotic} that the mutual information 
between $G$ and $\f$ has an asymptotic limit depending only on the limit
signal-to-noise ratio $\lambda$, which is non-trivial when $\lambda>1$.
This was proven by interpolating between the model (\ref{eq:SBM}) and a
``$\ZZ_2$-synchronization'' model where $\W \sim \GOE(n)$, and applying
an AMP analysis in the latter model. Our result of
Theorem \ref{thm:Wigner} implies that, under the additional condition
(\ref{eq:sparsity}), this AMP analysis may instead be directly applied to the
model (\ref{eq:SBM}), bypassing interpolation to the GOE.
\end{example}

\begin{example}\label{ex:biwhitenedcount}
Let $\Y \in \R^{m \times n}$ be a signal-plus-noise data matrix modeled as
\[\Y=\X+\Eb\]
where $\X=\E[\Y]=\sum_{j=1}^k \f_j\g_j^\top \in \R^{m \times n}$ is a low-rank
signal matrix, and $\Eb=\Y-\X$ is a mean-zero matrix of residual noise. We
assume that $\Eb$ has independent entries, although
in many applications involving count observations or missing data, these
entries may have a heteroskedastic variance profile $\Vb$ where
\[V[\a,i]:=\Var[E[\a,i]].\]
Such models where the variance $V[\a,i]$ is a quadratic
function $a+bX[\a,i]+cX[\a,i]^2$ of the mean were discussed recently
in \cite{landa2021biwhitening}, including Poisson and
negative-binomial models for $\Y$ in the context of
single-cell RNA sequencing applications
\cite{townes2019feature,hafemeister2019normalization,sarkar2021separating}.
Such models encompass also simple models of missing data, where $\Y$ is a partial
observation of an underlying low-rank signal matrix $\widetilde{\bX}$ so that
\[Y[\a,i]=\begin{cases} \widetilde{X}[\a,i] & \text{ with probability } p,\\
0 & \text{ with probability } 1-p,
\end{cases}\]
independently for each entry. 
Then $X[\a,i]=p \cdot \widetilde{X}[\a,i]$
and $V[\a,i]=p(1-p) \cdot \widetilde{X}[\a,i]^2$ are the corresponding means and
variances.

When the entries $V[\a,i]$ are heteroskedastic, the singular
value spectrum of $\Eb$ does not generally conform to the Marcenko-Pastur law.
However, row and column normalization is typically applied in
practice prior to data analysis,
with \cite{landa2021biwhitening} suggesting the following normalization
scheme: Determine via Sinkhorn iteration two
diagonal matrices $\D_1 \in \R^{m \times m}$ and $\D_2 \in \R^{n \times n}$
for which $\S=\D_1 \Vb \D_2$
has all rows summing to $n$ and all columns summing to $m$, and use these
to standardize $\Y$ into the biwhitened matrix
\[\widetilde{\Y}=\frac{1}{\sqrt{n}} \cdot \D_1^{1/2} \Y \D_2^{1/2}
=\frac{1}{\sqrt{n}} \cdot \D_1^{1/2} \X \D_2^{1/2}+\W,
\qquad \W=\frac{1}{\sqrt{n}} \cdot \D_1^{1/2}\Eb\D_2^{1/2}.\]
\cite{landa2021biwhitening} proved that such biwhitened count matrices 
have singular value spectra asymptotically described by the Marcenko-Pastur law,
and showed a remarkable empirical agreement with the Marcenko-Pastur law
for matrices arising in several domains of application,
from single-cell biology to topic modeling of text.

In this standardized model $\widetilde{\Y}$, the error matrix $\W$ now has
variance profile $\S=\D_1 \Vb\D_2$ which satisfies by construction
$\frac{1}{n}\sum_{i=1}^n S[\a,i]=\frac{1}{m}\sum_{\a=1}^m S[\a,i]=1$,
and Theorem \ref{thm:rect} describes conditions under which
state evolution holds for Gaussian AMP algorithms applied to this
matrix $\W$. We note that to analyze AMP applied instead
to $\widetilde{\Y}$, the condition
(\ref{eq:rectprofileassumption}) represents a potentially strong restriction
on the relation between the variance profile matrix $\bS$ and the low-rank
mean signal. A modified analysis of AMP may be needed in settings where this
restriction does not hold, and we leave this as a direction to explore in
future work.
\end{example}

\begin{example}\label{ex:CS}
Much of the early development of AMP algorithms was motivated by compressed sensing
applications of reconstructing sparse signals from linear
measurements. Consider a model of $m$ measurements
\[\y=\W \x+\bm{\varepsilon} \in \R^m\]
where $\x \in \R^n$ is the underlying signal, $\W \in \R^{m \times n}$ is a
random sensing matrix, and $\bm{\varepsilon}$ is measurement noise. 
For i.i.d.\ Gaussian sensing matrices $\W$, pioneering work of
\cite{donoho2009message,donoho2010messagea,donoho2010messageb}
proposed an AMP algorithm for reconstructing $\x$, where the non-linearities are
soft-thresholding functions tailored to the sparsity of $\x$.
Analysis of the dynamics of this algorithm leads to a derivation of a
sparsity-undersampling phase transition curve that matches a phase transition
for $\ell_1$-based reconstruction
in this model
\cite{donoho2005neighborliness,donoho2009message,bayati2011dynamics,bayati2015universality}.

Extensive numerical experiments performed in
\cite{donoho2009observed,monajemi2013deterministic} suggested that this
phase transition curve is universal across broad classes of non-Gaussian
sensing matrices. Theorem \ref{thm:syminvariant} provides an extension of the
AMP universality shown in \cite{bayati2015universality} for this application,
broadening the universality
class to matrices composed of subsampled Fourier or
Hadamard transforms and diagonal operators. Importantly, matrix-vector
multiplication operations for such matrices may be computed in $O(n \log n)$
time without explicitly storing the matrices in memory, allowing
applications of AMP at much larger scales than would be possible with
i.i.d.\ sensing designs.

As an example, consider
\begin{equation}\label{eq:CSmatrix}
\W=(\bPi_U\bH\bPi_E)\D(\bPi_V \bK\bPi_F)^\top \in \R^{m \times n}
\end{equation}
where $\D \in \R^{m \times n}$ is diagonal with its diagonal entries sampled
i.i.d.\ from a Marcenko-Pastur law;
$\bH,\bK \in \R^{n \times n}$ are orthogonal matrices
representing deterministic Hadamard or discrete Fourier transforms; and
$\bPi_U,\bPi_E,\bPi_V,\bPi_F$
are independent random signed permutations. 
We verify in Proposition \ref{prop:rectinvariant}(b1) that this class of
matrices satisfies Definition \ref{def:rectinvariant}. 
If the signal vector $\x$, residual error $\bm{\varepsilon}$, and
initialization $\x_1$ are each comprised of i.i.d.\ entries, then
the random permutations in $\bPi_U,\bPi_V$ may be further
absorbed into $\x_1,\x,\bm{\varepsilon}$. Thus
the AMP iterates are equal in law to those of AMP applied with a
simpler sensing matrix
\[\widetilde{\bW}=\bXi_U\bH\widetilde{\D}\widetilde{\bK}^\top\bXi_V\]
where $\bXi_U,\bXi_V$ are diagonal matrices of i.i.d.\ $\{+1,-1\}$ signs,
$\widetilde{\bK}^\top \in \R^{m \times n}$ is a random subsampling of $m$ rows 
of $\bK^\top$, and $\widetilde{\D} \in \R^{m \times m}$ is a diagonal matrix
whose diagonal entries are given by those of $\D$ also multiplied by 
i.i.d.\ $\{+1,-1\}$ signs.

Theorem \ref{thm:rectinvariant} implies
that AMP applied to the above matrix $\W$ admits
the same state evolution as when applied to an
i.i.d.\ Gaussian sensing matrix $\G$.
This universality extends beyond the Gaussian setting, to sensing matrices
(\ref{eq:CSmatrix}) where the diagonal entries of $\D$ are sampled from an
arbitrary compactly supported singular value distribution.
Theorem \ref{thm:rectinvariant} then shows that the state evolution
characterizations for the more general AMP algorithms
of \cite{fan2022approximate}---derived originally for bi-orthogonally invariant
ensembles---are valid in such settings. For this compressed sensing application,
we note that the resulting AMP algorithms are similar to the convolutional AMP
algorithms developed and studied recently in
\cite{takeuchi2020convolutional,takeuchi2021bayes}.
\end{example}

\section{Proofs for symmetric matrices}\label{sec:proofs}

\subsection{Universality for generalized Wigner matrices}\label{sec:wigner}

In this section, we prove
Lemma~\ref{lemma:Wignermoments} on the universality of the tensor
network value for generalized Wigner matrices.

Fix a tensor network $T = (\cV, \cE, \{q_v\}_{v \in \cV})$.
Let $\cP$ be the set of all partitions of $\cV$. For each
index tuple $\bi \in [n]^\cV$, define its induced partition $\pi(\bi) \in \cP$
such that vertices $u,v \in \cV$ belong to the same block of $\pi(\bi)$
if and only if $i_u=i_v$. 
Then we can decompose the value of $T$ as
\begin{align}\label{eq:sym_value_decomp}
    \val_T(\bW; \xb_1, \ldots, \xb_k) = \frac{1}{n}
\sum_{\pi \in \cP} \sum_{\bi \in [n]^\cV:\pi(\bi)=\pi} q_{\bi|T} \cdot W_{\bi|T}.
\end{align}

\begin{definition}\label{def:graphpartition}
Let $(\cV,\cE)$ be an undirected graph. For any partition $\pi$ of $\cV$,
the \emph{image of $(\cV,\cE)$ under $\pi$} is the undirected
multi-graph $G_\pi=(\cK_\pi,\cF_\pi)$ that is the image of $(\cV,\cE)$ under the
graph homomorphism sending each vertex $u \in \cV$ to the block of $\pi$
containing $u$.

I.e., the vertices $\cK_\pi \equiv \pi$ of $G_\pi$ are
the blocks of $\pi$, and $G_\pi$ has the same number of edges $|\cF_\pi|$
(counting multiplicity and self-loops) as $|\cE|$.
For each edge $(u,v) \in \cE$, there is a
corresponding edge $(U,V) \in \cF_\pi$ where $U,V \in \pi$ are the blocks
for which $u \in U$ and $v \in V$.
\end{definition}

For each $\pi \in \cP$, let $G_\pi=(\cK_\pi,\cF_\pi)$ be the image of
$(\cV,\cE)$ under $\pi$.
For each block $U \in \cK_\pi$, define the polynomial $Q_U=\prod_{u \in U} q_u$,
and for each unique (undirected) edge $(U,V)$ of $G_\pi$,
let $e(U,V)$ be the number of times it appears in $\cF_\pi$.
Then, identifying the sum over $\{\bi:\pi(\bi)=\pi\}$ as a sum
over one distinct index in $[n]$ for each block $U \in \cK_\pi$, we have
\[\sum_{\ib \in [n]^{\cV}:\pi(\bi)=\pi} q_{\bi|T} \cdot W_{\bi|T}
=\sum_{\bi \in [n]^{\cK_\pi}}^* Q_{\bi|G_\pi} \cdot W_{\bi|G_\pi}^e\]
where $\overset{*}{\sum}$ denotes the restriction of the summation to index tuples
$\bi=(i_U:U \in \cK_\pi) \in [n]^{\cK_\pi}$ having all indices distinct, and
\[Q_{\bi|G_\pi}=\prod_{U \in \cK_\pi} Q_U(x_1[i_U], \ldots, x_k[i_U]),
\qquad W_{\bi|G_\pi}^e=\prod_{\text{unique edges } (U,V) \text{ of } G_\pi}
W[i_U,i_V]^{e(U,V)}.\]
Applying this to (\ref{eq:sym_value_decomp}), we obtain
\begin{align}\label{eq:sym_value_decomp2}
\val_T(\bW; \xb_1, \ldots, \xb_k) = \frac{1}{n} \sum_{\pi \in \cP}
\sum_{\ib \in [n]^{\cK_\pi}}^* Q_{\bi|G_\pi} \cdot W_{\bi|G_\pi}^e.    
\end{align}

We will compute the expectation of (\ref{eq:sym_value_decomp2}), and see that
the only non-vanishing contributions in the limit $n \to \infty$
arise from partitions $\pi$ where $G_\pi$ is itself a tree and
$e(U,V)=2$ for each unique edge $(U,V)$ of $G_\pi$. These non-vanishing terms
may be related to the values of a reduced tensor network associated to
$G_\pi$, evaluated on the matrix $\bS/n$ in place of $\bW$.

In anticipation of this computation, we first show the following
lemma which establishes universality of the value of any tensor network
evaluated on $\S/n$.

\begin{lemma}\label{lem:VarianceProfile_universality}
Under the assumptions of Lemma~\ref{lemma:Wignermoments}, for any diagonal 
tensor network $T=(\cV,\cE,\{q_v\}_{v \in \cV})$ in $k$ variables,
\begin{align*}
\lim_{n \to \infty}
\val_{T}(\bS/n; \bx_1,\ldots,\bx_k)=\prod_{v \in \cV}
\E[q_v(X_1,\ldots,X_k)].
\end{align*}
\end{lemma}
\begin{proof}
Observe first that for any diagonal tensor network
$T=(\cV,\cE,\{q_v\}_{v \in \cV})$, we have
\begin{equation}\label{eq:absoluteqbound}
\frac{1}{n}\sum_{\bi \in [n]^{\cV}}
\prod_{v \in \cV} |q_v(x_{1:k})[i_v]| \cdot \prod_{(u,v) \in \cE} \frac{1}{n}
\leq C
\end{equation}
for a constant $C:=C(T)>0$, almost surely for all large $n$.
Indeed, since $T$ is a tree, 
we have $\frac{1}{n}\prod_{(u,v) \in \cE} \frac{1}{n}=n^{-|\cV|}$
in the above. Each function $|q_v|$ is continuous and satisfies the polynomial
growth condition (\ref{eq:polygrowth}), so (\ref{eq:absoluteqbound}) follows
from the assumption \eqref{eq:WignerMomentsWpConvergence}.

Now note that since $T$ is a tree, we can order its vertices as 
$1,2,\ldots,|\calV|$ such that removing one vertex at a time in this order,
the remaining graph is always still a tree. Denote the remaining tensor network
after removing vertices $1,\ldots,h-1$ by $T_h=(\calV_h,\calE_h,\{q_v\}_{v \geq h})$. 
The vertex $h$ has only one neighbor in $T_h$, which we denote by 
$u_h \in \{h+1,\ldots,|\calV|\}$. Then
\begin{align*}
&\val_{T}(\bS/n; \bx_{1:k})\\
&=\frac{1}{n}\sum_{\bi \in [n]^{\calV}} \prod_{v\in\calV} q_v(x_{1:k}[i_v]) 
\prod_{(u,v)\in\calE} \frac{S[i_u,i_v]}{n}\\
& =\frac{1}{n} \sum_{i_2, \ldots, i_{|\cV|} = 1}^n \bigg( \prod_{v\in\calV_2} q_v(x_{1:k}[i_v]) 
\prod_{(u,v) \in \calE_2} \frac{S[i_u,i_v]}{n} \bigg) 
\cdot \bigg(\sum_{i_1=1}^n q_1(x_{1:k}[i_1]) \frac{S[i_1,i_{u_1}]}{n} \bigg)\\
& = \frac{1}{n} \sum_{i_2, \ldots, i_{|\cV|} = 1}^n 
\bigg( \prod_{v \in \calV_2} q_v(x_{1:k}[i_v]) \prod_{(u,v) \in \calE_2} \frac{S[i_u,i_v]}{n}  \bigg) 
\cdot \Big(\langle q_1(\bx_{1:k}) \rangle \cdot \langle \bs_{i_{u_1}}\rangle +
\delta(i_{u_1})\Big) \\
& = \frac{1}{n}\sum_{i_2, \ldots, i_{|\cV|} = 1}^n 
\bigg( \prod_{v \in \calV_2} q_v(x_{1:k}[i_v]) \prod_{(u,v) \in \calE_2} \frac{S[i_u,i_v]}{n}  \bigg) 
\cdot \Big( \bbE[q_1(X_{1:k})] + \delta'(i_{u_1})\Big).
\end{align*}
Here, $\delta(i),\delta'(i)$ denote errors that 
satisfy $\lim_{n \to \infty} \max_{i \in [n]} |\delta(i)|,|\delta'(i)|=0$,
as follows from \eqref{eq:WignerMomentsVP},
the conditions of Definition~\ref{def:Wigner}(c), and
\eqref{eq:WignerMomentsWpConvergence}. Note that
\[\left|\frac{1}{n}\sum_{i_2, \ldots, i_{|\cV|} = 1}^n 
\bigg( \prod_{v \in \calV_2} q_v(x_{1:k}[i_v]) \prod_{(u,v) \in \calE_2} \frac{S[i_u,i_v]}{n}  \bigg) 
\cdot \delta'(i_{u_1})\right| \to 0\]
as $n \to \infty$
by the condition $|S[i_u,i_v]| \leq C$ of Definition \ref{def:Wigner}(c),
the bound (\ref{eq:absoluteqbound}) applied to the network $T_2$ with vertex 1
removed, and the convergence $\max_{i \in [n]} |\delta'(i)| \to 0$. Thus
\[\lim_{n \to \infty}
\val_{T}(\bS/n; \bx_{1:k})=
\lim_{n \to \infty} \frac{1}{n}\sum_{i_2, \ldots, i_{|\cV|} = 1}^n 
\bigg( \prod_{v \in \calV_2} q_v(x_{1:k}[i_v]) \prod_{(u,v) \in \calE_2}
\frac{S[i_u,i_v]}{n}  \bigg) \cdot \bbE[q_1(X_{1:k})].\]
Repeating the above procedure by removing vertices $1,2,\ldots,|\cV|-1$
sequentially, we are left with the single vertex $|\cV|$ and no edges, and
\[\lim_{n \to \infty} \val_{T}(\bS/n; \bx_{1:k})
=\lim_{n \to\infty} \frac{1}{n}\sum_{i_{|\cV|}=1}^n q_{|\cV|}(x_{1:k}[i_{|\cV|}])
\cdot \prod_{v=1}^{|\cV|-1} \bbE[q_v(X_{1:k})]
=\prod_{v=1}^{|\cV|} \bbE[q_v(X_{1:k})].\]
\end{proof}

The next lemma relates summations over distinct indices to ones without
the distinctness requirement.

\begin{lemma}\label{lemma:distinctness}
Let $\x_1,\ldots,\x_k \in \R^n$ and $(X_1,\ldots,X_k)$ be such that
\[(\x_1,\ldots,\x_k) \toW (X_1,\ldots,X_k).\]
For a finite index set $\cS$, let $(q_s:s \in \cS)$ be $|\cS|$
continuous functions satisfying the polynomial
growth condition~(\ref{eq:polygrowth}) for some order $p \geq 1$. Then
\begin{align*}
\lim_{n \to \infty} \frac{1}{n^{|\cS|}}\sum_{\bi \in [n]^\cS}^*
\prod_{s \in \cS} q_s(x_1[i_s],\ldots,x_k[i_s])
&=\lim_{n \to \infty} \frac{1}{n^{|\cS|}}\sum_{\bi \in [n]^\cS}
\prod_{s \in \cS} q_s(x_1[i_s],\ldots,x_k[i_s])\\
&=\prod_{s \in \cS} \E[q_s(X_1,\ldots,X_k)].
\end{align*}
\end{lemma}
\begin{proof}
Let $\cP$ be the set of partitions of $\cS$, and let
$\pi(\bi) \in \cP$ be the partition induced by $\bi \in [n]^{\cS}$.
Let $0_{\cP}$ the
partition having $|\cS|$ singleton blocks, corresponding to $\bi$ having all
indices distinct. Then for the first equality, it suffices to show that
\begin{equation}\label{eq:sym_S_residual}
\Delta:=\frac{1}{n^{|\cS|}}
\sum_{\pi \in \cP: \pi \neq 0_{\cP}}
\sum_{\bi \in [n]^{\cS}:\pi(\bi)=\pi} \prod_{s \in \cS} |q_s(x_{1:k}[i_s])|
\end{equation}
vanishes as $n \to \infty$.

For any $\pi \in \cP$ and block $R \in \pi$, define
$Q_R=\prod_{u \in R} q_u$, and let $|\pi|$ be the number of blocks of
$\pi$. Then, identifying the sum over $\{\bi:\pi(\bi)=\pi\}$ with the sum
over one distinct index in $[n]$ for each block of $\pi$,
\[\frac{1}{n^{|\pi|}} \sum_{\ib \in [n]^{\cS}:\pi(\bi)=\pi}
\prod_{s \in \cS} |q_s(x_{1:k}[i_s])|
=\frac{1}{n^{|\pi|}} \sum_{\bi \in [n]^\pi}^*
\prod_{R \in \pi} |Q_R(x_{1:k}[i_R])|.\]
As an upper bound, adding back the excluded index tuples $\bi \in [n]^{\pi}$
where some indices coincide,
\[\frac{1}{n^{|\pi|}} \sum_{\ib \in [n]^{\cS}:\pi(\bi)=\pi}
\prod_{s \in \cS} |q_s(x_{1:k}[i_s])|
\leq \prod_{R \in \pi} \left(\frac{1}{n}\sum_{i=1}^n |Q_R(x_{1:k}[i])|\right).\]
Since $\x_{1:k} \toW X_{1:k}$
and $|Q_R|$ is a continuous function satisfying the polynomial
growth condition (\ref{eq:polygrowth}),
this upper bound is at most a constant $C(\pi)$ for
all large $n$. For any $\pi \neq 0_\cP$, we have $|\pi| \leq |\cS|-1$.
As the number of partitions $\pi \in \cP$ is independent of $n$, applying these
observations to (\ref{eq:sym_S_residual}) shows
$\Delta \leq C/n$ for a constant $C>0$ and all large $n$,
and hence $\Delta \to 0$ as desired.

The second equality of the lemma follows from the given condition $\x_{1:k} \toW 
X_{1:k}$, hence
\[\frac{1}{n^{|\cS|}}\sum_{\bi \in [n]^\cS}
\prod_{s \in \cS} q_s(x_{1:k}[i_s])=\prod_{s \in \cS} \frac{1}{n}\sum_{i=1}^n
q_s(x_{1:k}[i]) \to \prod_{s \in \cS} \E[q_s(X_{1:k})].\]
\end{proof}

We now show that the limit tensor network value is universal in
expectation over $\W$.

\begin{lemma}\label{lem:Wignermoments_expected_value}
Let $\EE$ denote the expectation over $\W$, conditional on $\x_1,\ldots,\x_k$.
Then Lemma~\ref{lemma:Wignermoments} holds for
$\EE[\val_T(\W;\x_1,\ldots,\x_k)]$ in place of $\val_T(\W;\x_1,\ldots,\x_k)$.
\end{lemma}
\begin{proof}
Recall the decomposition of $\val_T(\bW; \xb_1, \ldots, \xb_k)$
in~\eqref{eq:sym_value_decomp2}, where $\cP$ is the set of partitions of the
vertices $\cV$ of $T$. Taking expectation on both sides yields
\begin{equation}\label{eq:Wignerval1}
\EE[\val_T(\bW; \xb_{1:k})]=\frac{1}{n} \sum_{\pi \in \cP}
\sum_{\bi \in [n]^{\cK_\pi}}^* Q_{\bi|G_\pi} \cdot \E[W_{\bi|G_\pi}^e].
\end{equation}
First note that, since the indices of $\bi$ are distinct and all entries of
$\bW$ have mean 0, $\EE[W_{\bi|G_\pi}^e]$ is nonzero only if 
each unique edge of $G_\pi=(\cK_\pi,\cF_\pi)$ appears at least twice.
Let $|\cK_\pi|$ and $|\cF_\pi|_*$ be the number of vertices and number
of \emph{unique} (undirected) edges of $G_\pi$. The graph $G_\pi$ must be connected
since the original tree $T$ was connected, so $|\cK_\pi| \leq
|\cF_\pi|_*+1$. Then any $G_\pi$ where each unique edge appears at least
twice has $|\cK_\pi| \leq |\cF_\pi|_*+1 \leq |\cE|/2+1$,
where $|\cE|$ is the number of edges of the original tree $T$.

Furthermore, we claim that the contribution from partitions $\pi$ where
$|\cK_\pi| \leq |\cE|/2$ is negligible.
To see this, we apply Definition~\ref{def:Wigner}(b) to get 
\begin{align*}
    \left|\EE[W_{\bi|G_\pi}^e]\right| &=
    \prod_{(U,V):e(U,V)=2} \EE[W[i_U,i_V]^2]
    \prod_{\substack{(U,V):e(U,V)>2}} \left|\EE[W[i_U,i_V]^{e(U,V)}]\right|\\
    &\leq \prod_{\substack{(U,V):e(U,V)=2}} \frac{C}{n}
    \prod_{\substack{(U,V):e(U,V)>2}} \frac{o(1)}{n}.
\end{align*}
If there is an edge $(U,V)$ of $G_\pi$ with $e(U,V)>2$, then this shows
$|\EE[W_{\bi|G_\pi}^e]| \leq o(1)/n^{|\cF_\pi|_*}
\leq o(1)/n^{|\cK_\pi|-1}$. If, conversely, every edge in $G_\pi$ appears exactly
twice, then by assumption $|\cF_\pi|_*=|\cE|/2 \geq |\cK_\pi|$,
so this shows $|\EE[W_{\bi|G_\pi}^e]| \leq (C/n)^{|\cF_\pi|_*}
\leq o(1)/n^{|\cK_\pi|-1}$ also. Therefore,
\[\bigg|\frac{1}{n}\sum_{\bi \in [n]^{\cK_\pi}}^* Q_{\bi|G_\pi}
\cdot \E[W_{\bi|G_\pi}^e]\bigg|
\leq \frac{o(1)}{n^{|\cK_\pi|}} \sum_{\ib \in
[n]^{\cK_\pi}}^* \left|Q_{\bi|G_\pi}\right|.\]
As an upper bound,
adding back the excluded tuples $\bi \in [n]^{\cK_\pi}$ where not all
indices are distinct, we have
\begin{align}
\frac{1}{n^{|\cK_\pi|}} \sum_{\ib \in [n]^{\cK_\pi}}^* |Q_{\bi|G_\pi}|
&\leq \prod_{U \in \cK_\pi} \frac{1}{n}
\sum_{i=1}^n |Q_U(x_{1:k}[i])|\label{eq:Qbound}
\end{align}
By \eqref{eq:WignerMomentsWpConvergence},
this upper bound is at most a constant $C(\pi)$ for all large $n$, so
\[\bigg|\frac{1}{n}\sum_{\bi \in [n]^{\cK_\pi}}^* Q_{\bi|G_\pi}
\cdot \E[W_{\bi|G_\pi}^e]\bigg| \to 0\]
as claimed.

Thus the only non-vanishing contributions to (\ref{eq:Wignerval1})
come from partitions $\pi$
where $|\cK_\pi|=|\cF_\pi|_*+1=|\cE|/2+1$.
Then each unique edge of $G_\pi$ appears exactly twice, and these edges
form a tree. In this case, we have
\begin{equation}\label{eq:reductiontoGnetwork}
\frac{1}{n}\sum_{\bi \in [n]^{\cK_\pi}}^* Q_{\bi|G_\pi}
\cdot \E[W_{\bi|G_\pi}^e]
=\frac{1}{n}\sum_{\bi \in [n]^{\cK_\pi}}^* 
Q_{\bi|G_\pi} \cdot \prod_{\text{unique edges } (U,V)\text{ of } G_\pi}
\frac{S[i_U,i_V]}{n}.
\end{equation}
Let $\cI_*$ be the set of tuples $\bi \in [n]^{\cK_\pi}$
where all indices are distinct. Then, applying $|S[i_U,i_V]| \leq C$ from
Definition \ref{def:Wigner}(c), for a constant $C'=C(\pi)>0$,
\[\bigg|\frac{1}{n} \sum_{\ib \in [n]^{\cK_\pi} \setminus \cI_*}
Q_{\bi|G_\pi} \prod_{\text{unique edges } (U,V) \text{ of } G_\pi} 
\frac{S[i_U,i_V]}{n}\bigg|
\leq \frac{C'}{n^{1+|\cF_\pi|_*}} 
\sum_{\ib \in [n]^{\cK_\pi} \setminus \cI_*} |Q_{\bi|G_\pi}|.\]
Since $|\cK_\pi|=|\cF_\pi|_*+1$, the first equality of
Lemma \ref{lemma:distinctness} shows
that this vanishes as $n \to \infty$. Then the right side of
(\ref{eq:reductiontoGnetwork}) has the same limit as
\[\frac{1}{n}\sum_{\bi \in [n]^{\cK_\pi}}
Q_{\bi|G_\pi} \cdot \prod_{\text{unique edges } (U,V)\text{ of } G_\pi}
\frac{S[i_U,i_V]}{n}.\]
This is the limit value of the
tensor network $T_\pi=(\cK_\pi,\text{unique edges of
}\cF_\pi,\{Q_U\}_{U \in \cK_\pi})$ applied to $\S/n$, which by
Lemma~\ref{lem:VarianceProfile_universality} equals
$\prod_{U \in \cK_\pi} \E[Q_U(X_{1:k})]$. Applying this back to
(\ref{eq:reductiontoGnetwork}) and (\ref{eq:Wignerval1}),
\begin{align}
\lim_{n \to \infty} \E[\val_T(\W;\x_{1:k})]
&=\sum_{\pi \in \cP:\,|\cK_\pi|=|\cF_\pi|_*+1=|\cE|/2+1}\;\;
\prod_{U \in \cK_\pi} \E[Q_U(X_{1:k})]\label{eq:limit}\\
&=:\limval_T(X_1,\ldots,X_k).\notag
\end{align}
This limit depends only
on $T$ and the joint law of $X_1,\ldots,X_k$, concluding the proof.
\end{proof}

We make a brief interlude to show here the asymptotic freeness result of 
Proposition \ref{prop:freeness}(a).

\begin{proof}[Proof of Proposition \ref{prop:freeness}(a)]
From the preceding proof, only partitions $\pi$ where
$|\cK_\pi|=|\cE|/2+1$ contribute to $\limval_T(X_1,\ldots,X_k)$.
For every such partition, since $T$ has $|\cE|+1$ vertices, this implies that
some vertex of $\cK_\pi$, i.e.\ some block $U$ of $\pi$,
contains only a single vertex $v$ of $T$. For this block $U$, we have
\[\E[Q_U(X_{1:k})]=\E[q_v(X_{1:k})]=0\]
by the condition (\ref{eq:qnormalized}). 
Therefore, every summand in \eqref{eq:limit} vanishes, implying
as desired $\limval_T(X_1,\ldots,X_k)=0$.
\end{proof}

To complete the proof of Lemma~\ref{lemma:Wignermoments},
it remains to establish the concentration of the value of any tensor 
network around its mean as $n \to \infty$.

\begin{lemma}\label{lem:Wignermoments_concentration}
Let $\EE$ denote the expectation over $\W$, conditional on $\x_1,\ldots,\x_k$.
    Under the setting of Lemma~\ref{lemma:Wignermoments}, almost 
    surely as $n \to \infty$, 
    \begin{align*}
        \val_T(\bW; \xb_1, \ldots, \xb_k) 
        - \EE[\val_T(\bW; \xb_1, \ldots, \xb_k)] \to 0.
    \end{align*}
\end{lemma}
\begin{proof}
We write $\val(\W)=\val_T(\W;\x_1,\ldots,\x_k)$. We will bound the
fourth moment of $\val(\bW)-\E[\val(\bW)]$ and apply the Borel-Cantelli lemma.
(Note that $\val(\bW)-\E[\val(\bW)]$ typically fluctuates on the order of
$1/\sqrt{n}$, so that bounding the variance would not suffice to show
almost-sure convergence.)

First, we expand
\begin{align}
&\EE[\left(\val(\bW) - \EE[\val(\bW)]\right)^4]\nonumber\\
    &= \EE[\val(\bW)^4] - 4 \EE[\val(\bW)^3] \EE[\val(\bW)]
    + 6 \EE[\val(\bW)^2] \EE[\val(\bW)]^2 - 3 \EE[\val(\bW)]^4.
\label{eq:sym_moment_expand}
\end{align}
We introduce four independent copies of the matrix $\W$ as
$\bW^{(1)},\bW^{(2)},\bW^{(3)},\bW^{(4)}$, define four index tuples
$\bi_1,\bi_2,\bi_3,\bi_4 \in [n]^{\cV}$, and write as shorthand
\begin{align*}
q_{\bi_{1:4}}=q_{\bi_1|T} \cdot q_{\bi_2|T} \cdot q_{\bi_3|T} \cdot q_{\bi_4|T},
\qquad W_{\bi_{1:4}}^{(a_1,a_2,a_3,a_4)}
= W_{\bi_1|T}^{(a_1)} \cdot W_{\bi_2|T}^{(a_2)} \cdot W_{\bi_3|T}^{(a_3)}
\cdot W_{\bi_4|T}^{(a_4)}
\end{align*}
where each $W_{\bi|T}^{(a)}$ is defined by the copy $\W^{(a)}$. Then
\begin{align}\label{eq:sym_moments}
\begin{aligned}
    \EE[\val(\bW)^4] &= \frac{1}{n^4} \sum_{\ib_1,\ldots,\ib_4 \in [n]^{\cV}} 
    q_{\ib_{1:4}} \cdot \EE[W_{\bi_{1:4}}^{(1,1,1,1)}],\\
    \EE[\val(\bW)^3] \EE[\val(\bW)] &= \frac{1}{n^4} 
    \sum_{\ib_1,\ldots,\ib_4 \in [n]^{\cV}} 
    q_{\ib_{1:4}} \cdot \EE[W_{\bi_{1:4}}^{(1,1,1,2)}],\\
    \EE[\val(\bW)^2] \EE[\val(\bW)]^2 &= \frac{1}{n^4} 
    \sum_{\ib_1,\ldots,\ib_4 \in [n]^{\cV}} 
    q_{\ib_{1:4}} \cdot \EE[W_{\bi_{1:4}}^{(1,1,2,3)}],\\
    \EE[\val(\bW)]^4 &= \frac{1}{n^4} \sum_{\ib_1,\ldots,\ib_4 \in [n]^{\cV}} 
    q_{\ib_{1:4}} \cdot \EE[W_{\bi_{1:4}}^{(1,2,3,4)}].
\end{aligned}
\end{align}

Corresponding to each index tuple $\bi_{1:4}$, consider a multi-graph
$G(\ib_{1:4})$ whose vertices are the unique index values in $\bi_{1:4}$,
with one edge $(i_{a,u},i_{a,v})$ for every combination of
$a=1,2,3,4$ and edge $(u,v) \in \cE$, counting multiplicity.
(One may visualize $G(\ib_{1:4})$ as a
multi-graph whose vertices are a subset of $[n]$, and having edges of 4 colors
corresponding to $a=1,2,3,4$.) Then the edges of $G(\ib_{1:4})$
corresponding to each single index $a=1,2,3,4$ must belong to a single connected
component, so the number of connected 
components in $G(\ib_{1:4})$ can be either 1, 2, 3 or 4.
Let us partition the index tuples $\bi_1,\bi_2,\bi_3,\bi_4 \in [n]^\cV$
into the three sets
\begin{align*}
    \cI_2 &= \{\ib_{1:4}: G(\ib_{1:4}) \text{ has 1 or 2 connected components}\},\\
    \cI_3 &= \{\ib_{1:4}: G(\ib_{1:4}) \text{ has 3 connected components}\},\\
    \cI_4 &= \{\ib_{1:4}: G(\ib_{1:4}) \text{ has 4 connected components}\},
\end{align*}
and define correspondingly for $j = 2, 3, 4$,
\begin{align}\label{eq:sym_moment_decomp}
    A_j = \frac{1}{n^4} \sum_{\ib_{1:4} \in \cI_j} q_{\ib_{1:4}} 
    \Big(\EE[W_{\ib_{1:4}}^{(1, 1, 1, 1)}] 
    - 4 \cdot \EE[W_{\ib_{1:4}}^{(1, 1, 1, 2)}]
    + 6 \cdot \EE[W_{\ib_{1:4}}^{(1, 1, 2, 3)}] 
    - 3 \cdot \EE[W_{\ib_{1:4}}^{(1, 2, 3, 4)}]\Big).
\end{align}
Then by \eqref{eq:sym_moment_expand} and \eqref{eq:sym_moments}, we have 
$\EE[(\val(\bW) - \EE[\val(\bW)])^4] = A_2 + A_3 + A_4$.
Below, we will show that $A_3 = A_4 = 0$ and $A_2=O(1/n^2)$ as $n \to \infty$.

For $A_4$, observe that
since $G(\ib_{1:4})$ has 4 connected components, the tuples
$\ib_1, \ib_2, \ib_3, \ib_4$ have no common indices.
Then due to the independence between entries of $\bW$, we have 
\begin{align*}
    \EE[W^{(a_1, a_2, a_3, a_4)}_{\ib_{1:4}}] &= \EE[W^{(a_1)}_{\ib_1}] \cdot 
    \EE[W^{(a_2)}_{\ib_2}] \cdot \EE[W^{(a_3)}_{\ib_3}] \cdot
\EE[W^{(a_4)}_{\ib_4}]\\
    &= \EE[W^{(1)}_{\ib_1}] \cdot \EE[W^{(1)}_{\ib_2}] \cdot \EE[W^{(1)}_{\ib_3}] 
    \cdot \EE[W^{(1)}_{\ib_4}]
    = \EE[W_{\ib_{1:4}}^{(1, 1, 1, 1)}]
\end{align*}
for any $a_1, a_2, a_3, a_4$.
Applying this to $A_4$ defined in \eqref{eq:sym_moment_decomp}, we get $A_4=0$.

Next, for $A_3$, we write $\ib_j \parallel \ib_{j'}$ if $\ib_j$ and
$\ib_{j'}$ share at
least one index. Note that for any $\ib_{1:4} \in \cI_3$, there is a unique pair
$\ib_j, \ib_{j'}$ such that $\ib_j \parallel \ib_{j'}$, so
\begin{align*}
    \cI_3 = \bigsqcup_{1 \leq j<j' \leq 4} \{\ib_{1:4} \in \cI_3:
    \ib_j \parallel \ib_{j'}\}.
\end{align*}
We will repeatedly apply this six-fold decomposition of $\cI_3$ and the
independence between the different copies of $\bW$.
If $\bi_3 \parallel \bi_4$, we have
\begin{align*}
    \EE[W_{\ib_{1:4}}^{(1, 1, 1, 1)}] &= \EE[W_{\ib_1}^{(1)}] \cdot 
    \EE[W_{\ib_2}^{(1)}] \cdot \EE[W_{\ib_3}^{(1)} W_{\ib_4}^{(1)}]
    = \EE[W_{\ib_{1:4}}^{(1, 2, 3, 3)}]
\end{align*}
which together with permutation symmetry between the labels 1, 2, 3, and 4
further implies that 
\begin{align}\label{eq:sym_moment_decomp_2_1}
    \sum_{\ib_{1:4} \in \cI_3} q_{\ib_{1:4}} \cdot \EE[W_{\ib_{1:4}}^{(1, 1, 1, 1)}]
    &= 6 \sum_{\ib_{1:4} \in \cI_3, \ib_3 \parallel \ib_4} 
    q_{\ib_{1:4}} \cdot \EE[W_{\ib_{1:4}}^{(1, 1, 1, 1)}]
    = 6 \sum_{\ib_{1:4} \in \cI_3, \ib_3 \parallel \ib_4} q_{\ib_{1:4}} 
    \cdot \EE[W_{\ib_{1:4}}^{(1, 2, 3, 3)}].
\end{align}
Similarly, considering the two cases $\bi_3 \parallel \bi_4$
and $\bi_2 \parallel \bi_3$ and their symmetric equivalents,
\begin{align}\label{eq:sym_moment_decomp_2_2}
    \sum_{\ib_{1:4} \in \cI_3} q_{\ib_{1:4}} \cdot \EE[W_{\ib_{1:4}}^{(1, 1, 1, 2)}]
    &= 3 \sum_{\ib_{1:4} \in \cI_3, \ib_3 \parallel \ib_4}
    q_{\ib_{1:4}} \cdot \EE[W_{\ib_{1:4}}^{(1, 2, 3, 4)}]
    + 3 \sum_{\ib_{1:4} \in \cI_3, \ib_2 \parallel \ib_3} q_{\ib_{1:4}} 
    \cdot \EE[W_{\ib_{1:4}}^{(1, 2, 2, 3)}]\notag \\
    &= 3 \sum_{\ib_{1:4} \in \cI_3, \ib_3 \parallel \ib_4}
    q_{\ib_{1:4}} \cdot \EE[W_{\ib_{1:4}}^{(1, 2, 3, 4)}]
    + 3 \sum_{\ib_{1:4} \in \cI_3, \ib_3 \parallel \ib_4} q_{\ib_{1:4}} 
    \cdot \EE[W_{\ib_{1:4}}^{(1, 2, 3, 3)}].
\end{align}
Considering the three cases $\bi_3 \parallel \bi_4$, $\bi_2 \parallel \bi_4$,
$\bi_1 \parallel \bi_2$ and their symmetric equivalents,
\begin{align}\label{eq:sym_moment_decomp_2_3}
    &\sum_{\ib_{1:4} \in \cI_3} q_{\ib_{1:4}} \cdot \EE[W_{\ib_{1:4}}^{(1, 1, 2,
3)}]\notag\\
    &= \sum_{\ib_{1:4} \in \cI_3, \ib_3 \parallel \ib_4}
    q_{\ib_{1:4}} \cdot \EE[W_{\ib_{1:4}}^{(1, 2, 3, 4)}]
    + 4 \sum_{\ib_{1:4} \in \cI_3, \ib_2 \parallel \ib_3} q_{\ib_{1:4}} 
    \cdot \EE[W_{\ib_{1:4}}^{(1, 2, 3, 4)}]
    + \sum_{\ib_{1:4} \in \cI_3, \ib_1 \parallel \ib_2} q_{\ib_{1:4}} 
    \cdot \EE[W_{\ib_{1:4}}^{(1, 1, 2, 3)}] \notag\\
    &= 5\sum_{\ib_{1:4} \in \cI_3, \ib_3 \parallel \ib_4} q_{\ib_{1:4}} 
    \cdot \EE[W_{\ib_{1:4}}^{(1, 2, 3, 4)}]
    + \sum_{\ib_{1:4} \in \cI_3, \ib_3 \parallel \ib_4} q_{\ib_{1:4}} 
    \cdot \EE[W_{\ib_{1:4}}^{(1, 2, 3, 3)}].
\end{align}
Finally, by symmetry,
\begin{align}\label{eq:sym_moment_decomp_2_4}
    \sum_{\ib_{1:4} \in \cI_3} q_{\ib_{1:4}} \cdot \EE[W_{\ib_{1:4}}^{(1, 2, 3,
4)}]&=6\sum_{\ib_{1:4} \in \cI_3, \ib_3 \parallel \ib_4}
    q_{\ib_{1:4}} \cdot \EE[W_{\ib_{1:4}}^{(1, 2, 3, 4)}].
\end{align}
Collecting \eqref{eq:sym_moment_decomp_2_1}, \eqref{eq:sym_moment_decomp_2_2},
\eqref{eq:sym_moment_decomp_2_3}, and \eqref{eq:sym_moment_decomp_2_4}
and applying them to $A_3$ defined in \eqref{eq:sym_moment_decomp}, we get
$A_3=0$.
 
Finally, we bound $A_2$. Let $|\cV_{G(\ib_{1:4})}|$ and
$|\cE_{G(\ib_{1:4})}|_*$ be the number of vertices and number of
unique (undirected) edges of
$G(\ib_{1:4})$. Since $G(\ib_{1:4})$ has at most 2 connected
components, we have $|\cV_{G(\ib_{1:4})}| \leq |\cE_{G(\ib_{1:4})}|_*+2$.
By Definition~\ref{def:Wigner}(b), for a constant $C>0$ and any
$a_1,a_2,a_3,a_4$, we have 
$|\EE[W_{\ib_{1:4}}^{(a_1, a_2, a_3, a_4)}]| \leq
C/n^{|\cE_{G(\ib_{1:4})}|_*} \leq C/n^{|\cV_{G(\ib_{1:4})}|-2}$. Therefore, 
\begin{align*}
    \frac{1}{n^4} \bigg|\sum_{\ib_{1:4} \in \cI_2} q_{\ib_{1:4}} 
    \EE[\bW_{\ib_{1:4}}^{(a_1, a_2, a_3, a_4)}]\bigg| &\leq \frac{1}{n^4} 
\sum_{v=1}^{4|\cV|} \frac{C}{n^{v-2}} 
    \sum_{\ib_{1:4} \in \cI_2:\,|\cV_{G(\ib_{1:4})}|=v} |q_{\ib_{1:4}}|.
\end{align*}
Stratifying the inner summation over $\{\bi_{1:4} \in
\cI_2:|\cV_{G(\bi_{1:4})}|=v\}$ by its induced partition
$\pi(\bi_{1:4})$ of the $4|\cV|$ total indices (having exactly $v$
blocks), and applying the same argument as in
(\ref{eq:Qbound}), this inner summation may be bounded
as $\sum_{\ib_{1:4} \in \cI_2:|\cV_{G(\ib_{1:4})}|=v} |q_{\ib_{1:4}}| \leq
Cn^v$ for a constant $C>0$. Applying this bound for each term of $A_2$,
we obtain $|A_2| \leq C/n^2$.

Combining the analyses of $A_2$, $A_3$, and $A_4$, we get 
$\EE[(\val(\bW) - \EE[\val(\bW)])^4] \leq C/n^2$.
Then by Markov's inequality, for any $\eps > 0$,
$\PP[|\val(\bW) - \EE[\val(\bW)]| > \eps] \leq C/(\eps^4 n^2)$. This
bound is summable over all $n \geq 1$, so almost-sure convergence follows by
the Borel-Cantelli lemma.
\end{proof}

Combining Lemmas~\ref{lem:Wignermoments_expected_value} and
\ref{lem:Wignermoments_concentration} concludes the proof of
Lemma~\ref{lemma:Wignermoments}.

\subsection{Universality for symmetric generalized invariant matrices}

In this section, we prove Lemma~\ref{lemma:syminvmoments} on the universality
of tensor network values for generalized invariant matrices.

Fix the tensor network $T=(\cV,\cE,\{q_v\}_{v \in \cV})$.
Expanding the product $\W=\bPi\bM\bPi^\top$, the tensor network value is given by
\[\val_T(\W;\x_{1:k})=\frac{1}{n}\sum_{\ib \in [n]^{\cV}}
\sum_{\bj,\bl \in [n]^{\cE}}
\prod_{v \in \cV} q_v(x_{1:k}[i_v])
\prod_{e=(u,v) \in \cE} \Pi[i_u,j_e]M[j_e,l_e]\Pi[i_v,l_e].\]
The matrix $\bPi$ may be written as $\bPi=\bXi\bP$
where $\bXi$ is a random sign matrix and $\bP$ is a random permutation matrix
independent of $\bXi$.
Let $\sigma$ denote the permutation of $[n]$ for which
$P[i,\sigma(i)]=1$ for all $i\in[n]$.
Then $\Pi[i_u,j_e]$ is non-zero if and only if $j_e=\sigma(i_u)$.
Therefore, the tensor network value is equivalently expressed as
\begin{equation}\label{eq:invariantval}
	\val_T(\W;\x_{1:k})=\frac{1}{n}\sum_{\ib \in [n]^{\cV}}
	\prod_{v \in \cV} q_v(x_{1:k}[i_v]) \prod_{e=(u,v) \in \cE} 
	\Xi[i_u]\cdot\Xi[i_v]\cdot M[\sigma(i_u),\sigma(i_v)].
\end{equation}

Let $\cP$ be the set of partitions of $\cV$. For each $\pi \in \cP$, let
$G_\pi=(\cK_\pi,\cF_\pi)$ be the image of $(\cV,\cE)$ under $\pi$, in the sense of
Definition \ref{def:graphpartition}. For each $\bi\in [n]^\cV$,
let $\pi(\bi) \in \cP$ be the partition induced by $\bi$.
Stratifying the summation over $\bi\in[n]^\cV$ by its
induced partition $\pi(\bi)$,
\[\val_T(\W;\x_{1:k})=\sum_{\pi \in \cP} \frac{1}{n}\sum_{\ib \in
[n]^{\cV}:\pi(\bi)=\pi}
	\prod_{v \in \cV} q_v(x_{1:k}[i_v]) \prod_{e=(u,v) \in \cE} 
	\Xi[i_u]\cdot\Xi[i_v]\cdot M[\sigma(i_u),\sigma(i_v)].\]
Then, defining $Q_R=\prod_{u \in R} q_u$ for the blocks $R \in \cK_\pi \equiv \pi$,
and identifying the sum over $\{\bi\in[n]^\cV:\pi(\bi)=\pi\}$
as a sum over one distinct index for each block
$R \in \cK_\pi$, we have
\begin{align*}
	\val_T(\W;\x_{1:k}) &= \sum_{\pi\in\cP}\frac{1}{n}
	\sum_{\bi\in[n]^{\cK_\pi}}^* \prod_{R \in \cK_\pi} Q_R(x_{1:k}[i_R])
	\prod_{(R,S)\in\cF_\pi} \Xi[i_R] \cdot \Xi[i_S] \cdot
	M[\sigma(i_R),\sigma(i_S)]
\end{align*}
where we recall the notation that $\overset{*}{\sum}$ restricts the summation 
to index tuples $\bi$ having all indices distinct.
For every $R\in\cK_\pi$, let $\degext(R)$ be its \emph{external degree} in $G_\pi$,
i.e.\ the total number of edges of $\cF_\pi$ containing $R$ (counting multiplicity)
that are not self-loops. Then for every $R\in\cK_\pi$, the number of times the
factor $\Xi[i_R]$ appears in the above product is exactly $\degext(R)$ plus
twice the number of self-loops on $R$. Since $\Xi[i_R]^2=1$, this implies
\begin{align}\label{eq:orthovalasO}
	\val_T(\W;\x_{1:k}) &= \sum_{\pi\in\cP}\frac{1}{n}
	\sum_{\bi\in[n]^{\cK_\pi}}^* \prod_{R \in \cK_\pi} Q_R(x_{1:k}[i_R])
	\Xi[i_R]^{\degext(R)} \prod_{(R,S)\in\cF_\pi} M[\sigma(i_R),\sigma(i_S)].
\end{align}

\begin{lemma}\label{lemma:Esyminvmoments}
Let $\E$ be the expectation over $\bPi$
conditional on $\bM$ and $\x_1,\ldots,\x_k$.
Then Lemma~\ref{lemma:syminvmoments} holds for
$\E[\val_T(\W;\x_1,\ldots,\x_k)]$ in place of $\val_T(\W;\x_1,\ldots,\x_k)$.
\end{lemma}

\begin{proof}
Taking expectations in (\ref{eq:orthovalasO})
with respect to the independent signs $\bXi$ and
permutation $\sigma$, observe that
\begin{itemize}
\item If $R\in\cK_\pi$ is such that 
$\degext(R)$ is odd, then $\EE[\Xi[i_R]^{\degext(R)}]=0$.
Thus by independence of the diagonal entries of $\bXi$ and distinctness of the
indices of $\bi$,
\begin{align}\label{eq:sign}
	\EE\bigg[\prod_{R\in\cK_\pi} \Xi[i_R]^{\degext(R)}\bigg] = 
	\1\{\degext(R) \text{ is even for all } R\in\cK_\pi\}.
\end{align}
\item Since $\sigma$ is a uniformly random permutation on $[n]$,
for any fixed tuple $\bi \in [n]^{\cK_\pi}$ with all entries distinct,
\begin{align}\label{eq:permutation}
	\EE\bigg[\prod_{(R,S)\in\cF_\pi} M[\sigma(i_R),\sigma(i_S)]\bigg]
	&= \frac{(n-|\cK_\pi|)!}{n!} \sum_{\bj\in[n]^{\cK_\pi}}^* 
	\prod_{(R,S)\in\cF_\pi} M[j_R,j_S]
\end{align}
where $n!/(n-|\cK_\pi|)!$ counts the total number of tuples
$\bj \in [n]^{\cK_\pi}$ having distinct entries, and the right side
represents a uniform average over such tuples $\bj$.
\end{itemize}

Let us call a partition $\pi\in\cP$ \emph{even} if  
every vertex of $\cK_\pi$ has even external degree.
Then applying the above observations to take the expectation in
\eqref{eq:orthovalasO}, we obtain
\begin{equation}\label{eq:orthoval3}
\EE[\val_T(\W;\x_{1:k})] = \sum_{\text{even }\pi\in\cP}
B_n(\pi) \cdot Q_n(\pi) \cdot M_n(\pi)
\end{equation}
where we set
\begin{align}
B_n(\pi)&=n^{|\cK_\pi|} \cdot \frac{(n-|\cK_\pi|)!}{n!},\label{eq:An}\\
Q_n(\pi)&=\frac{1}{n^{|\cK_\pi|}} \sum_{\bi \in [n]^{\cK_\pi}}^*
	\prod_{R \in \cK_\pi} Q_R(x_{1:k}[i_R]),\label{eq:Qn}\\
M_n(\pi)&=\frac{1}{n}\sum_{\bj \in [n]^{\cK_\pi}}^* 
	\prod_{(R,S)\in\cF_\pi} M[j_R,j_S].\label{eq:Mn}
\end{align}

It is clear that $\lim_{n \to \infty} B_n(\pi)=1$ for every fixed $\pi$.
For $Q_n(\pi)$, we may apply Lemma~\ref{lemma:distinctness} with the 
identifications $\cS \leftrightarrow \cK_\pi$
and $\{q_s:s \in \cS\} \leftrightarrow \{Q_R:R \in \cK_\pi\}$. Then
\[\lim_{n \to \infty} Q_n(\pi)=\prod_{R \in \cK_\pi} \EE[Q_R(X_{1:k})].\]
For $M_n(\pi)$, since the original tensor network is connected, the graph
$G_\pi=(\cK_\pi,\cF_\pi)$ corresponding to each partition $\pi$
must also be connected.
Consequently, applying Lemma~\ref{lemma:Hlemma} below (with $p_e(\bM)=\bM$ for
every edge $e$), there exists a 
deterministic limit value $M(G_\pi,\cD_\tdiag)$ depending only on $G_\pi$ and 
$\cD_\tdiag$ such that, almost surely,
$\lim_{n \to \infty} M_n(\pi)=M(G_\pi,\cD_\tdiag)$.
Applying these statements to
every $\pi$ in \eqref{eq:orthoval3}, we obtain 
\begin{align*}
	\lim_{n\to\infty}\E[\val_T(\bW;\x_{1:k})] &= 
	\sum_{\text{even } \pi \in \cP}
\bigg(\prod_{R \in \cK_\pi} \EE[Q_R(X_{1:k})]\bigg)M(G_\pi, \cD_\tdiag)
=:\limval_T(X_{1:k},\cD_\tdiag)
\end{align*}
This limit value depends only on $T$, the joint law of $(X_1,\ldots,X_k)$, and
the limit diagonal distribution $\cD_\tdiag$, and does not depend on the specific 
matrix $\bM$, concluding the proof.
\end{proof}

\begin{lemma}\label{lemma:Hlemma}
Let $\Mb \in \R^{n \times n}$ be a deterministic symmetric matrix with 
limit diagonal distribution $\cD_\tdiag$, satisfying the following
condition: For any fixed $\eps>0$, any diagonal monomial
$p(\x) \in \Delta\langle \x \rangle$, and all large $n$,
\[\max_{i\neq j} |p(\bM)[i,j]|<n^{-1/2+\eps}.\]
Let $G=(\cK,\cF)$ be a connected multi-graph such that the external degree
$\degext(R)$ is even for every vertex $R\in\cK$.
For every edge $e\in\cF$, let $p_e(\x)$ be a diagonal monomial labeling this edge.
Then there exists a value $M(G,\cD_\tdiag)$ depending only on
$G$ and $\cD_\tdiag$ such that
\[\lim_{n \to \infty} \frac{1}{n}
\sum_{\bj\in[n]^{\cK}}^* \prod_{e=(R,R')\in\cF} p_e(\bM)[j_R, j_{R'}]
= M(G,\cD_\tdiag).\]
\end{lemma}

\begin{proof}
For convenience, we denote
\begin{align*}
	M_n(G) := \frac{1}{n}\sum_{\bj\in[n]^{\cK}}^* \prod_{e=(R,R')\in\cF} 
	p_e(\bM)[j_R, j_{R'}].
\end{align*}
We proceed by induction over the number of vertices $|\cK|$.
For the base case $|\cK|=1$, all edges of $\cF$ must be self-loops, and we have
\begin{align*}
	M_n(G) 
	= \frac{1}{n} \sum_{j=1}^n \prod_{e\in\cF} p_e(\bM)[j,j]
	= \frac{1}{n} \Tr \prod_{e\in\cF}\Delta(p_e(\bM))
\end{align*}
Here $\prod_{e\in\cF} \Delta(p_e(\x))$ is a diagonal monomial.
Then, since $\bM$ has a limit diagonal distribution, the above quantity 
admits a limit value as $n\to\infty$.

Next, supposing that the result is true for every multi-graph $G=(\cK,\cF)$
with $|\cK|\leq K$, we prove the result for $|\cK| = K+1$.
Define $\cK_* := \{R\in\cK: \degext(R)=2\}$. 

First, consider the case where $|\cK_*|=0$. Then
\begin{itemize}
\item Since every $\degext(R)$ is even, we must have $\degext(R) \geq 4$
for all $R\in\cK$. Therefore, denoting by $\cF_{\ext} \subseteq \cF$ those edges
that are not self-loops, we have $4|\cK| \leq 2|\cF_{\ext}|$.
\item We may assume
without loss of generality that each vertex $R \in \cK$ has
exactly one self-loop: For $R$ without a self-loop, we may add the self-loop
$e=(R,R)$ with the identity label $p_e(\bM)=\Id$. For $R$ with multiple
self-loops $\{e \in \cF:e=(R,R)\}$, we may replace these by a single self-loop
$e'=(R,R)$ having label $p_{e'}(\bM)=\prod_{e \in \cF:e=(R,R)}
\Delta(p_e(\bM))$.
These operations do not change the value of $M_n(G)$.
\end{itemize}
We denote by $e_R$ the unique self-loop on each vertex
$R \in \cK$. Then it follows that 
\begin{align*}
|M_n(G)| &\leq \frac{1}{n}\sum_{\bj\in[n]^{\cK}}^*
	\prod_{e=(R,R')\in\cF} |p_e(\bM)[j_R, j_{R'}]|\\
	&\leq \frac{1}{n} \sum_{\bj \in [n]^{\cK}}
\prod_{R \in \cK} |p_{e_R}(\bM)[j_R,j_R]|
\cdot \prod_{e\in\cF_{\ext}}
	\max_{i\neq j}  |p_e(\bM)[i,j]|\\
	&\leq \frac{1}{n} \cdot n^{(-1/2+\eps)|\cF_{\ext}|}
\cdot \sum_{\bj \in [n]^{\cK}} \prod_{R \in \cK} |p_{e_R}(\bM)[j_R,j_R]|\\
&\leq n^{-1+\eps|\cF_{\ext}|}
\prod_{R \in \cK} \bigg(\frac{1}{n}\sum_{j=1}^n
|p_{e_R}(\bM)[j,j]|\bigg).
\end{align*}
Here, the second inequality uses the constraint that indices $j_R,j_{R'}$ are
distinct if $R \neq R'$, 
the third inequality holds for any fixed $\eps>0$ and all large $n$
by the given assumption on $\bM$, and the last inequality applies
$n^{-|\cF_{\ext}|/2} \leq n^{-|\cK|}$ as follows from the above bound
$4|\cK| \leq 2|\cF_{\ext}|$. By Cauchy-Schwarz, we have
\[\Bigg(\frac{1}{n}\sum_{j=1}^n |p_{e_R}(\bM)[j,j]|\Bigg)^2
\leq \frac{1}{n}\sum_{j=1}^n p_{e_R}(\bM)[j,j]^2
=\frac{1}{n}\Tr \Delta(p_{e_R}(\bM))\Delta(p_{e_R}(\bM))\]
where $\Delta(p_{e_R}(\x))\Delta(p_{e_R}(\x))$ is a diagonal monomial. Then
this quantity has a limit value as $n \to \infty$, for each $R \in \cK$.
Choosing $\eps<1/|\cF_{\ext}|$, we conclude that $M_n(G) \to 0$.

Next, consider the case where $|\cK_*|>0$. We pick an arbitrary vertex
$R_*\in\cK_*$, 
and let $R_1,R_2\in\cK$ be its two neighbors (where $R_1 \neq R_*$ and $R_2 \neq
R_*$, but possibly $R_1=R_2$). Denote $e_1=(R_*,R_1)$, $e_2=(R_*,R_2)$, and
assume without loss of generality as above that $R_*$ has a unique self-loop
$e_*=(R_*,R_*)$. Then
\begin{align*}
	M_n(G) &= 
	\frac{1}{n}\sum_{\bj\in[n]^{\cK\setminus R_*}}^* \prod_{e=(R,R')\in\cF\setminus\{e_1,e_2\}} 
	p_e(\bM)[j_R,j_{R'}]\\
	&\qquad \times \mathop{\sum_{j_{R_*}=1}^n}_{j_{R_*} \notin \{j_S:S \in \cK \setminus
R_*\}} p_{e_1}(\bM)[j_{R_*},j_{R_1}] 
	p_{e_2}(\bM)[j_{R_*}, j_{R_2}] p_{e_*}(\bM)[j_{R_*},j_{R_*}]\\
	&:=\mathrm{I}-\sum_{S \in \cK \setminus R_*} \mathrm{II}(S)
\end{align*}
where we set
\begin{align*}
\mathrm{I}&=\frac{1}{n}\sum_{\bj\in[n]^{\cK\setminus R_*}}^*
\left(\prod_{e=(R,R')\in\cF\setminus\{e_1,e_2\}} p_e(\bM)[j_R,j_{R'}]\right)
\cdot \big(p_{e_1}\Delta(p_{e_*}) p_{e_2}\big)(\bM)
	[j_{R_1}, j_{R_2}],\\
\mathrm{II}(S)&=\frac{1}{n}\sum_{\bj\in[n]^{\cK\setminus R_*}}^*
\left(\prod_{e=(R,R')\in\cF\setminus\{e_1,e_2\}} p_e(\bM)[j_R,j_{R'}]\right)\\
&\hspace{1.5in}
\times p_{e_1}(\bM)[j_S,j_{R_1}]p_{e_2}(\bM)[j_S,j_{R_2}]p_{e_*}(\bM)[j_S,j_S].\notag
\end{align*}
Here $\mathrm{I}$ corresponds to the full summation over $j_{R_*} \in [n]$
without restriction, and each term $-\mathrm{II}(S)$ removes the contribution from
the case $j_{R_*}=j_S$.

The term $\mathrm{I}$ is exactly equal to $M_n(G')$ for a 
multi-graph $G'$ obtained from $G$ by removing vertex $R_*$ and the edges
$e_1,e_2$, adding a new edge between $R_1$ and $R_2$ with label
$p_{e_1}(\x) \Delta(p_{e_*}(\x)) p_{e_2}(\x)$. This graph
$G'$ is connected and has one fewer vertex than $G$. Each remaining
vertex in $G'$ has the same external degree as in $G$ if
$R_1 \neq R_2$, and if $R_1=R_2$ then the external degree of $R_1=R_2$ is
reduced by 2. In both cases, all external degrees in $G'$ remain
even. Then applying the inductive hypothesis to $G'$,
$\lim_{n \to \infty} \mathrm{I}$ exists and depends only on
$(G',\cD_\tdiag)$.

Each term $\mathrm{II}(S)$ is exactly equal to $M_n(G')$ for a multi-graph $G'$
that merges the vertices $S$ and $R_*$ of $G$ into a single vertex $S_*$ in $G'$,
and preserves all edges and their labels. The new vertex $S_*$
in $G'$ has external degree
equal to $\degext(S)+\degext(R_*)-2|\{e \in \cF:e=(S,R_*)\}|$, which is even.
It is clear that
$G'$ remains connected, and the external degrees of all other
vertices of $G'$ remain the same as in $G$. Then applying the inductive
hypothesis to $G'$, also $\lim_{n \to \infty} \mathrm{II}(S)$
exists and depends only on $(G',\cD_\tdiag)$, completing the induction.
\end{proof}

\begin{remark}
In the language of \cite{male2020traffic}, our proof of Lemma
\ref{lemma:Esyminvmoments} shows that if $\W$ is invariant in law
under conjugation by
permutations, then the expected tensor network value has a limit if
$\W$ converges in traffic distribution, and this value is universal
across matrices having the same limiting traffic distribution. Our arguments of
Lemmas \ref{lemma:Esyminvmoments} and \ref{lemma:Hlemma} further establish that
if $\W$ is also invariant under conjugation by random signs and satisfies the 
additional delocalization conditions of Definition \ref{def:syminvariant},
then it has a limit traffic distribution that is uniquely determined by
its limit diagonal law.
\end{remark}

We provide in Appendix~\ref{appendix:orthogonal} an alternative computation of 
$\limval_T(X_1,\ldots,X_k,\cD_\tdiag)$ for the special case where
$\bW$ is orthogonally invariant in law,
using the orthogonal Weingarten calculus \cite{collins2006integration}. We
establish the asymptotic freeness statement of Proposition \ref{prop:freeness}(b)
also in Appendix~\ref{appendix:orthogonal} via this computation.

Finally, we conclude the proof of Lemma~\ref{lemma:syminvmoments} by showing
concentration of the tensor network value.

\begin{lemma}\label{lemma:syminvfourthmoment}
Let $\E$ be the expectation over $\bPi$ conditional on $\bM$
and $\x_1,\ldots,\x_k$. Under the setting of Lemma \ref{lemma:syminvmoments},
almost surely as $n \to \infty$,
\[\val_T(\W;\x_1,\ldots,\x_k)-\E[\val_T(\W;\x_1,\ldots,\x_k)] \to 0.\]
\end{lemma}

\begin{proof}
Let us write as shorthand $\val(\W)=\val_T(\W;\x_{1:k})$.
By Jensen's inequality,
\[\E[(\val(\W)-\E \val(\W))^4] \leq \E[(\val(\W)-\val(\bar{\W}))^4]\]
where $\bar{\W}=\bar{\bPi}\bM\bar{\bPi}^\top$, $\bar{\bPi}$ is an independent
copy of $\bPi$, and the expectation on the right side is over $(\bPi,\bar{\bPi})$.
We proceed to bound this expectation.

Recall the tensor network $T=(\cV,\cE,\{q_v\}_{v \in \cV})$.
Let $(\cV^{(1)},\cE^{(1)}),\ldots,(\cV^{(4)},\cE^{(4)})$ denote four
copies of the tree $(\cV,\cE)$. For any subset $A \subseteq \{1,2,3,4\}$, let
\[(\cV_A,\cE_A) \cong \bigsqcup_{a \in A} (\cV^{(a)},\cE^{(a)})\]
denote the graph that is the disjoint union of those copies corresponding to
$a \in A$, i.e.\ $(\cV_A,\cE_A)$ has $|A|$ connected components, each a copy of
$(\cV,\cE)$. We label each vertex $v \in \cV^{(a)} \subseteq \cV_A$ with the same
label $q_v$ as in the original tensor network $T$.
We write $\bar{A}=\{1,2,3,4\} \setminus
A$ as the complement of $A$, and $\bar{\bXi}$ and $\bar\sigma$ for the random
sign matrix and random permutation corresponding to $\bar\bPi$.
Then we have, similarly to (\ref{eq:invariantval}),
\begin{align}
&(\val(\W)-\val(\bar\W))^4\notag\\
&=\sum_{A \subseteq \{1,2,3,4\}} (-1)^{|A|} \prod_{a \in A} \val(\W)
\prod_{a \in \bar{A}} \val(\bar\W)\notag\\
&=\sum_{A \subseteq \{1,2,3,4\}} (-1)^{|A|}\frac{1}{n^4}
\sum_{\bi \in [n]^{\cV_A}}
\prod_{v \in \cV_A} q_v(x_{1:k}[i_v]) \prod_{(u,v) \in \cE_A}
\Xi[i_u] \cdot \Xi[i_v] \cdot M[\sigma(i_u),\sigma(i_v)]\notag\\
&\hspace{1in} \times \sum_{\bj \in [n]^{\cV_{\bar A}}}
\prod_{v \in \cV_{\bar A}} q_v(x_{1:k}[j_v]) \prod_{(u,v) \in \cE_{\bar A}}
\bar \Xi[j_u] \cdot \bar \Xi[j_v] \cdot M[\bar
\sigma(j_u),\label{eq:invariantvalfourth}
\bar \sigma(j_v)].
\end{align}

Let $\cP_A$ be the set of partitions of $\cV_A$, and denote by
$\pi(\bi) \in \cP_A$ the partition induced by $\bi \in [n]^{\cV_A}$. For each
$\pi \in \cP_A$,
let $G_\pi=(\cK_\pi,\cF_\pi)$ be the image of $(\cV_A,\cE_A)$ under $\pi$, in
the sense of Definition \ref{def:graphpartition}. Note that here,
$G_\pi$ is not necessarily connected but
can consist of up to $|A| \leq 4$ connected components.
Define $B_n(\pi)$ and $Q_n(\pi)$ exactly as in (\ref{eq:An}--\ref{eq:Qn}),
let $\cC(\pi)$ denote the set of connected components of $G_\pi$, and define
\begin{equation}\label{eq:Mngeneral}
M_n(\pi)=\frac{1}{n^{|\cC(\pi)|}} \sum_{\bj \in [n]^{\cK_\pi}}^*
\prod_{(R,S) \in \cF_\pi} M[j_R,j_S].
\end{equation}
This coincides with our previous definition of (\ref{eq:Mn}) when $\cC(\pi)=1$.
Define similarly $B_n(\bar\pi),Q_n(\bar\pi),M_n(\bar\pi)$ via the graph
$G_{\bar \pi}=(\cK_{\bar \pi},\cF_{\bar\pi})$ that is the image of $(\cV_{\bar
A},\cE_{\bar A})$ under $\bar\pi \in \cP_{\bar A}$.
Then, stratifying the sums over $\bi$ and $\bj$
by $\pi(\bi) \in \cP_A$ and $\pi(\bj) \in \cP_{\bar A}$,
and taking the expectation in (\ref{eq:invariantvalfourth})
over $(\bPi,\bar\bPi)$ using 
(\ref{eq:sign}--\ref{eq:permutation}), we get analogously to
(\ref{eq:orthoval3})
\begin{align}\label{eq:orthoEvalfourth}
&\E[(\val(\W)-\val(\bar\W))^4]\notag\\
&\qquad=\sum_{A \subseteq \{1,2,3,4\}} (-1)^{|A|}
\mathop{\sum_{\text{even } \pi \in \cP_A}}_{\text{even } \bar\pi \in \cP_{\bar A}}
\frac{n^{|\cC(\pi)|+|\cC(\bar\pi)|}}{n^4}
B_n(\pi)B_n(\bar\pi) \cdot Q_n(\pi)Q_n(\bar\pi) \cdot M_n(\pi)M_n(\bar\pi).
\end{align}

For $\pi \in \cP_A$ and $\bar\pi \in \cP_{\bar A}$,
we define $\tau=\pi \oplus \bar\pi \in \cP_{\{1,2,3,4\}}$ as the combined
partition of all vertices in $\cV_{\{1,2,3,4\}}$ given by taking the blocks of
both $\pi$ and $\bar\pi$. We write $G_\tau=(\cK_\tau,\cF_\tau)$ as the image of
$(\cV_{\{1,2,3,4\}},\cE_{\{1,2,3,4\}})$ under $\tau$;
this is the disjoint union of $G_\pi$ and $G_{\bar\pi}$, so in particular
\[|\cK_\tau|=|\cK_\pi|+|\cK_{\bar\pi}|,
\qquad |\cC(\tau)|=|\cC(\pi)|+|\cC(\bar\pi)|.\]
We now proceed to approximate $B_n(\pi)B_n(\bar\pi)$,
$Q_n(\pi)Q_n(\bar\pi)$, and $M_n(\pi)M_n(\bar\pi)$ by quantities that depend
only on $\tau$, and not on the individual partitions $\pi,\bar\pi$. We write
$O(n^{-\nu})$ for any error of magnitude at most $C/n^{\nu}$ for a constant
$C:=C(\pi,\bar\pi)>0$ and all large $n$.

For $B_n$, observe from the definition (\ref{eq:An}) that
\begin{align*}
B_n(\tau)&=\frac{n}{n} \cdot \frac{n}{n-1} \cdot \frac{n}{n-2}
\cdot \ldots \cdot \frac{n}{n-|\cK_\tau|+1}\\
&=1+\frac{\sum_{k=0}^{|\cK_\tau|-1} k}{n}+O(n^{-2})
=1+n^{-1}\binom{|\cK_\tau|}{2}+O(n^{-2}).
\end{align*}
Similarly,
\begin{align*}
B_n(\pi)B_n(\bar\pi)&=
1+n^{-1}\left(\binom{|\cK_\pi|}{2}+\binom{|\cK_{\bar\pi}|}{2}\right)+O(n^{-2}).
\end{align*}
In particular,
\begin{equation}\label{eq:Anapprox1}
B_n(\pi)B_n(\bar\pi)=B_n(\tau)+O(n^{-1})=1+O(n^{-1}).
\end{equation}
In the case where $G_\tau=(\cK_\tau,\cF_\tau)$ has 4 connected components,
i.e.\ each block of both $\pi$ and $\bar\pi$ is contained
within a single copy $\cV^{(a)}$ of $\cV$, let us write $G_\tau(a)=(\cK_\tau(a),
\cF_\tau(a))$ for the component corresponding to the partition of $\cV^{(a)}$.
Given any $\pi \in \cP_A$, $\bar\pi \in \cP_{\bar A}$, and
$\tau=\pi \oplus \bar\pi$, we then have
\[\binom{|\cK_\tau|}{2}
=\binom{|\cK_\pi|}{2}+\binom{|\cK_{\bar\pi}|}{2}
+\sum_{a \in A,b \notin A} |\cK_\tau(a)| \cdot |\cK_\tau(b)|\]
because to choose two elements of $\cK_\tau$, we may choose them both from
$\cK_\pi$, both from $\cK_{\bar\pi}$, or one from $\cK_\tau(a) \subseteq \cK_\pi$
and the other from $\cK_\tau(b) \subseteq \cK_{\bar \pi}$ for some $a \in A,b
\in \bar A$. This gives a refinement of (\ref{eq:Anapprox1}),
\begin{equation}\label{eq:Anapprox2}
B_n(\pi)B_n(\bar\pi)=B_n(\tau)+\sum_{a \in A,b \notin A}
B_n(\tau,a,b)+O(n^{-2})
\end{equation}
where we define $B_n(\tau,a,b)=-n^{-1}|\cK_\tau(a)| \cdot |\cK_\tau(b)|$.
Here $B_n(\tau,a,b)=O(n^{-1})$.

For $Q_n$, observe from the definition (\ref{eq:Qn}) that
the distinction between $Q_n(\pi)Q_n(\bar\pi)$ and $Q_n(\tau)$ is that the
former does not restrict indices of summation corresponding to $\pi$ to be
distinct from those corresponding to $\bar\pi$, i.e.
\begin{align*}
Q_n(\pi)Q_n(\bar\pi)&=Q_n(\tau)
+\frac{1}{n^{|\cK_\tau|}} \sum_{\bi \in [n]^{\cK_\pi}}^*
\sum_{\bj \in [n]^{\cK_{\bar\pi}}}^* \1\{\text{there is at least 1
pair of coinciding}\\
&\hspace{0.5in} \text{ indices between } \bi \text{ and } \bj\}
\times \prod_{R \in \cK_\pi} Q_R(x_{1:k}[i_R])
\prod_{R' \in \cK_{\bar\pi}} Q_{R'}(x_{1:k}[j_{R'}]).
\end{align*}
By Lemma \ref{lemma:distinctness}, the quantities $Q_n(\pi)$, $Q_n(\bar\pi)$,
and $Q_n(\tau)$ all have deterministic limit values as $n \to \infty$.
Furthermore, by a simple inclusion-exclusion argument together with
Lemma \ref{lemma:distinctness}, in the above double summation
the contribution from pairs $(\bi,\bj)$ coinciding in exactly $k$ pairs of
indices is of size $O(n^{-k})$. Then in particular,
\begin{equation}\label{eq:Qnapprox1}
Q_n(\pi)Q_n(\bar\pi)=Q_n(\tau)+O(n^{-1})=O(1).
\end{equation}
In the case where $G_\tau$ has 4 connected components, let us write more
explicitly
\begin{align*}
Q_n(\pi)Q_n(\bar\pi)&=Q_n(\tau)
+\frac{1}{n^{|\cK_\tau|}} \sum_{\bi \in [n]^{\cK_\pi}}^*
\sum_{\bj \in [n]^{\cK_{\bar\pi}}}^* \1\{\text{there is exactly 1 pair of
coinciding}\\
&\text{ indices between } \bi \text{ and } \bj\}
\times \prod_{R \in \cK_\pi} Q_R(x_{1:k}[i_R])
\prod_{R' \in \cK_{\bar\pi}} Q_{R'}(x_{1:k}[j_{R'}]) + O(n^{-2}).
\end{align*}
We may choose the coinciding index pair by choosing 1 vertex
$R \in \cK_\tau(a) \subseteq \cK_\pi$ for some $a \in A$,
and 1 vertex $R' \in \cK_\tau(b) \subseteq \cK_{\bar \pi}$
for some $b \in \bar A$. Now viewing $R \in \pi$ and $R' \in \bar\pi$ as disjoint
blocks of vertices of $\cV$, note that if $S=R \cup R' \subseteq \cV$
is the block obtained upon merging $R,R'$,
then by definition $Q_S=Q_R \cdot Q_{R'}$. Thus, the above is equivalent to
\begin{equation}\label{eq:Qnapprox2}
Q_n(\pi)Q_n(\bar\pi)=Q_n(\tau)+\sum_{a \in A,b \notin A} Q_n(\tau,a,b)+O(n^{-2})
\end{equation}
where we define
\begin{align*}
Q_n(\tau,a,b)
&=\frac{1}{n^{|\cK_\tau|}} \sum_{\bj \in [n]^{\cK_\tau}}
\1\{\bj \text{ has } |\cK_\tau|-1 \text{ distinct indices, and 1 index from }\\
&\hspace{1in} \cK_\tau(a) \text{ coincides with 1 index
from } \cK_\tau(b)\} \times \prod_{R \in \cK_\tau} Q_R(x_{1:k}[j_R]).
\end{align*}
By the preceding arguments, $Q_n(\tau,a,b)=O(n^{-1})$.

For $M_n$, consider any $A \subseteq \{1,2,3,4\}$ and
$\pi \in \cP_A$. Recall that $\cC(\pi)$ is the set of
connected components of $G_\pi$. Each component in
$\cC(\pi)$ takes the form $G_\sigma=(\cK_\sigma,\cF_\sigma)$ where $\sigma$ is a
partition that contains a subset of the blocks of $\pi$. Let us write
$\sum_{\bj \in [n]^{\cK_\pi}}^{**}$
for the summation over tuples $\bj$ such that indices corresponding to
each component $\cK_\sigma \subseteq \cK_\pi$ are distinct,
but they are not necessarily distinct
across different components. Recalling (\ref{eq:Mngeneral}), define
\[M_n^{**}(\pi):=\prod_{G_\sigma \in \cC(\pi)} M_n(\sigma)
=\frac{1}{n^{|\cC(\pi)|}} \sum_{\bj \in [n]^{\cK_\pi}}^{**} \prod_{(R,S) \in
\cF_\pi} M[j_R,j_S]\]
where $M_n^{**}(\pi)$ is now a multiplicative function over connected components
of $\pi$. Since each $G_\sigma$ is connected, Lemma \ref{lemma:Hlemma}(a) implies
the existence of the limit
\begin{equation}\label{eq:tildeMnlimit}
M_n^{**}(\pi) \to \prod_{G_\sigma \in \cC(\pi)} M(G_\sigma,\cD_{\tdiag}).
\end{equation}
Comparing the definitions of $M_n^{**}(\pi)$ and $M_n(\pi)$, we have
$M_n^{**}(\pi)=M_n(\pi)$ if $G_\pi$ has a single connected component
$|\cC(\pi)|=1$, and more generally
\begin{align*}
M_n(\pi)&=M_n^{**}(\pi)-\frac{1}{n^{|\cC(\pi)|}}
\sum_{\bj \in [n]^{\cK_\pi}}^{**}
\1\{\text{some indices of } \bj \text{ for different connected }\\
&\hspace{1in} \text{ components of } G_\pi \text{ coincide}\}
\times \prod_{(R,S) \in \cF_\pi} M[j_R,j_S].
\end{align*}
For each $\bj$ where this summand is non-zero, define
$\pi'(\bj) \in \cP_A$ as the partition that merges those blocks of $\pi$
where the corresponding indices of $\bj$ coincide.
Let $\cP(\pi)$ be the set of all possible such partitions $\pi'(\bj)$.
(If $|\cC(\pi)|=1$, then $\cP(\pi)=\emptyset$.)
Then, stratifying the summation over $\bj$ by $\pi'(\bj) \in \cP(\pi)$,
letting $G_{\pi'}=(\cK_{\pi'},\cF_{\pi'})$ be the image of $(\cV_A,\cE_A)$
under $\pi'$, and identifying the sum over $\{\bj:\pi'(\bj)=\pi'\}$ as a sum over
one distinct index for each $R \in \cK_{\pi'}$,
\begin{align}
M_n(\pi)&=M_n^{**}(\pi)-\frac{1}{n^{|\cC(\pi)|}} \sum_{\pi' \in \cP(\pi)}
\sum_{\bj \in [n]^{\cK_{\pi'}}}^* \prod_{(R,S) \in \cF_{\pi'}} M[j_R,j_S]\notag\\
&=M_n^{**}(\pi)-\sum_{\pi' \in \cP(\pi)} \frac{1}{n^{|\cC(\pi)|-|\cC(\pi')|}}
M_n(\pi').\label{eq:Mnrelation}
\end{align}
For any $\pi' \in \cP(\pi)$, its number of connected components
satisfies $|\cC(\pi')| \leq |\cC(\pi)|-1$. In particular, if $|\cC(\pi)|=2$,
then $|\cC(\pi')|=1$ for all $\pi' \in \cP(\pi)$, so $M_n(\pi')=M_n^{**}(\pi')$
on the right side of (\ref{eq:Mnrelation}).
If $|\cC(\pi)| \geq 3$, then we may apply this identity (\ref{eq:Mnrelation})
recursively to further approximate $M_n(\pi')$ on the right
side of (\ref{eq:Mnrelation}) by $M_n^{**}(\pi')$, until only instances of
$M_n^{**}$ and no instances of $M_n$ remain.
Applying (\ref{eq:tildeMnlimit}) to each instance of $M_n^{**}$ in
this final expression, this shows that
\[M_n(\pi)=M_n^{**}(\pi)+O(n^{-1})=O(1).\]
Applying this for $\pi \in \cP_A$, $\bar\pi \in \cP_{\bar A}$, and $\tau=\pi
\oplus \bar\pi$, and recalling that
$M_n^{**}$ is multiplicative across connected components so that
$M_n^{**}(\tau)=M_n^{**}(\pi)M_n^{**}(\bar\pi)$, this yields
\begin{equation}\label{eq:Mnapprox1}
M_n(\pi)M_n(\bar\pi)=M_n(\tau)+O(n^{-1})=O(1).
\end{equation}

When $G_\tau$ has 4 connected components, let us derive a more explicit
expression for this $O(n^{-1})$ error. Applying (\ref{eq:Mnrelation}) and the
above arguments to $\tau$, we have
\[M_n(\tau)=M_n^{**}(\tau)
-\frac{1}{n}\sum_{\tau' \in \cP(\tau):|\cC(\tau')|=|\cC(\tau)|-1}
M_n^{**}(\tau')+O(n^{-2}).\]
If $\tau' \in \cP(\tau)$ and $|\cC(\tau')|=|\cC(\tau)|-1$, then $\tau'$ is
obtained by picking exactly two connected components of $G_\tau$, say
$G_\tau(a)=(\cK_\tau(a),\cF_\tau(a))$ and $G_\tau(b)=(\cK_\tau(b),\cF_\tau(b))$,
and merging one or more pairs of blocks $R \in \cK_\tau(a)$ with $R' \in
\cK_\tau(b)$. We write the set of such partitions $\tau' \in \cP(\tau)$
corresponding to the two fixed indices $a \neq b$ as $\cP(\tau,a,b)$. Then
\[M_n(\tau)=M_n^{**}(\tau)
-\frac{1}{n}\sum_{1 \leq a<b \leq 4} \sum_{\tau' \in \cP(\tau,a,b)}
M_n^{**}(\tau')+O(n^{-2}).\]
Similarly
\begin{align*}
M_n(\pi)&=M_n^{**}(\pi)
-\frac{1}{n}\mathop{\sum_{a<b}}_{a,b \in A} \sum_{\pi' \in \cP(\pi,a,b)}
M_n^{**}(\pi')+O(n^{-2}),\\
M_n(\bar\pi)&=M_n^{**}(\bar\pi)
-\frac{1}{n}\mathop{\sum_{a<b}}_{a,b \in \bar A} \sum_{\bar\pi' \in \cP(\bar\pi,a,b)}
M_n^{**}(\bar\pi')+O(n^{-2}).
\end{align*}
Taking the product of these two expressions and applying multiplicativity of
$M_n^{**}$, we deduce
\begin{equation}\label{eq:Mnapprox2}
M_n(\pi)M_n(\bar\pi)=M_n(\tau)+\sum_{a \in A,b \in \bar A}
M_n(\tau,a,b)+O(n^{-2})
\end{equation}
where we define
$M_n(\tau,a,b)=\frac{1}{n}\sum_{\tau' \in \cP(\tau,a,b)} M_n^{**}(\tau')$.
Here again, $M_n(\tau,a,b)=O(n^{-1})$.

Equipped with these approximations, we now bound (\ref{eq:orthoEvalfourth}).
Given $\tau \in \cP_{\{1,2,3,4\}}$, let $\cA(\tau)$ be the set of $A \subseteq
\{1,2,3,4\}$ for which $\tau=\pi \oplus \bar\pi$ for some $\pi \in \cP_A$ and
$\bar\pi \in \cP_{\bar A}$, i.e.\ $A \in \cA(\tau)$ if and only if each
connected component of $G_\tau$ corresponds to vertices belonging entirely to
$\cV_A$ or entirely to $\cV_{\bar A}$. Note that given
$\tau=\pi \oplus \bar\pi$ and $A \in \cA(\tau)$, this uniquely determines
$\pi \in \cP_A$ and $\bar\pi \in \cP_{\bar A}$. Then, stratifying the summation 
in (\ref{eq:orthoEvalfourth}) by the number of connected components
$|\cC(\tau)|=|\cC(\pi)|+|\cC(\bar\pi)|$, we have
\[\E[(\val(\W)-\val(\bar \W))^4]=E_1+E_2+E_3+E_4\]
where
\begin{align*}
E_j&=\frac{1}{n^{4-j}}\mathop{\sum_{\text{even } \tau \in
\cP_{\{1,2,3,4\}}}}_{|\cC(\tau)|=j}
\sum_{A \in \cA(\tau)} (-1)^{|A|} B_n(\pi)B_n(\bar\pi)
\cdot Q_n(\pi)Q_n(\bar\pi) \cdot M_n(\pi)M_n(\bar\pi).
\end{align*}
Here, $\pi$ and $\bar\pi$ on the right side denote those partitions that are
uniquely determined by $\tau=\pi \oplus \bar\pi$ and $A \in \cA(\tau)$;
we omit their dependence on $(\tau,A)$ for brevity.

Applying the simple the bound
$B_n(\pi)B_n(\bar \pi)Q_n(\pi)Q_n(\bar \pi)M_n(\pi)M_n(\bar \pi)=O(1)$
from (\ref{eq:Anapprox1}), (\ref{eq:Qnapprox1}), and (\ref{eq:Mnapprox1}), we
get $E_1=O(n^{-3})$ and $E_2=O(n^{-2})$.

For $E_3$, applying the approximation
$B_n(\pi)B_n(\bar \pi)=B_n(\tau)+O(n^{-1})$,
$Q_n(\pi)Q_n(\bar \pi)=Q_n(\tau)+O(n^{-1})$, and
$M_n(\pi)M_n(\bar \pi)=M_n(\tau)+O(n^{-1})$ 
from (\ref{eq:Anapprox1}), (\ref{eq:Qnapprox1}), and (\ref{eq:Mnapprox1}),
we have
\[E_3=\frac{1}{n}\sum_{\text{even } \tau \in
\cP_{\{1,2,3,4\}}:|\cC(\tau)|=3} B_n(\tau)Q_n(\tau)M_n(\tau)
\sum_{A \in \cA(\tau)} (-1)^{|A|}+O(n^{-2}).\]
Importantly, the leading term $B_n(\tau)Q_n(\tau)M_n(\tau)$ does not depend on
$A$, so we have factored it outside of the sum over $A$, and the lower order
terms all contribute to the $O(n^{-2})$ error.
For any $\tau$ where $|\cC(\tau)|=3$, we have $\sum_{A \in
\cA(\tau)} (-1)^{|A|}=0$: For example, if the 3 connected components of
$G_\tau$ correspond to vertices in $\cV_1$, $\cV_2$, and $\cV_{\{3,4\}}$, then
$\cA(\tau)=\{\emptyset,\{1\},\{2\},\{1,2\},\{3,4\},\{1,3,4\},\{2,3,4\},
\{1,2,3,4\}\}$. Thus, we get $E_3=O(n^{-2})$.

Finally, for $E_4$, we apply the finer approximations
(\ref{eq:Anapprox2}), (\ref{eq:Qnapprox2}), and (\ref{eq:Mnapprox2}) which hold
when $|\cC(\tau)|=4$. In this case $\cA(\tau)$ consists of all subsets of
$\{1,2,3,4\}$, so
\begin{align*}
E_4&=\sum_{\text{even } \tau \in
\cP_{\{1,2,3,4\}}:|\cC(\tau)|=4} \Bigg(B_n(\tau)Q_n(\tau)M_n(\tau)
\sum_{A \subseteq \{1,2,3,4\}} (-1)^{|A|}\\
&\qquad+\sum_{a \neq b \in \{1,2,3,4\}} \Big[B_n(\tau)Q_n(\tau)M_n(\tau,a,b)
+B_n(\tau)Q_n(\tau,a,b)M_n(\tau)\\
&\hspace{1in}+B_n(\tau,a,b)Q_n(\tau)M_n(\tau)\Big]
\mathop{\sum_{A \subseteq \{1,2,3,4\}}}_{a \in A,b \in \bar{A}}
(-1)^{|A|}\Bigg)+O(n^{-2}).
\end{align*}
Importantly, we have exchanged the order of summations over $A$ and
over $(a \in A,b \in \bar A)$, and used that each term $B_n(\tau,a,b),Q_n(\tau,a,b),M_n(\tau,a,b)$
does not depend on the assignment of
the remaining indices $\{1,2,3,4\} \setminus \{a,b\}$ to $A$ and $\bar A$.
Then, applying $\sum_{A \subseteq \{1,2,3,4\}} (-1)^{|A|}=0$
and also $\sum_{A \subseteq \{1,2,3,4\}:a \in A,b \in \bar{A}} (-1)^{|A|} =0$ for each fixed
pair $a,b$, we get $E_4=O(n^{-2})$.

Combining the above, we have $\E[(\val(\W)-\E \val(\W))^4] \leq C/n^2$
for a constant $C>0$ and all large $n$.
Then Lemma \ref{lemma:syminvfourthmoment} follows from
Markov's inequality and the Borel-Cantelli lemma.
\end{proof}

Combining Lemmas~\ref{lemma:Esyminvmoments} and \ref{lemma:syminvfourthmoment}
completes the proof of Lemma~\ref{lemma:syminvmoments}.

\subsection{Universality of AMP via polynomial approximation}
We now prove Lemma \ref{lemma:Wignercompare},
showing that the universality of AMP for Lipschitz non-linearities\footnote{By
this we mean that each non-linearity $u_{t+1}(\cdot)$ is Lipschitz 
in its first $t$ arguments $z_1,\ldots,z_t$.} can be obtained from universality of tensor network values by polynomial approximation.

For the given AMP algorithm with Lipschitz non-linearities $u_{t+1}(\cdot)$,
we approximate it by an auxiliary AMP algorithm with polynomial 
non-linearities $\tilde{u}_{t+1}(\cdot)$, where each $\tilde{u}_{t+1}$ is an 
$L_2$-approximation for $u_{t+1}$ with respect to the state evolution of its 
arguments.
A similar method of approximation was recently
used in \cite{dudeja2022universality}.
Combining this approximation, the validity of state evolution for
polynomial AMP applied to $\G$, and
the universality of tensor network values for $\G$ and $\W$,
we show that iterates of the Lipschitz and polynomial AMP algorithms applied to
$\W$ are close in (normalized) $\ell_2$ distance. This will imply the desired 
$W_2$-convergence of the AMP iterates \eqref{eq:AMP} to their state evolution.

We construct the auxiliary AMP algorithm as follows: Fix any $\eps>0$.
For the same initialization $\tilde \ub_1 = \ub_1$ and vectors of side 
information $\mathbf{f}_1, \ldots, \mathbf{f}_k$ as in the given Lipschitz AMP 
algorithm \eqref{eq:AMP}, define the iterates for $t=1,2,3,\ldots$
\begin{align}
\begin{aligned}\label{eq:aux_AMP_sym}
\tilde\zb_t &= \bW \tilde \ub_t - \sum_{s=1}^t \tilde b_{ts} \tilde \ub_s\\
\tilde\ub_{t+1} &= \tilde{u}_{t+1}(\tilde\zb_1, \ldots, \tilde\zb_t, 
\fb_1, \ldots, \fb_k)
\end{aligned}
\end{align}
such that
\begin{enumerate}
\item Each coefficient $\tilde{b}_{ts}$ is defined by $\tilde u_2,\tilde
u_3,\ldots,\tilde u_t$
and the orthogonally invariant prescription (\ref{eq:symorthob}).
\item Let $\widetilde{\bSigma}_t$ be the orthogonally invariant
prescription (\ref{eq:symorthobSigma}), and let
$(U_1,F_{1:k},\widetilde{Z}_{1:t})$ be the state evolution where 
$\widetilde{Z}_{1:t} \sim \Normal(0,\widetilde{\bSigma}_t)$ is independent of
$(U_1,F_{1:k})$. 
Then each polynomial $\tilde u_{t+1}(\cdot)$ is chosen to satisfy
\begin{equation}\label{eq:polyapprox}
\E\big[\big(\tilde u_{t+1}(\widetilde{Z}_{1:t},F_{1:k})
-u_{t+1}(\widetilde{Z}_{1:t},F_{1:k})\big)^2\big]<\eps.
\end{equation}
\item For any fixed arguments $z_{1:(t-1)}$ and $f_{1:k}$, the function $z_t
\mapsto \tilde u_{t+1}(z_{1:t},f_{1:k})$ has non-linear dependence in $z_t$.
\end{enumerate}
We write the iterates as $\tilde\zb_t(\bW), \tilde\ub_t(\bW)$ if we want to 
make explicit that the algorithm is evaluated on the matrix $\bW$.

The choice of $\tilde u_{t+1}$ in condition (2) above is possible by the
polynomial density condition in Assumption~\ref{assump:ufconvergence}, and by 
Lemma~\ref{lemma:prod_determinance} which
ensures that the same density condition holds for 
$(U_1,F_{1:k},\widetilde{Z}_{1:t})$.
If condition (3) does not hold for this polynomial $\tilde u_{t+1}$, then it 
must hold upon adding to $\tilde u_{t+1}$ a small multiple of $z_t^2$. 
The conditions of \cite[Assumption 4.2]{fan2022approximate} are verified by 
Assumption~\ref{assump:ufconvergence}, the condition $\Var[D]>0$ given
in Lemma \ref{lemma:Wignercompare}, and the above condition (3). Then
\cite[Theorem 4.3]{fan2022approximate} ensures, almost surely as $n \to \infty$,
\begin{equation}\label{eq:Gstateevolution}
(\u_1,\f_1, \ldots, \fb_k,\tilde{\z}_1(\G), \ldots, \tilde\zb_t(\bG)) \toW 
(U_1,F_1, \ldots, F_k,\widetilde{Z}_1, \ldots, \widetilde Z_t).
\end{equation}

\begin{lemma}\label{lemma:polynomial_tensor_network_decomposition}
Fix any $t \geq 1$.
Let $(\tilde\ub_1, \tilde\zb_1, \ldots, \tilde\zb_t, 
\tilde\ub_t)$ be the iterates of any algorithm of the
form \eqref{eq:aux_AMP_sym}, where $\{\tilde{b}_{ts}\}$ are scalar constants and
$\tilde{u}_{t+1}:\R^{t+k} \to \R$ are polynomial functions applied row-wise.
Then for any polynomial $p: \RR^{2t+k} \to \RR$ and for some finite set $\cF$ of
diagonal tensor networks in $k+1$ variables,
\begin{align*}
\langle p(\tilde\ub_1, \ldots, \tilde\ub_t,\tilde\zb_1, \ldots, \tilde\zb_t,
\fb_1, \ldots, \fb_k) \rangle 
= \sum_{T \in \cF} \val_T(\bW; \ub_1, \fb_1, \ldots, \fb_k). 
\end{align*} 
\end{lemma}
\begin{proof}
First note that 
\begin{align*}
\langle p(\tilde\ub_{1:t}, \tilde\zb_{1:t},\fb_{1:k}) \rangle 
= \val_T(\bW; \tilde\ub_{1:t}, \tilde\zb_{1:t}, \fb_{1:k})
\end{align*}
where $T$ is a tensor network with only one vertex $v$ whose associated 
polynomial is $q_v = p$.

We claim that given any tensor network $T=(\cV,\cE,\{q_v\}_{v \in \cV})$
in the variables $(\tilde u_{1:t},\tilde z_{1:t},f_{1:k})$, we can decompose
\begin{align}\label{eq:sym_reduction_1}
\val_T(\bW; \tilde\ub_{1:t}, \tilde\zb_{1:t}, \fb_{1:k}) 
&= \sum_{T' \in \cF} \val_{T'}(\bW; \tilde\ub_{1:t}, \tilde\zb_{1:(t-1)}, \fb_{1:k})
\end{align}
where $\cF$ is a finite set of tensor networks in the variables
$(\tilde u_{1:t}, \tilde z_{1:(t-1)}, f_{1:k})$. 
To show this, recall that
\begin{align*}
\val_T(\bW; \tilde\ub_{1:t}, \tilde\zb_{1:t}, \fb_{1:k})
&= \frac{1}{n}\sum_{\bi \in [n]^\cV} \prod_{v \in \cV}
q_v(\tilde u_{1:t}[i_v], \tilde z_{1:t}[i_v], f_{1:k}[i_v]) W_{\bi|T}.
\end{align*}
Applying $\tilde\zb_t = \bW \tilde\ub_t 
- \sum_{s=1}^t \tilde b_{ts} \tilde\ub_s$ and
expanding each $q_v$ in terms of $(\tilde\ub_{1:t}, \tilde\zb_{1:(t-1)}, 
\fb_{1:k})$ and $\bW\tilde \ub_t$, we have 
\begin{align*}
q_v(\tilde u_{1:t}[i_v], \tilde z_{1:t}[i_v], f_{1:k}[i_v]) = 
\sum_{\theta=0}^{\Theta_v} q_{v,\theta}(\tilde u_{1:t}[i_v], 
\tilde z_{1:(t-1)}[i_v], f_{1:k}[i_v]) 
\cdot \bigg(\sum_{j=1}^n W[i_v, j] \tilde u_t[j]\bigg)^\theta
\end{align*}
where $\Theta_v$ is the maximum degree of $q_v$ in $\tilde z_t$, and $q_{v,0},
q_{v,1}, \ldots, q_{v, \Theta_v}$ are polynomials that depend on $q_v$ and 
$\{\tilde b_{ts}\}$.
Therefore
\begin{align*}
\val_T(\bW; \tilde\ub_{1:t}, \tilde\zb_{1:t}, \fb_{1:k}) 
&=\frac{1}{n} \sum_{\btheta \in \prod_{v \in \cV} \{0,\ldots,\Theta_v\}}
\sum_{\bi \in [n]^\cV} \prod_{v \in \cV} q_{v,\theta_v}
(\tilde u_{1:t}[i_v], \tilde z_{1:(t-1)}[i_v], f_{1:k}[i_v])\notag\\
&\hspace{1in}\cdot \bigg(\sum_{j=1}^n W[i_v, j] \tilde u_t[j]\bigg)^{\theta_v} 
\prod_{(u,v)\in \cE} W[i_u,i_v]
\end{align*}
For each $\btheta \in \prod_{v\in\cV} \{0,\ldots,\Theta_v\}$, we define a new 
tensor network $T_{\btheta}$ from $T$ as follows: (1) for each $v \in \cV$, 
replace the associated polynomial $q_v$ by $q_{v, \theta_v}$; (2) for each 
$v \in \cV$, connect $v$ with $\theta_v$ new vertices, where the associated 
polynomial for each new vertex is 
$q(\tilde u_{1:t}, \tilde z_{1:(t-1)}, f_{1:k}) = \tilde u_t$.
Then the above is precisely
\begin{align*}
\val_T(\bW; \tilde\ub_{1:t}, \tilde\zb_{1:t}, \fb_{1:k}) = 
\sum_{\btheta \in \prod_{v \in \cV}\{0,\ldots,\Theta_v\}}
\val_{T_{\btheta}}(\bW; \tilde\ub_{1:t}, \tilde\zb_{1:(t-1)}, \fb_{1:k})
\end{align*}
which shows the claim \eqref{eq:sym_reduction_1}.

We next claim that for any tensor network $T$ in the variables
$(\tilde u_{1:t}, \tilde z_{1:(t-1)}, f_{1:k})$, we have
\begin{align}\label{eq:sym_reduction_2}
\val_T(\bW; \tilde\ub_{1:t}, \tilde\zb_{1:(t-1)},\fb_{1:k})
&=\val_{T'}(\bW; \tilde\ub_{1:(t-1)},\tilde\zb_{1:(t-1)}, \fb_{1:k})
\end{align}
for a tensor network $T'$ in the variables
$(\tilde u_{1:(t-1)}, \tilde z_{1:(t-1)}, f_{1:k})$. This holds because
$\tilde\ub_t = \tilde u_t(\tilde\zb_{1:(t-1)}, \fb_{1:k})$ is itself a 
polynomial of $(\tilde\zb_{1:(t-1)}, \fb_{1:k})$, so for each vertex $v$ of 
$T$, we may write
\[q_v(\tilde\ub_{1:(t-1)}, \tilde u_t(\tilde\zb_{1:(t-1)}, \fb_{1:k}), 
\tilde\zb_{1:(t-1)}, \fb_{1:k}) = \tilde q_v(\tilde\ub_{1:(t-1)}, 
\tilde\zb_{1:(t-1)}, \fb_{1:k})\]
for some polynomial $\tilde q_v$. Then we can define $T'$ by replacing each
polynomial $q_v$ with $\tilde q_v$ and preserving all other structures of $T$.

Having shown the reductions (\ref{eq:sym_reduction_1}) and
(\ref{eq:sym_reduction_2}), the proof is completed by recursively applying
these reductions for $t,t-1,t-2,\ldots,1$.
\end{proof}

Combining the above lemma, the state evolution (\ref{eq:Gstateevolution})
for the polynomial AMP algorithm applied to $\G$, and the given condition in
Lemma \ref{lemma:Wignercompare} that
tensor network values have the same limit for $\G$ and $\W$, we obtain the
following state evolution guarantee for the polynomial AMP algorithm applied to $\W$.

\begin{lemma}\label{lem:polyLawLimit}
In the setting of Lemma~\ref{lemma:Wignercompare}, for any fixed $t \geq 1$,
almost surely as $n \to \infty$
\[(\ub_1, \fb_1, \ldots, \fb_k, \tilde\zb_1(\W), \ldots, \tilde\zb_t(\bW)) 
\toW (U_1,F_1, \ldots, F_k,\widetilde Z_1, \ldots, \widetilde Z_t).\]
\end{lemma}
\begin{proof}
By Lemma~\ref{lemma:polynomial_tensor_network_decomposition}, for any polynomial 
$p: \RR^{t+k+1} \to \RR$, we have 
\begin{align*}
    \langle p(\u_1,\f_{1:k},\tilde\zb_{1:t}(\W))\rangle = \sum_{T \in \cF} 
    \val_T(\bW; \ub_1, \fb_{1:k})
\end{align*}
where $\cF$ is a finite set of diagonal tensor networks, and the same
decomposition holds for $\G$ in place of $\bW$. 
Then by the condition given in Lemma
\ref{lemma:Wignercompare} and the state evolution (\ref{eq:Gstateevolution}),
almost surely
\begin{equation}\label{eq:polyconvergence}
\lim_{n \to \infty} \langle p(\u_1,\f_{1:k},\tilde\zb_{1:t}(\W)) \rangle
=\lim_{n \to \infty} \langle p(\u_1,\f_{1:k},\tilde\zb_{1:t}(\G)) \rangle
=\EE[p(U_1,F_{1:k},\widetilde Z_{1:t})].
\end{equation}
In particular, this shows that on an event $\cE$ having probability 1,
all mixed moments of the empirical distribution of rows of
$(\u_1,\f_{1:k},\tilde\zb_{1:t}(\W))$ converge to those of
$(U_1,F_{1:k},\widetilde Z_{1:t})$.
Lemma~\ref{lemma:prod_determinance} implies that the joint law of
$(U_1,F_{1:k},\widetilde Z_{1:t})$ is uniquely determined by its mixed moments,
so on this event $\cE$, the empirical distribution of rows converges weakly to
$(U_1,F_{1:k},\widetilde Z_{1:t})$
(cf.\ \cite[Theorem 30.2]{billingsley1995probability}, which extends to the
multivariate setting by the same proof). On this event $\cE$, also
\[\lim_{n \to \infty} \frac{1}{n}\sum_{i=1}^n
\|(u_1[i],f_{1:k}[i],\tilde z_{1:t}(\W)[i])\|_2^{2d}
=\EE[\|(U_1,F_{1:k},\widetilde Z_{1:t})\|_2^{2d}]\]
for each integer $d \geq 1$, which shows (cf.\ \cite[Definition 6.8 and Theorem
6.9]{villani2009optimal}) that
\[(\u_1,\f_{1:k},\tilde\zb_{1:t}(\W)) \toW
(U_1,F_{1:k},\widetilde Z_{1:t}).\]
\end{proof}

\begin{remark}\label{remark:polynomialuniversality}
Lemmas \ref{lem:polyLawLimit}, \ref{lemma:Wignermoments}, and
\ref{lemma:syminvmoments} already imply
universality of the state evolution for polynomial AMP algorithms, 
without requiring the assumption $\|\W\|_{\op}<C$.
\end{remark}

We now proceed with an inductive comparison of the given Lipschitz AMP algorithm
(\ref{eq:AMP}) and the polynomial AMP algorithm (\ref{eq:aux_AMP_sym}), both
applied to $\W$. 
For each $t \geq 1$, let $(U_1,F_{1:k},Z_{1:t})$ describe the 
state evolution of the given Lipschitz AMP algorithm \eqref{eq:AMP}, where $Z_{1:t} 
\sim \cN(0, \bSigma_t)$ and $\bSigma_t$ is non-singular by assumption in Lemma
\ref{lemma:Wignercompare}.
Let $(U_1,F_{1:k},\widetilde Z_{1:t})$ describe
the state evolution of the auxiliary AMP algorithm~\eqref{eq:aux_AMP_sym} where 
$\widetilde Z_{1:t} \sim \cN(0, \widetilde\bSigma_t)$. We write as shorthand
\begin{align*}
&U_{s+1}=u_{s+1}(Z_{1:s},F_{1:k}), \qquad
\nabla U_{s+1}=(\partial_1 u_{s+1},\ldots,\partial_s u_{s+1})(Z_{1:s},F_{1:k}),\\
&\widetilde U_{s+1}=\tilde u_{s+1}(\widetilde Z_{1:s},F_{1:k}), \qquad
\nabla \widetilde U_{s+1}=(\partial_1 \tilde u_{s+1},\ldots,\partial_s \tilde
u_{s+1})(\widetilde Z_{1:s},F_{1:k})
\end{align*}
where the gradients are with respect to the first $s$ arguments.

All subsequent constants may depend on the Lipschitz
non-linearities $u_2,u_3,u_4,\ldots$, the corresponding Onsager coefficients
$\{b_{ts}\}$ and state evolution covariances $\{\bSigma_t\}$,
and joint laws of $(U_1,F_{1:k},Z_{1:t})$, which we treat as
fixed throughout this argument.

\begin{lemma}\label{lem:limit_law_compare}
Fix $t \geq 1$. Suppose (\ref{eq:polyapprox}) holds for $\eps>0$
and every polynomial $\tilde u_2,\ldots,\tilde u_{t+1}$. Suppose also
$\|\bSigma_t - \widetilde\bSigma_t\|_\op<\delta$ for $\delta>0$. Then for any
sufficiently small $\delta,\eps$, we have
\begin{align*}
\max_{s=1}^t \big\|\EE[\nabla U_{s+1}]-\EE[\nabla \widetilde U_{s+1}]
\big\|_2<\iota(\delta,\eps), \qquad
\max_{r, s = 1}^{t+1} \big|\EE[U_rU_s] 
- \EE[\widetilde U_r\widetilde U_s]\big|<\iota(\delta,\eps)
\end{align*}
for a constant $\iota(\delta,\eps)>0$ satisfying
$\iota(\delta,\eps) \to 0$ as $(\delta,\eps) \to (0,0)$.
\end{lemma}
\begin{proof}
We write $\iota(\delta,\eps)$ for any positive constant satisfying
$\iota(\delta,\eps) \to 0$ as $(\delta,\eps) \to (0,0)$ and changing from
instance to instance.
Since $\|\bSigma_t-\widetilde \bSigma_t\|_\op<\delta$, $\bSigma_t$ is
invertible, and $\bSigma_s$ is the upper-left submatrix of $\bSigma_t$ for $s \leq t$,
for sufficiently small $\delta>0$ and each $s=1,\ldots,t$ we have
\begin{equation}\label{eq:Zcoupling}
\|\bSigma_s^{-1}-\widetilde \bSigma_s^{-1}\|_\op<\iota(\delta,\eps), \qquad
\E\big[\big\|Z_{1:s}-\widetilde Z_{1:s}\big\|_2^2\big]<\iota(\delta,\eps)
\end{equation}
for a coupling of $Z_{1:s}$ and $\widetilde Z_{1:s}$.

We introduce the additional abbreviations for the intermediate quantities
\[\overline U_{s+1}=u_{s+1}(\widetilde Z_{1:s},F_{1:k}),
\qquad \nabla \overline U_{s+1}=(\partial_1 u_{s+1},\ldots,\partial_s
u_{s+1})(\widetilde Z_{1:s},F_{1:k}).\]
Then for any $s \in [t]$,
\begin{equation}\label{eq:induction_b_diff}
\big\|\EE[\nabla U_{s+1}] - \EE[\nabla\widetilde U_{s+1}]\big\|_2 
\leq \big\|\EE[\nabla U_{s+1}] - \EE[\nabla \overline U_{s+1}]\big\|_2
+ \big\|\EE[\nabla \overline U_{s+1}] - \EE[\nabla \widetilde U_{s+1}]\big\|_2
\end{equation}
Applying Stein's lemma, $(a+b+c)^2 \leq 3(a^2+b^2+c^2)$, and Cauchy-Schwarz,
the first term of (\ref{eq:induction_b_diff}) is bounded as
\begin{align*}
&\big\|\EE[\nabla U_{s+1}] - \EE[\nabla \overline U_{s+1}]\big\|_2^2\\
&=\big\|\EE[U_{s+1} \cdot \bSigma_s^{-1} Z_{1:s}^\top]
- \EE[\overline U_{s+1} \cdot \widetilde \bSigma_s^{-1} \widetilde Z_{1:s}^\top]
\big\|_2^2\\
&\leq 3\E\big[\big(U_{s+1}-\overline U_{s+1}\big)^2\big] \cdot
\E\big[\big\|\bSigma_s^{-1}Z_{1:s}^\top\big\|_2^2\big]
+3\E[\overline U_{s+1}^2] \cdot
\E\big[\big\|(\bSigma_s^{-1}-\widetilde \bSigma_s^{-1})Z_{1:s}^\top\big\|_2^2\big]\\
&\qquad+3\E[\overline U_{s+1}^2] \cdot 
\E\big[\big\|\widetilde \bSigma_s^{-1}(Z_{1:s}^\top-\widetilde
Z_{1:s}^\top)\big\|_2^2\big].
\end{align*}
The latter two terms are at most $\iota(\delta,\eps)$ by (\ref{eq:Zcoupling}),
and for the first term we have
\begin{equation}\label{eq:barcompare}
\E\big[\big(U_{s+1}-\overline U_{s+1}\big)^2\big]
\leq L_s^2 \cdot \E\big[\big\|Z_{1:s}-\widetilde Z_{1:s}\big\|_2^2\big]
<\iota(\delta,\eps)
\end{equation}
where $L_s$ is the Lipschitz constant of $u_{s+1}(\cdot)$. The second term of 
(\ref{eq:induction_b_diff}) is bounded similarly as
\begin{align*}
\big\|\EE[\nabla \overline U_{s+1}] - \EE[\nabla \widetilde U_{s+1}]\big\|_2^2
&=\big\|\EE[\overline U_{s+1} \cdot \widetilde \bSigma_s^{-1} \widetilde Z_{1:s}^\top]
- \EE[\widetilde  U_{s+1} \cdot \widetilde \bSigma_s^{-1} \widetilde Z_{1:s}^\top]
\big\|_2^2\\
&\leq \E\big[\big(\overline U_{s+1}-\widetilde U_{s+1}\big)^2\big] \cdot
\E\big[\big\|\widetilde \bSigma_s^{-1} \widetilde Z_{1:s}^\top\big\|_2^2\big],
\end{align*}
where by (\ref{eq:polyapprox}) we have
\begin{equation}\label{eq:bartildecompare}
\E\big[\big(\overline U_{s+1}-\widetilde U_{s+1}\big)^2\big]
=\E[(u_{s+1}(\widetilde Z_{1:s}, F_{1:K}) 
- \tilde{u}_{s+1}(\widetilde Z_{1:s},F_{1:K}))^2]
<\eps.
\end{equation}
Combining these bounds and applying them to (\ref{eq:induction_b_diff})
yields the first claim of the lemma,
$\|\EE[\nabla U_{s+1}]-\EE[\nabla \widetilde U_{s+1}]\|_2<\iota(\delta,\eps)$.
For the second claim, for any $r,s \in [t+1]$, we have
\begin{align*}
\big|\EE[U_{r+1} U_{s+1}] - \EE[\widetilde U_{r+1} \widetilde U_{s+1}]\big|
\leq \big|\EE[(U_{r+1}-\widetilde U_{r+1})U_{s+1}]\big|
+\big|\EE[\widetilde U_{r+1} (U_{s+1}-\widetilde U_{s+1})]\big|.
\end{align*}
Applying again Cauchy-Schwarz and the bounds (\ref{eq:barcompare})
and (\ref{eq:bartildecompare}) yields the second claim
$|\EE[U_{r+1} U_{s+1}] - \EE[\widetilde U_{r+1} \widetilde U_{s+1}]|
<\iota(\delta,\eps)$.
\end{proof}

\begin{lemma}\label{lemma:AMPpolyAPprox}
Fix $t \geq 1$. Suppose (\ref{eq:polyapprox}) holds for $\eps>0$ and
every polynomial $\tilde u_2,\ldots,\tilde u_{t+1}$. Then for any sufficiently
small $\eps$, almost surely for all large $n$, we have
\[\max_{s=1}^t \frac{1}{\sqrt{n}} \|\zb_s(\W) - \tilde\zb_s(\W)\|_2<\iota(\eps),
\qquad \max_{s=1}^t \frac{1}{\sqrt{n}} \|\zb_s(\W)\|_2<C,
\qquad \|\bSigma_t - \widetilde{\bSigma}_t\|_\op<\iota(\eps)\]
for constants $C>0$ and $\iota(\eps)>0$ satisfying $\iota(\eps) \to 0$ as
$\eps \to 0$.
\end{lemma}
\begin{proof}
We write $\zb_t,\tilde \zb_t,\ub_t,\tilde \ub_t$ for the iterates of the
Lipschitz and polynomial AMP algorithms applied to $\W$.
We prove the extended claim that there are constants $\iota(\eps)>0$ and $C>0$,
satisfying $\iota(\eps) \to 0$ as $\eps \to 0$,
for which almost surely for all large $n$,
\begin{enumerate}[(a)]
\item $\max_{s=1}^t |b_{ts} - \tilde b_{ts}|<\iota(\eps)$;
\item $\max_{s=1}^t \frac{1}{\sqrt{n}} \|\bz_s - \tilde\bz_s\|_2<\iota(\eps)$
and $\max_{s=1}^t \frac{1}{\sqrt{n}} \|\bz_s\|_2<C$;
\item $\|\bSigma_t - \widetilde{\bSigma}_t\|_\op<\iota(\eps)$;
\item $\max_{s=0}^t \frac{1}{\sqrt{n}} 
\|\bu_{s+1} - \tilde\bu_{s+1}\|_2<\iota(\eps)$
and $\max_{s=0}^t \frac{1}{\sqrt{n}} \|\bu_{s+1}\|_2<C$.
\end{enumerate}

We induct on $t$. For the base case $t=0$, statements (a--c) are vacuous, and
(d) holds by the equality of initializations $\tilde \ub_1 = \ub_1$ and by the
convergence of $\ub_1$ in Assumption \ref{assump:ufconvergence}.

Consider any $t \geq 1$, and suppose inductively that claims (a--d) all hold
for $t-1$. Denoting the constants in this inductive claim for $t-1$ as
$\iota_{t-1}(\eps)$ and $C_{t-1}$, let us write $\iota_t(\eps)$ and $C_t$ for
any positive constants depending on $\iota_{t-1}(\eps)$ and $C_{t-1}$,
satisfying $\iota_t(\eps) \to 0$ as $\eps \to 0$ and 
$\iota_{t-1}(\eps) \to 0$, and changing from instance to instance.

For (a), by the prescription (\ref{eq:symorthob}), each
$b_{ts}$ is a continuous function of $\EE[\nabla U_{s+1}]$ for $s \leq t-1$
and $\EE[U_r U_s]$ for $1 \leq r \leq s \leq t$. Then 
$\max_{s=1}^t |b_{ts}-\tilde b_{ts}|<\iota_t(\eps)$ by statement (c) of the
inductive hypothesis and Lemma~\ref{lem:limit_law_compare}.

For (b), by the definition of $\zb_t$ and $\tilde\zb_t$, we have
\[\frac{\|\bz_t - \tilde{\bz}_t\|_2}{\sqrt{n}} \leq 
\frac{\|\bW (\bu_t - \tilde{\bu}_t)\|_2}{\sqrt{n}} 
+ \sum_{s=1}^t |b_{ts}| \cdot \frac{\|\bu_s - \tilde\bu_s\|_2}{\sqrt{n}}
+ \sum_{s=1}^t |b_{ts} - \tilde{b}_{ts}| 
\cdot \frac{\|\tilde\bu_s\|_2}{\sqrt{n}}.\]
By the assumption that $\|\bW\|_\op \leq C$ almost surely for all large $n$,
by claim (d) of the inductive hypothesis, and by claim (a) already shown,
this is at most $\iota_t(\eps)$. Also,
\[\frac{\|\bz_t\|_2}{\sqrt{n}} \leq 
\frac{\|\bW \bu_t\|_2}{\sqrt{n}} 
+ \sum_{s=1}^t |b_{ts}| \cdot \frac{\|\bu_s\|_2}{\sqrt{n}}\]
which similarly is at most $C_t$.

For (c), by the prescription (\ref{eq:symorthobSigma}), the matrix $\bSigma_t$ is
(as in the proof of (a) above) a continuous function of
$\EE[\nabla U_{s+1}]$ for $s \leq t-1$ 
and $\EE[U_r U_s]$ for $1 \leq r \leq s \leq t$. Then $\|\bSigma_t-\widetilde
\bSigma_t\|_\op<\iota_t(\eps)$ again by statement (c) of the
inductive hypothesis and Lemma~\ref{lem:limit_law_compare}.

For (d), it follows from the 
definitions of $\ub_{t+1}$ and $\tilde\ub_{t+1}$ that
\begin{align*}
\frac{\|\bu_{t+1} - \tilde\bu_{t+1}\|_2}{\sqrt{n}}
&= \frac{\|u_{t+1}(\zb_{1:t}, \fb_{1:k}) 
- \tilde u_{t+1}(\tilde\zb_{1:t}, \fb_{1:k})\|_2}{\sqrt{n}}\\
&\leq \frac{\|u_{t+1}(\bz_{1:t}, \bff_{1:k}) 
- u_{t+1}(\tilde\bz_{1:t},\bff_{1:k})\|_2}{\sqrt{n}}
+ \frac{\|u_{t+1}(\tilde\bz_{1:t}, \fb_{1:k}) 
- \tilde u_{t+1}(\tilde\bz_{1:t},\bff_{1:k})\|_2}{\sqrt{n}}. 
\end{align*}
The first term is at most $\iota_t(\eps)$ by statement (b) already proved and
the fact that $u_{t+1}(\cdot)$ is Lipschitz.
For the second term, Lemma \ref{lem:polyLawLimit} shows
$(\tilde\zb_{1:t}, \fb_{1:k})
\toW (\widetilde Z_{1:t}, F_{1:k})$ almost surely as $n \to \infty$, and the
function $(u_{t+1}(\cdot)-\tilde u_{t+1}(\cdot))^2$ satisfies the polynomial
growth condition (\ref{eq:polygrowth}) by the given conditions for
$u_{t+1}(\cdot)$. Then
\begin{align}\label{eq:u_converge_2}
\lim_{n \to \infty} \frac{\|u_{t+1}(\tilde\zb_{1:t}, \fb_{1:k}) 
- \tilde u_{t+1}(\tilde\zb_{1:t},\fb_{1:k})\|_2^2}{n}
&= \EE[(u_{t+1}(\widetilde Z_{1:t}, F_{1:k}) 
- \tilde u_{t+1}(\widetilde Z_{1:t}, F_{1:k}))^2]<\eps
\end{align}
where the inequality is due to (\ref{eq:polyapprox}) in the construction of
$\tilde u_{t+1}$. Thus
$\|\bu_{t+1} - \tilde\bu_{t+1}\|_2/\sqrt{n}<\iota_t(\eps)$.
Similarly,
\begin{align*}
\frac{\|\bu_{t+1}\|_2}{\sqrt{n}}
&\leq \frac{\|u_{t+1}(\bz_{1:t}, \bff_{1:k}) -
u_{t+1}(\mathbf{0},\bff_{1:k})\|_2}{\sqrt{n}} +
\frac{\|u_{t+1}(\mathbf{0},\bff_{1:k})\|_2}{\sqrt{n}}.
\end{align*}
The first term is at most $C_t$ by statement (b) already proved and the fact
that $u_{t+1}(\cdot)$ is Lipschitz. For the second term, we have 
$\lim_{n \to \infty} \frac{1}{n}\|u_{t+1}(\mathbf{0},\bff_{1:k})\|_2^2
=\E[u_{t+1}(0,F_{1:k})^2]$
which is also at most a constant $C_t$. This concludes the proof of (d) and
completes the induction.
\end{proof}

Finally, we apply Lemma~\ref{lemma:AMPpolyAPprox} to prove 
Lemma~\ref{lemma:Wignercompare}.

\begin{proof}[Proof of Lemma~\ref{lemma:Wignercompare}]
Let $\ub_{1:t},\tilde \ub_{1:t},\z_{1:t},\tilde \zb_{1:t}$ be the iterates of 
the Lipschitz AMP algorithm and polynomial AMP algorithm applied to $\W$. 
We write $\iota(\eps)$ for a positive constant such that $\iota(\eps) \to 0$ as 
$\eps \to 0$, and changing from instance to instance.

To show $W_2$-convergence of $(\ub_1, \fb_{1:k}, \zb_{1:t})$,
consider any function $g: \RR^{t+k+1} \to \RR$ satisfying 
$|g(\xb) - g(\yb)| \leq C(1 + \|\xb\|_2 + \|\yb\|_2) \|\xb - \yb\|_2$ for a
constant $C > 0$. Then
\begin{align*}
&\big|\langle g(\ub_1, \fb_{1:k}, \zb_{1:t}) 
- g(\ub_1, \fb_{1:k}, \tilde\zb_{1:t})\rangle\big|\\
&\leq \frac{C}{n} \sum_{i=1}^n 
\left(1 + \|(u_1[i], f_{1:k}[i], z_{1:t}[i])\|_2
+ \|(u_1[i], f_{1:k}[i], \tilde z_{1:t}[i])\|_2\right) 
\cdot \|z_{1:t}[i] - \tilde z_{1:t}[i]\|_2\\
&\leq \frac{C}{n} \sqrt{3\sum_{i=1}^n 
(1 + \|(u_1[i], f_{1:k}[i], z_{1:t}[i])\|_2^2
+ \|(u_1[i], f_{1:k}[i], \tilde z_{1:t}[i])\|_2^2)}
\cdot \sqrt{\sum_{i=1}^n \|z_{1:t}[i] - \tilde z_{1:t}[i]\|_2^2}\\
&= \frac{C}{n} \sqrt{3n + 6\|\ub_1\|_2^2 + 6\sum_{j=1}^k \|\fb_j\|_2^2 
+ 3\sum_{s=1}^t (\|\zb_s\|_2^2 + \|\tilde\zb_s\|_2^2)} 
\cdot \sqrt{\sum_{s=1}^t \|\zb_s - \tilde\zb_s\|_2^2}
\end{align*}
This implies, by the statements for $\zb_{1:t}$ in
Lemma \ref{lemma:AMPpolyAPprox} and the convergence of
$(\u_1,\f_1,\ldots,\f_k)$ in Assumption~\ref{assump:ufconvergence}, that
almost surely for all large $n$,
\begin{equation}\label{eq:pseudo_lip_diff_bound}
\big|\langle g(\ub_1, \fb_{1:k}, \zb_{1:t}) 
- g(\ub_1, \fb_{1:k}, \tilde\zb_{1:t})\rangle\big|<\iota(\eps).
\end{equation}
Since $(\ub_1, \fb_{1:k}, \tilde\zb_{1:t}) \toW (U_1, F_{1:k}, 
\widetilde Z_{1:t})$ by Lemma \ref{lem:polyLawLimit}, we have 
\begin{align}\label{eq:pseudo_lip_aux_limit}
\lim_{n \to \infty} \langle g(\ub_1, \fb_{1:k}, \tilde\zb_{1:t})\rangle 
= \EE[g(U_1, F_{1:k}, \widetilde Z_{1:t})].
\end{align}
By the statement for $\bSigma_t$ in Lemma \ref{lemma:AMPpolyAPprox}, there 
is a coupling of $Z_{1:t}$ and $\widetilde Z_{1:t}$ such that
$\E[\|Z_{1:t}-\widetilde Z_{1:t}\|_2^2]<\iota(\eps)$. Then similarly
\begin{align}\label{eq:pseudo_lip_aux_limit_bound}
&\big|\EE[g(U_1, F_{1:k}, Z_{1:t})] 
- \EE[g(U_1, F_{1:k}, \widetilde Z_{1:t})]\big|\notag\\
&\qquad \leq C \cdot \EE\big[(1 + \|(U_1, F_{1:k}, Z_{1:t})\|_2 
+ \|(U_1, F_{1:k}, \widetilde Z_{1:t})\|_2)
\cdot \|Z_{1:t}-\widetilde Z_{1:t}\|_2\big]\notag\\
&\qquad \leq C \sqrt{\EE[3+6U_1^2+6\|F_{1:k}\|_2^2
+3\|Z_{1:t}\|_2^2+3\|\widetilde Z_{1:t}\|_2^2} \cdot
\sqrt{\E[\|Z_{1:t}-\widetilde Z_{1:t}\|_2^2]}<\iota(\eps).
\end{align}

Combining \eqref{eq:pseudo_lip_diff_bound}, 
\eqref{eq:pseudo_lip_aux_limit}, and \eqref{eq:pseudo_lip_aux_limit_bound}, we 
obtain for a (different) constant $\iota(\eps)>0$, almost surely for all large
$n$, $|\langle g(\ub_1, \fb_{1:k}, \zb_{1:t})\rangle - \EE[g(U_1, F_{1:k},
Z_{1:t})]| < \iota(\eps)$.
Since $\eps>0$ is arbitrary and $\iota(\eps) \to 0$ as $\eps \to 0$, we conclude
that
$\lim_{n \to \infty} \langle g(\ub_1, \fb_{1:k}, \zb_{1:t})\rangle = 
\EE[g(U_1, F_{1:k}, Z_{1:t})]$.
This holds for all bounded Lipschitz functions $g(\cdot)$ as well as for
$g(U_1, F_{1:k}, Z_{1:t})=\|(U_1, F_{1:k}, Z_{1:t})\|_2^2$,
which implies $(\ub_1, \fb_{1:k}, \zb_{1:t}) \toWtwo (U_1, F_{1:k}, Z_{1:t})$
(cf.\ \cite[Definition 6.8 and Theorem 6.9]{villani2009optimal}).
\end{proof}

Combining Lemmas~\ref{lemma:Wignercompare} and \ref{lemma:Wignermoments}
for $\bG \sim \GOE(n)$ concludes the proof of Theorem \ref{thm:Wigner}, and 
combining Lemmas~\ref{lemma:Wignercompare} and \ref{lemma:syminvmoments}
for an orthogonally invariant matrix $\bG$ with limit spectral distribution $D$
concludes the proof of Theorem \ref{thm:syminvariant}.

\section{Discussion}\label{sec:future}
In this work, we have established universality of the state evolution
for AMP algorithms applied to ensembles of matrices in both
Gaussian and non-Gaussian
universality classes, using an unfolding of polynomial AMP algorithms
into linear combinations of matrix-tensor networks.
Our analyses also reveal universality classes of matrices for which these tensor
networks have common limiting values, but where a more succinct
characterization of the limiting behavior of first-order iterative algorithms
is currently unknown. We hope that our work may inspire the development of
dynamical mean-field theory descriptions of such algorithms for these broader
matrix ensembles.

Recently, motivated by statistical applications, a burgeoning line of work
\cite{rush2018finite,li2022non,cademartori2023non,li2023approximate} has
studied non-asymptotic guarantees for AMP algorithms, in settings where the
underlying structure (e.g.\ sparsity) and the non-linearities applied may
depend on the dimension $n$, and for a number of iterations of the algorithm
that may also grow with the dimension $n$.
The study of AMP universality in such settings falls outside the scope of our
current analyses, and we believe this is an interesting direction for future work.

\appendix
\section{Density of polynomials}
\begin{lemma}\label{lemma:prod_determinance}
Let $\mu_X$ and $\mu_Y$ be probability laws on $\R^m$ and $\R^n$ having finite
moments of all orders, such that
multivariate polynomials are dense in the real $L^2$-spaces $L^2(\mu_X)$ and
$L^2(\mu_Y)$. Then multivariate polynomials are also dense in $L^2(\mu_X \times
\mu_Y)$.
\end{lemma}
\begin{proof}
Consider any measurable $A \subseteq \R^m$ and $B \subseteq \R^n$, and let
$\chi_A,\chi_B,\chi_{A \times B}$ be the indicator functions of $A$, $B$,
and $A \times B$. For any $\eps>0$, by the density conditions for $L^2(\mu_X)$
and $L^2(\mu_Y)$, we may take polynomials $p_A,p_B$ such that
$\|\chi_A-p_A\|_{L^2(\mu_X)}<\eps/2$ and 
$\|\chi_B-p_B\|_{L^2(\mu_Y)}<\eps/(2\|p_A\|_{L^2(\mu_X)})$.
Then
\begin{align*}
&\|\chi_{A \times B}-p_Ap_B\|_{L^2(\mu_X\times\mu_Y)}\\
&\leq \|\chi_A-p_A\|_{L^2(\mu_X)} \|\chi_B\|_{L^2{(\mu_Y)}} +
\|p_A\|_{L^2(\mu_X)} \|\chi_B - p_B\|_{L^2(\mu_Y)}<\eps.
\end{align*}
Taking $\eps \to 0$ shows that polynomials are dense in the linear span of
indicator functions $\{\chi_{A \times B}: \text{measurable } A \subseteq \R^m,B \subseteq \R^n\}$.
This linear span is in turn dense in $L^2(\mu_X \times \mu_Y)$, showing the
lemma.
\end{proof}

\section{Sufficient conditions for generalized
invariance}\label{appendix:syminvariant}

In this appendix, we prove Proposition~\ref{prop:syminvariant}, providing
examples of matrix models that satisfy the generalized invariance condition
of Definition \ref{def:syminvariant}.

\begin{lemma}\label{lemma:syminvariantb2}
Let $\bM \in \R^{n \times n}$ be a symmetric matrix having eigenvalues
$\bd \in \R^n$. Suppose $\bd \toW D$ almost surely as $n \to \infty$, where
$D$ has finite moments of all orders. Suppose $\bM$ satisfies
(\ref{eq:Mcondition}) almost surely for all large $n$.
Then for any $p(\bx) \in \Delta\langle \bx \rangle$,
\begin{enumerate}[(a)]
\item $\lim_{n \to \infty} \frac{1}{n} \Tr p(\bM)$ exists almost
surely, is finite, and depends only on the law of $D$.
\item For any $\eps>0$ and all large $n$, we have
\[\max_{i=1}^n |p(\bM)[i,i]-\tfrac{1}{n}\Tr p(\bM)|<n^{-1/2+\eps},
\qquad \max_{i \neq j} |p(\bM)[i,j]|<n^{-1/2+\eps}\]
\end{enumerate}
\end{lemma}
\begin{proof}
By the definition of diagonal monomials, every $p(\bx)\in\Delta\langle\bx
\rangle$ is a word of the form
\begin{align}\label{eq:diag_monomial}
	p(\bx) = \bx^{r_1} \Delta(p_1(\bx)) \bx^{r_2} \Delta(p_2(\bx))
	\cdots \bx^{r_{L}} \Delta(p_{L}(\bx))\bx^{r_{L+1}}
\end{align}
where each $p_\ell(\x) \in \Delta \langle \bx \rangle$ and each $r_\ell \geq 0$.
We define the \emph{depth} of $p(\bx)$, denoted by $\delta(p)$, as
$\delta(p)=0$ if $L=0$ (so that $p(\bx)=\bx^r$ for some $r\geq 0$), and
$\delta(p)=1+\max_{\ell=1}^L \delta(p_\ell)$ if $L \geq 1$. Thus
$\delta(p)$ is the maximum number of ``nested'' applications of
$\Delta(\cdot)$. We induct on $\delta(p)$.

For the base case where $\delta(p)=0$ and $p(\bx)=\bx^r$, we have
\[\tfrac{1}{n} \Tr p(\bM)= \tfrac{1}{n} \Tr(\bM^r) 
= \frac{1}{n} \sum_{i=1}^n d[i]^r 
\to \EE[D^r]\]
almost surely,
by the assumption $\bd\toW D$. Thus statement (a) holds, and statement (b) holds
by the assumed condition (\ref{eq:Mcondition}).

Suppose inductively that the lemma is true for all $p(\bx)$ with 
$\delta(p) \leq K$, and consider $p(\bx)$ with $\delta(p)=K+1$.
Fix any $\eps>0$. By the definition of depth, every $p_\ell(\bx)$ in
\eqref{eq:diag_monomial} satisfies $\delta(p_\ell)\leq K$.  
Then for every $\ell=1,2,\ldots,L$, by claim (b) of the induction hypothesis,
we can decompose $p_\ell(\bM)=\tfrac{1}{n} \Tr p_\ell(\bM) \cdot\Id+\bE_\ell$
where $\bE_\ell$ satisfies $\max_{i,j\in[n]} |E_\ell[i,j]|<n^{-1/2+\eps}$
almost surely for all large $n$.
Fix any $i,j\in[n]$ and write $i_0\equiv i$ and $i_{L+1}\equiv j$.
Then, applying this decomposition to every $p_\ell(\bM)$, we obtain
\begin{align*}
	p(\bM)[i,j] &= \sum_{\bi\in[n]^L} 
	M^{r_1}[i_0,i_1] p_1(\bM)[i_1,i_1]M^{r_2}[i_1,i_2]
	\cdots p_{L}(\bM)[i_{L},i_{L}]M^{r_{L+1}}[i_L,i_{L+1}]\\
	&= \sum_{\bi\in[n]^L} 
	\prod_{\ell=1}^{L+1} M^{r_\ell}[i_{\ell-1},i_{\ell}]
	\prod_{\ell=1}^L \bigg(\frac{1}{n} \Tr p_\ell(\bM) + E_\ell[i_{\ell}, i_{\ell}]\bigg)\\
	&= \sum_{\cJ\subseteq[L]} \bigg(\prod_{\ell\in[L]\setminus \cJ} 
	\frac{1}{n} \Tr p_\ell(\bM)\bigg)
	\sum_{\bi\in[n]^L} \prod_{\ell=1}^{L+1}
	M^{r_\ell}[i_{\ell-1},i_\ell] \prod_{\ell \in \cJ} E_\ell[i_\ell,i_\ell].
\end{align*}
By the induction hypothesis, the limit
\begin{equation}\label{eq:MJdef}
M_{\cJ}:=\lim_{n\to\infty}\prod_{\ell\in[L]\setminus \cJ} 
\frac{1}{n} \Tr p_\ell(\bM)
\end{equation}
exists, is finite, and depends only on the law of $D$.
We set $M_\cJ=1$ if $\cJ=[L]$.
Note that this convergence is uniform over pairs $i,j \in [n]$.
Therefore, for an error $\xi_\cJ=o(1)$ independent of $i$ and $j$,
\begin{align*}
	p(\bM)[i,j] &= \sum_{\cJ\subseteq[L]} (M_\cJ+\xi_\cJ)
	\sum_{\bi\in[n]^L} \prod_{\ell=1}^{L+1}
	M^{r_\ell}[i_{\ell-1},i_{\ell}] \prod_{\ell \in \cJ}
    E_\ell[i_\ell,i_\ell].
\end{align*}
We first sum over all indices $\{i_\ell:\ell \notin \cJ\}$:
Write explicitly 
$\cJ=\{\ell_1,\ldots,\ell_{|\cJ|}\}$ where $1\leq \ell_1<\ldots<\ell_{|\cJ|}\leq L$.
Let $\ell_0=0$ and $\ell_{|\cJ|+1}=L+1$, and denote
$R_\rho=r_{\ell_{\rho-1}+1}+\ldots+r_{\ell_\rho}$. Then this gives
\begin{align}\label{eq:diag_m_tr}
	p(\bM)[i,j] &= \sum_{\cJ\subseteq[L]} (M_\cJ+\xi_\cJ)
	\sum_{\bi\in[n]^\cJ} \prod_{\rho=1}^{|\cJ|+1}
	M^{R_\rho}[i_{\ell_{\rho-1}},i_{\ell_{\rho}}]
	\prod_{\ell \in\cJ}E_\ell[i_\ell,i_\ell].
\end{align}

We denote by $C>0$ a constant depending only on $p(\bx)$,
$\cJ$, and the law of $D$, and changing from instance to instance.
By (\ref{eq:Mcondition}), we have
$\max_{i \in [n]} |M^{R_\rho}[i,i]|<\frac{1}{n}\Tr \bM^{R_\rho}+n^{-1/2+\eps}<C$
and $\max_{i \neq j} |M^{R_\rho}[i,j]|<n^{-1/2+\eps}$
for each $\rho \in [|\cJ|+1]$, almost surely
for all large $n$. For any $\bi\in[n]^\cJ$, define 
$\Psi(\bi)=\{\rho\in[|\cJ|+1]:i_{\ell_{\rho-1}}\neq i_{\ell_{\rho}}\}$. Then
this implies
\begin{equation}\label{eq:MEprodbound}
	\prod_{\rho=1}^{|\cJ|+1} \big|
	M^{R_\rho}[i_{\ell_{\rho-1}},i_{\ell_{\rho}}]\big|
	\prod_{\ell \in\cJ} |E_\ell[i_\ell,i_\ell]| \leq 
	Cn^{(-1/2+\eps)(|\Psi(\bi)|+|\cJ|)}.
\end{equation}
Moreover, if $\psi \geq 1$, then note that $|\{\bi\in[n]^\cJ:|\Psi(\bi)|=\psi\}|
\leq Cn^{\psi-1}$, because $i_{\ell_0}=i_0=i$ and $i_{\ell_{|\cJ|+1}}=i_{L+1}=j$
are fixed, so there is freedom to choose $\psi-1$ remaining index values.
Combining this with (\ref{eq:MEprodbound}), for any $\psi \geq 1$,
\begin{align}\label{eq:diag_m_bound}
	\sum_{\bi\in[n]^\cJ:|\Psi(\bi)|=\psi}
	\prod_{\rho=1}^{|\cJ|+1} \big|
	M^{R_\rho}[i_{\ell_{\rho-1}},i_{\ell_{\rho}}]\big|
	\prod_{\ell \in\cJ} |E_\ell[i_\ell,i_\ell]| &\leq 
	Cn^{\psi-1} \cdot n^{(-1/2+\eps)(\psi+|\cJ|)}\notag\\
	&\leq Cn^{-1/2+\eps(2|\cJ|+1)}
\end{align}
where the last inequality follows from the observation that we always have 
$|\Psi(\bi)| \leq |\cJ|+1$. For $\psi=0$, we must have $i=j$ and
$|\{\bi\in[n]^\cJ:|\Psi(\bi)|=\psi\}|=1$. Then by (\ref{eq:MEprodbound}),
this bound
(\ref{eq:diag_m_bound}) still holds as long as $|\cJ| \geq 1$, i.e.\ $\cJ \neq
\emptyset$.

Applying \eqref{eq:diag_m_bound} for all non-empty 
$\cJ\subseteq [L]$, it follows from \eqref{eq:diag_m_tr} that
\begin{equation}\label{eq:diag_m_tr_final}
\big|p(\bM)[i,j] - \1\{i=j\}(M_\emptyset+\xi_\emptyset)
M^{R}[i,i]\big| \leq Cn^{-1/2+\eps(2L+1)}
\end{equation}
where we set $R=r_1+\ldots+r_{L+1}$. The above bounds all hold uniformly
over $i,j \in [n]$, and hence (\ref{eq:diag_m_tr_final})
holds simultaneously for all
pairs $i,j \in [n]$, almost surely for all large $n$. Thus, combining with the
condition (\ref{eq:Mcondition}) for $M^R[i,i]$, we conclude that both
$\max_{i\neq j} |p(\bM)[i,j]|$ and
$\max_{i=1}^n |p(\bM)[i,i]-(M_\emptyset+\xi_\emptyset) \cdot \frac{1}{n}\Tr
\bM^R|$ are at most $n^{-1/2+\eps(2L+2)}$ for all large $n$.
Then $\{p(\bM)[i,i]:i \in [n]\}$
are uniformly close to a value independent of $i \in [n]$, which implies also
$\max_{i=1}^n |p(\bM)[i,i]-\frac{1}{n}\Tr p(\bM)|<2n^{-1/2+\eps(2L+2)}$.
These statements hold for any $\eps>0$, showing the inductive claim (b) for
$p(\bx)$. Moreover, averaging (\ref{eq:diag_m_tr_final}) over $i=j$ gives
\begin{align*}
	\lim_{n\to\infty}\frac{1}{n}\Tr p(\bM)=\lim_{n\to\infty} 
	(M_\emptyset+\xi_\emptyset) \cdot \frac{1}{n} \Tr M^R = M_\emptyset \cdot \EE[D^R],
\end{align*}
and we recall from (\ref{eq:MJdef}) that $M_\emptyset$ depends only on the law
of $D$. This shows the inductive claim (a) for $p(\bx)$, completing the
induction.
\end{proof}

\begin{proof}[Proof of Proposition \ref{prop:syminvariant}]
Lemma \ref{lemma:syminvariantb2} implies that the matrix model
in Proposition~\ref{prop:syminvariant}(b2) satisfies
Definition \ref{def:syminvariant}, where the limit diagonal law $\cD_\tdiag$
is determined uniquely by the limit spectral distribution $D$. To complete the
proof of Proposition~\ref{prop:syminvariant}, it suffices to verify that
the orthogonally invariant matrix model of part (a)
and the model of part (b1) are both special cases of the model in part (b2).

If $\bW=\bO\bD\bO^\top$ is orthogonally invariant, i.e.\ $\bO \sim \Haar(\OO(n))$
is independent of $\bD$, then also $\bO\overset{L}{=} \bPi_V \bO \bPi_E$ where
$\bPi_V,\bPi_E$ are uniformly random signed permutations independent of $\bO$.
The entries of $\bO$ satisfy the delocalization condition
(\ref{eq:symdelocalization}) almost surely for all large $n$,
as is implied by \cite[Theorem 1]{jiang2005maxima}. Thus $\bW$ is a special case
of the model in part (b1).

Now suppose $\bW$ is any matrix satisfying the description of part (b1).
Then $\bW$ has the simpler form $\bW=\bPi\bM\bPi^\top$ where $\bPi=\bPi_V$,
\[\bM=\bH\bP\bD\bP^\top \bH^\top,\]
and $\bP=\bP_E$ is the random permutation
corresponding to $\bPi_E=\bP_E\bXi_E$. Here, we have eliminated the diagonal sign
matrices $\bXi_E$ from the expression using $\bXi_E\bD\bXi_E^\top=\bD$. To show
that $\bW$ is an example of the model in part (b2),
it remains to show that this matrix $\bM$ satisfies the condition
(\ref{eq:Mcondition}) almost surely for all large $n$.

Consider $\bM^\nu$ for any fixed integer $\nu \geq 1$.
Let $\bh_i \in \R^n$ denote the $i^\text{th}$ row of $\bH$,
and let $\sigma$ be the permutation of $[n]$ for which $P[i,\sigma(i)]=1$ for
all $i \in [n]$. Then
\begin{equation}\label{eq:Mnuexpr}
M^\nu[i,j]=(\bH\bP\bD^\nu \bP^\top \bH^\top)[i,j]
=\sum_{k=1}^n h_i[k]D^\nu[\sigma(k),\sigma(k)]h_j[k]
\end{equation}
We condition on $(\bD,\bH)$, and write $\E$ for the expectation over only the
permutation $\sigma$. Then for each fixed $k \in [n]$, we have
$\E[D^\nu[\sigma(k),\sigma(k)]]=n^{-1} \Tr \bD^\nu$, so
\[\E[M^\nu[i,j]]=\tfrac{1}{n}\Tr \bD^\nu \cdot \bh_i^\top \bh_j
=\tfrac{1}{n}\Tr \bM^\nu \cdot 1\{i=j\}.\]
We now show concentration of $M^\nu[i,j]$ around this expectation by
computing its high moments: Consider first any fixed $i \neq j \in [n]$,
and abbreviate $\tilde h[k]=h_i[k]h_j[k]$ and $\tilde d[k]=D^\nu[k,k]$. Then 
from (\ref{eq:Mnuexpr}), $M^\nu[i,j]=\sum_{k=1}^n \tilde h[k] \tilde d[\sigma(k)]$,
so for any even integer $p \geq 2$,
\[\E[(M^\nu[i,j])^p]=\sum_{\bk \in [n]^p}
\tilde h[k_1]\ldots \tilde h[k_p]
\E\Big[\tilde d[\sigma(k_1)]\ldots \tilde d[\sigma(k_p)]\Big].\]
Let $\cP$ be the lattice of partitions of $[p]$, endowed with the usual partial
ordering by refinement. For each $\bk \in [n]^p$, let
$\pi(\bk) \in \cP$ be the partition induced by $\bk$,
i.e.\ $i,j \in [p]$ belong to a common block of $\pi$ if and only if $k_i=k_j$.
Then
\begin{align}
\E[(M^\nu[i,j])^p]&=\sum_{\pi \in \cP}
\sum_{\bk \in [n]^p:\pi(\bk)=\pi} 
\tilde h[k_1]\ldots \tilde h[k_p]
\E\Big[\tilde d[\sigma(k_1)]\ldots \tilde d[\sigma(k_p)]\Big]
\notag\\
&=\sum_{\pi \in \cP} \mathop{\sum_{\bk \in [n]^p}}_{\pi(\bk)=\pi} \tilde h[k_1]\ldots \tilde h[k_p]
\cdot \frac{(n-|\pi|)!}{n!} 
\mathop{\sum_{\bl \in [n]^p}}_{\pi(\bl)=\pi} \tilde d[l_1]\ldots \tilde d[l_p],\label{eq:EMijoff}
\end{align}
the second equality using that the permutation $\sigma$ is uniformly random, so
the expectation over $\sigma$ yields a uniform average over new choices for the
$|\pi|$ distinct index values of $\bk$. 

Let $\mu(\pi,\pi')$ for $\pi \leq \pi'$ be the M\"obius function over $\cP$,
satisfying the inversion relation
(see e.g.\ \cite[Eq.\ (10.10)]{nica2006lectures})
$\sum_{\tau \in \cP:\pi \leq \tau \leq \pi'} \mu(\pi,\tau)
=\1\{\pi=\pi'\}$. Then for any function $f$,
\begin{equation}\label{eq:inclusionexclusion}
\mathop{\sum_{\bk \in [n]^p}}_{\pi(\bk)=\pi} f(\bk)
=\mathop{\sum_{\bk \in [n]^p}}_{\pi(\bk) \geq \pi} f(\bk) \cdot
\mathop{\sum_{\tau \in \cP}}_{\pi \leq \tau \leq \pi(\bk)} \mu(\pi,\tau)
=\mathop{\sum_{\tau \in \cP}}_{\tau \geq \pi}
\mu(\pi,\tau) \mathop{\sum_{\bk \in [n]^p}}_{\pi(\bk) \geq \tau} f(\bk).
\end{equation}
Applying this to the term involving $\tilde h$ in (\ref{eq:EMijoff}),
\[\sum_{\bk \in [n]^p:\pi(\bk)=\pi} \tilde h[k_1]\ldots \tilde h[k_p]
=\sum_{\tau \in \cP:\tau \geq \pi} \mu(\pi,\tau)
\prod_{R \in \tau} \sum_{k=1}^n \tilde h[k]^{|R|}.\]
Recalling $\tilde h[k]=h_i[k]h_j[k]$ where $i \neq j$, we have
$\sum_{k=1}^n \tilde h[k]=\bh_i^\top \bh_j=0$. Thus the summand for $\tau$
vanishes if $\tau$ has a singleton block.
For all other partitions $\tau \in \cP$, its number of blocks satisfies
$|\tau| \leq p/2$. Then applying $|\tilde{h}[k]| \leq n^{2(-1/2+\eps)}$
by the delocalization condition (\ref{eq:symdelocalization}) for $\bH$,
for any fixed $\eps>0$ and all large $n$,
\[\left|\prod_{R \in \tau} \sum_{k=1}^n \tilde h[k]^{|R|}\right|
\leq n^{2p(-1/2+\eps)} \cdot n^{|\tau|} \leq n^{-p/2+2p \eps}.\]
Thus $|\sum_{\bk \in [n]^p:\pi(\bk)=\pi} \tilde h[k_1]\ldots \tilde h[k_p]|
\leq Cn^{-p/2+2p \eps}$ where, here and below, we denote by $C>0$ a
$(\pi,D)$-dependent constant that may change from instance to instance.
By the assumption $\bd \toW D$ and
Lemma~\ref{lemma:distinctness} (applied with $\cS$
being the blocks of $\pi$ and $q_S(x)=\tilde{d}(x)^{|S|}$ for $S \in \pi$),
also
$n^{-|\pi|} |\sum_{\bl \in [n]^p:\pi(\bl)=\pi} \tilde d[l_1]\ldots \tilde d[l_p]| \leq C$.
Applying these to (\ref{eq:EMijoff}), we obtain
$\E[(M^\nu[i,j])^p] \leq Cn^{-p/2+2p \eps}$, so
$\PP[|M^\nu[i,j]|>n^{-1/2+3\eps}] \leq Cn^{-p\eps}$ by Markov's inequality.
Choosing even $p \geq 2$ sufficiently large and taking a union bound over all
$i \neq j$, this shows that the second condition of (\ref{eq:Mcondition}) 
holds almost surely for all large $n$.

The case $i=j$ is similar: Fix $i \in [n]$ and now abbreviate
$\tilde h[k]=h_i[k]^2$ and $\tilde d[k]=D^\nu[k,k]-n^{-1}\Tr \bD^\nu$.
Then from (\ref{eq:Mnuexpr}),
$M^\nu[i,i]-n^{-1}\Tr \bM^\nu=\sum_k \tilde h[k]\tilde d[k]$, so we obtain
analogously to (\ref{eq:EMijoff})
\[\E[(M^\nu[i,i]-n^{-1}\Tr \bM^\nu)^p]
=\mathop{\sum_{\bk \in [n]^p}}_{\pi(\bk)=\pi} \tilde h[k_1]\ldots \tilde h[k_p]
\cdot \frac{(n-|\pi|)!}{n!} 
\mathop{\sum_{\bl \in [n]^p}}_{\pi(\bl)=\pi} \tilde d[l_1]\ldots \tilde d[l_p].\]
Applying the M\"obius inversion relation (\ref{eq:inclusionexclusion}) now to
the second summation over $\bl$,
\[\sum_{\bl \in [n]^p:\pi(\bl)=\pi} \tilde d[l_1]\ldots \tilde d[l_p]
=\sum_{\tau \in \cP:\tau \geq \pi} \mu(\pi,\tau)
\prod_{R \in \tau} \sum_{l=1}^n \tilde{d}[l]^{|R|}.\]
Using that $\sum_{k=1}^n \tilde d[k]=0$, the summand for $\tau$ vanishes if
$\tau$ has a singleton block. For all other partitions $\tau \in \cP$,
applying $\bd \toW D$, we obtain
\[\left|\prod_{R \in \tau} \sum_{l=1}^n \tilde{d}[l]^{|R|}\right|
\leq Cn^{|\tau|} \leq Cn^{p/2}.\]
Then $|\sum_{\bl \in [n]^p:\pi(\bl)=\pi} \tilde d[l_1]
\ldots \tilde d[l_p]| \leq Cn^{p/2}$.
From (\ref{eq:symdelocalization}), we have also
\[n^{-|\pi|} \sum_{\bk \in [n]^p:\pi(\bk)=\pi} |\tilde h[k_1]\ldots \tilde h[k_p]|
\leq n^{-2p(1/2+\eps)}.\]
Then $\E[(M^\nu[i,i]-n^{-1}\Tr \bM^\nu)^p] \leq Cn^{-p/2+2p\eps}$,
so the first condition of (\ref{eq:Mcondition}) follows also by Markov's
inequality and a union bound. This verifies that $\bW$ satisfying part (b1) also
satisfies part (b2), as desired.
\end{proof}

\section{Tensor network value under orthogonal invariance}\label{appendix:orthogonal}

In this Appendix, we derive a more explicit combinatorial form for the tensor
network value of Lemma~\ref{lemma:syminvmoments} when $\bW$ is an
orthogonally invariant matrix, using the orthogonal Weingarten calculus. We then
prove the asymptotic freeness statement of Proposition \ref{prop:freeness}(b).

Let $T$ be a tensor network with $w+1$ vertices and $w$ edges.
Then there are $2w$ vertex-edge pairs $(v,e)$
where edge $e$ is incident to vertex $v$. We label these vertex-edge pairs
arbitrarily as $1,2,\ldots,2w$. Let $\cP$ be the lattice of partitions
of $[2w]$, endowed with the usual partial ordering by refinement.
We define two distinguished partitions $\pi_V,\pi_E \in \cP$, such that
vertex-edge pairs $\rho,\tau \in [2w]$ belong
to the same block of $\pi_V$ if and only if they have the same vertex $v$, and
to the same block of $\pi_E$ if and only if they have the same edge $e$. (Thus
$\pi_V$ has $w+1$ blocks, one for each vertex of $T$, and $\pi_E$ is a pairing
with $w$ pairs, one for each edge of $T$.)

Define a metric over $\cP$ by
\begin{equation}\label{eq:metric}
d(\pi,\pi')=|\pi|+|\pi'|-2|\pi \vee \pi'|
\end{equation}
where $\pi \vee \pi'$ is the join (i.e.\ least upper bound) of $\pi$ and $\pi'$.
This is shown in
\cite{arabie1973multidimensional,boorman1973metrics} to be equivalent to the
smallest number of merge and divide operations needed to transform $\pi$ into
$\pi'$, where a merge operation combines any two blocks into one block, and a
divide operation splits any one block into two blocks. From this
characterization, it is immediate that $d(\cdot,\cdot)$ satisfies the triangle
inequality $d(\pi,\pi')+d(\pi',\pi'') \geq d(\pi,\pi'')$.
We call a path $\pi_0 \to \pi_1 \to \ldots \to \pi_k$ of partitions
a \emph{$d$-geodesic} if it is a shortest path from $\pi_0$ to $\pi_k$ in the
metric $d(\cdot,\cdot)$, i.e.\ if
\[d(\pi_0,\pi_k)=d(\pi_0,\pi_1)+d(\pi_1,\pi_2)+\ldots+d(\pi_{k-1},\pi_k).\]

The main result of this Appendix is the following proposition.

\begin{proposition}\label{prop:orthogonallimval}
In the setting of Lemma \ref{lemma:syminvmoments}, suppose in addition that
$\bW=\bO\bD\bO^\top$ is orthogonally invariant,
where $\bD=\diag(\bd)$ and $\bd\toW D$ almost surely as $n \to \infty$.
For $\pi \geq \pi_V$ and $\pi' \geq \pi_E$, define
\begin{align}
q(\pi)&=\prod_{S \in \pi} \E\left[\prod_{\text{distinct vertices } v
\text{ in vertex-edge pairs of } S} q_v(X_1,\ldots,X_k)\right]
\label{eq:qlimortho_appendix}\\
D(\pi')&=\prod_{S \in \pi'} \E\left[D^{\text{number of distinct edges in
vertex-edge pairs of } S}\right].\label{eq:Dlimortho}
\end{align}
Then
\[\lim_{n \to \infty} \val_T(\W;\x_1,\ldots,\x_k)
=\sum_{j \geq 0} \mathop{\sum_{\text{distinct pairings }
\pi_0,\ldots,\pi_j \text{ of } [2w]}}_{\pi_V \to \pi_0 \to \ldots \to \pi_j
\to \pi_E \text{ is a } d\text{-geodesic}}
(-1)^j q(\pi_V \vee \pi_0)D(\pi_E \vee \pi_j).\]
(Here $\pi_j$ is not required to be distinct from $\pi_E$.)
\end{proposition}

To show this result, we apply the following statements
derived from the orthogonal Weingarten calculus
of \cite{collins2006integration} for mixed moments of entries of
Haar-orthogonal random matrices.

\begin{lemma}\label{lem:weingarten}
Let $\Ob \sim \Haar(\OO(n))$. Let $\bi=(i_1,\ldots,i_{2w})$ and
$\bj=(j_1,\ldots,j_{2w})$ be any index tuples in $[n]^{2w}$. Then
\begin{equation}\label{eq:orthogmoments}
\E\bigg[\prod_{p=1}^{2w} O[i_p,j_p]\bigg]
=\mathop{\sum_{\text{pairings } \pi,\pi' \text{ of } [2w]}}_{\pi \leq
\pi(\bi),\;\pi' \leq \pi(\bj)} \Wg_n[\pi,\pi']
\end{equation}
where $\Wg_n$ is the orthogonal Weingarten function. For fixed $w$,
as $n \to \infty$, this satisfies
\begin{equation}\label{eq:weingartenasymp}
\Wg_n[\pi,\pi']=n^{-w-d(\pi,\pi')/2} \cdot \Moeb(\pi,\pi')
+O(n^{-w-d(\pi,\pi')/2-1})
\end{equation}
where $d(\pi,\pi')$ is the metric (\ref{eq:metric}), and $\Moeb(\pi,\pi')$ is
the M\"obius function on the non-crossing partition lattice, given by
\begin{equation}\label{eq:moeb}
\Moeb(\pi,\pi')=\sum_{k \geq 0}
\mathop{\sum_{\text{distinct pairings } \pi_0,\pi_1,\ldots,\pi_k \text{ of }
[2w]}}_{\pi_0 \to \pi_1 \to
\ldots \to \pi_k \text{ is a } d\text{-geodesic from }
\pi_0=\pi \text{ to } \pi_k=\pi'} (-1)^k.
\end{equation}
\end{lemma}
\begin{proof}
We may identify pairings $\pi,\pi'$ of $[2w]$ as permutations in
the symmetric group $\mathrm{S}_{2w}$, each a product of $w$ disjoint
transpositions corresponding to the $w$ pairs.
The cycle decomposition of their product $\pi\pi'$ in $\mathrm{S}_{2w}$
has exactly two cycles for each set of their join partition
$\pi \vee \pi'$. Then, the metric $l(\pi,\pi')=|\pi\pi'|/2$ used in
\cite[Section 3]{collins2006integration} (where $|\cdot|$ is the Cayley distance
to the identity permutation in $\mathrm{S}_{2w}$, given by $2w$ minus the number
of cycles) is equivalently
\begin{equation}\label{eq:metricequiv}
l(\pi,\pi')=\frac{2w-2|\pi \vee \pi'|}{2}=\frac{d(\pi,\pi')}{2}
\end{equation}
where the right side is our metric $d(\cdot,\cdot)$ restricted to pairings. The
statements (\ref{eq:orthogmoments}) and (\ref{eq:weingartenasymp}) then follow
from \cite[Corollary 3.4 and Theorem 3.13]{collins2006integration}.
The form (\ref{eq:moeb}) for the M\"obius function follows from comparing
\cite[Theorem 3.13]{collins2006integration} with \cite[Lemma
3.12]{collins2006integration}, noting that the leading-order terms of
\cite[Lemma 3.12]{collins2006integration} come from paths of pairings
satisfying $\pi_i \neq \pi_{i+1}$ for each $i=0,\ldots,k-1$
and also 
$l(\pi_0,\pi_1)+\ldots+l(\pi_{k-1},\pi_k)=l(\pi_0,\pi_k)$. Any such path must
be a geodesic of $k+1$ unique pairings in the metric $l(\cdot,\cdot)$, and
hence also in the metric $d(\cdot,\cdot)$ by the equivalence
(\ref{eq:metricequiv}), and this shows (\ref{eq:moeb}).
\end{proof}

\begin{proof}[Proof of Proposition \ref{prop:orthogonallimval}]

Expanding the product $\W=\Ob\Db\Ob^\top$,
the tensor network value is given by
\[\val_T(\W;\x_{1:k})=\frac{1}{n}\sum_{\bi \in [n]^{\cV}}
\sum_{\bj \in [n]^{\cE}} \prod_{v \in \cV} q_v(x_{1:k}[i_v])
\prod_{e=(u,v) \in \cE} O[i_u,j_e]D[j_e,j_e]O[i_v,j_e].\]
For each vertex $v$ or edge $e$, let $\rho(v),\rho(e) \in [2w]$ be an arbitrary
choice of vertex-edge pair containing this vertex or this edge. 
Then this is equivalently expressed as
\begin{equation}\label{eq:orthoval_appendix}
\val_T(\W;\x_{1:k})=\frac{1}{n}
\mathop{\sum_{\bi \in [n]^{2w}}}_{\pi(\bi) \geq \pi_V}
\mathop{\sum_{\bj \in [n]^{2w}}}_{\pi(\bj) \geq \pi_E}
\prod_{v \in \cV} q_v(x_{1:k}[i_{\rho(v)}])
\prod_{e \in \cE} D[j_{\rho(e)},j_{\rho(e)}] \prod_{\rho=1}^{2w} 
O[i_\rho,j_\rho].
\end{equation}
Note that by the constraints $\pi(\bi) \geq \pi_V$ and $\pi(\bj) \geq \pi_E$,
this expression is the same for any choices of vertex-edge pairs
$\rho(v),\rho(e) \in [2w]$. 

Let $\E$ be the expectation over $\bO$, conditional on $\x_1,\ldots,\x_k$ and
$\D$. By Lemma \ref{lem:weingarten}, we have
\[\E\bigg[\prod_{p=1}^{2w} O[i_p,j_p]\bigg]
=\sum_{\text{pairings } \pi,\pi' \text{ of } [2w]}
\1_{\pi(\bi) \geq \pi}\1_{\pi(\bj) \geq \pi'} \cdot
n^{-w-d(\pi,\pi')}(\Moeb(\pi,\pi')+o(1)).\]
Note that
\[\1_{\pi(\bi) \geq \pi_V}\1_{\pi(\bi) \geq \pi}=\1_{\pi(\bi) \geq
\pi_V \vee \pi}, \qquad
\1_{\pi(\bj) \geq \pi_E}\1_{\pi(\bj) \geq \pi'}=\1_{\pi(\bj) \geq
\pi_E \vee \pi'}.\]
Identifying summations over
$\bi,\bj \in [n]^{2w}$ with $\pi(\bi) \geq \pi_V \vee \pi$
and $\pi(\bj) \geq \pi_E \vee \pi'$ as a summation over one index in $[n]$ for
each block of $\pi_V \vee \pi$ and $\pi_E \vee \pi'$, and applying the
given conditions that $\x_{1:k} \toW X_{1:k}$ and $\diag(\D) \toW D$ almost
surely, observe that
\begin{align*}
\frac{1}{n^{|\pi_V \vee \pi|}}\sum_{\bi \in [n]^{2w}}
\1_{\pi(\bi) \geq \pi_V \vee \pi} \prod_{v \in \cV} q_v(x_{1:k}[i_{\rho(v)}])
&\to q(\pi_V \vee \pi),\\
\frac{1}{n^{|\pi_E \vee \pi|}}\sum_{\bj \in [n]^{2w}}
\1_{\pi(\bj) \geq \pi_E \vee \pi'} \prod_{e \in \cE} D[j_{\rho(e)},j_{\rho(e)}]
&\to D(\pi_E \vee \pi')
\end{align*}
where $q(\cdot)$ and $D(\cdot)$ are as defined in
(\ref{eq:qlimortho_appendix}) and (\ref{eq:Dlimortho}).
Then, taking the expectation over $\Ob$ in (\ref{eq:orthoval_appendix})
and applying these observations,
\begin{align}
\E[\val_T(\W;\x_{1:k})]
&=\sum_{\text{pairings } \pi,\pi' \text{ of } [2w]}
\frac{1}{n} \cdot n^{|\pi_V \vee \pi|} \cdot n^{|\pi_E \vee \pi'|}
\cdot n^{-w-d(\pi,\pi')}\notag\\
&\qquad \cdot \Big(\Moeb(\pi,\pi') \cdot q(\pi_V \vee \pi) \cdot
D(\pi_E \vee \pi')+o(1)\Big).\label{eq:finalorthoval}
\end{align}

Recall that $|\pi_V|=w+1$ and $|\pi|=|\pi'|=|\pi_E|=w$ as these are all pairings
of $[2w]$. Then by definition of the metric $d(\cdot,\cdot)$,
\[|\pi_V \vee \pi|=\frac{2w+1-d(\pi_V,\pi)}{2},
\qquad |\pi_E \vee \pi|=\frac{2w-d(\pi_E,\pi)}{2}.\]
So the above value simplifies to
\[\sum_{\text{pairings } \pi,\pi' \text{ of } [2w]}
n^{\frac{2w-1-d(\pi_V,\pi)-d(\pi,\pi')-d(\pi',\pi_E)}{2}}
\Big(\Moeb(\pi,\pi')q(\pi_V \vee \pi)D(\pi_E \vee \pi')+o(1)\Big).\]
Applying the triangle inequality for $d(\cdot,\cdot)$ and the identity $|\pi_V
\vee \pi_E|=1$ since $T$ is a connected tree, we have
\[d(\pi_V,\pi)+d(\pi,\pi')+d(\pi',\pi_E)
\geq d(\pi_V,\pi_E)=(w+1)+w-2=2w-1,\]
and equality holds if and only if $\pi_V \to \pi \to \pi' \to \pi_E$ is a
$d$-geodesic. Thus, we obtain the limit value
\begin{equation}\label{eq:orthoglimitvalue}
\lim_{n \to \infty} \E[\val_T(\W;\x_{1:k})]
=\mathop{\sum_{\text{pairings } \pi,\pi' \text{ of }
[2w]}}_{\pi_V \to \pi \to \pi' \to \pi_E \text{ is a } d\text{-geodesic}}
\Moeb(\pi,\pi')q(\pi_V \vee \pi)D(\pi_E \vee \pi').
\end{equation}
Here, $\pi$ and $\pi'$ are pairings of $[2w]$ that may coincide with each
other and/or with $\pi_E$.

Finally, we apply (\ref{eq:moeb}) to express $\Moeb(\pi,\pi')$ also as a
summation over geodesic paths of pairings from $\pi$ to $\pi'$, giving
\[\lim_{n \to \infty} \E[\val_T(\W;\x_{1:k})]
=\sum_{j \geq 0} \mathop{\sum_{\text{distinct pairings }
\pi_0,\ldots,\pi_j \text{ of } [2w]}}_{\pi_V \to \pi_0 \to \ldots \to \pi_j
\to \pi_E \text{ is a } d\text{-geodesic}}
(-1)^j q(\pi_V \vee \pi_0)D(\pi_E \vee \pi_j).\]
We have set $\pi_0=\pi$ and $\pi_j=\pi'$, and the terms of the sum with $j=0$
correspond to $\pi=\pi'$. This shows that the stated form is the almost-sure
limit of $\E[\val_T(\W;\x_{1:k})]$ where $\E$ is the expectation over
$\bO$. Comparing with the result of Lemma
\ref{lemma:syminvmoments}, we conclude that this must be
$\limval_T(X_{1:k},\cD_\tdiag)$.
\end{proof}

\begin{proof}[Proof of Proposition \ref{prop:freeness}(b)]
By the universality established in Lemma~\ref{lemma:syminvmoments}, it suffices
to check that the limit of $\E[\val_T(\W;\x_{1:k})]$ for orthogonally invariant
matrices $\W$, as computed in the preceding Proposition
\ref{prop:orthogonallimval}, equals 0
under the given conditions.

The given condition $\frac{1}{n}\Tr \W \to 0$ implies $\E[D]=0$.
If $\pi_E \vee \pi'$ has any block containing only the two 
vertex-edge pairs for a single edge, then this implies
$D(\pi_E \vee \pi')=0$ in (\ref{eq:Dlimortho}). Otherwise, each
block must correspond to at least two edges, so $|\pi_E \vee \pi'| \leq w/2$.
Similarly, if $\pi_V \vee \pi$ is such that any block contains the
vertex-edge pairs for only a single vertex, then the condition
(\ref{eq:qnormalized}) implies $q(\pi)=0$ in (\ref{eq:qlimortho_appendix}). 
Otherwise,each block must correspond to at least two vertices,
so $|\pi_V \vee \pi| \leq (w+1)/2$. Thus if $q(\pi_V  \vee \pi)D(\pi_E \vee
\pi') \neq 0$, then
\[\frac{1}{n} \cdot n^{|\pi_V \vee \pi|} \cdot n^{|\pi_E \vee \pi'|}
\cdot n^{-w-d(\pi,\pi')}
\leq n^{-1+(w+1)/2+w/2-w} \leq n^{-1/2}.\]
Applying this to (\ref{eq:finalorthoval}), we get
$\E[\val_T(\W;\x_{1:k})] \to 0$ as desired.
\end{proof}

\section{Details for rectangular matrices}\label{appendix:rect}

In Appendix \ref{subsec:rectsetup}, we provide further details on the
Onsager corrections and state evolutions of AMP algorithms for rectangular
matrices. In Appendix \ref{subsec:rectinvariant}, we state and prove a result
analogous to Proposition \ref{prop:syminvariant} in the symmetric setting,
providing examples of matrices $\bW \in \R^{m \times n}$ that
satisfy the generalized invariance conditions of Definition
\ref{def:rectinvariant}.

The remaining subsections prove
Theorems \ref{thm:rect} and \ref{thm:rectinvariant} on AMP universality:
In Appendix \ref{subsec:recttensor}, we introduce a definition of tensor networks with
dimensions alternating between $m$ and $n$, and show that AMP universality may
be reduced to universality of the values of such tensor networks. Then, in
Appendices \ref{subsec:whitenoise} and \ref{subsec:rectinv}, we establish
universality of the tensor network values for the classes of generalized white
noise and rectangular generalized invariant matrices.

\subsection{Onsager corrections and state evolution}\label{subsec:rectsetup}

In an AMP algorithm,
the coefficients $\{a_{ts}\}$ and $\{b_{ts}\}$ in (\ref{eq:AMPrect}) are chosen
so that $\{\y_t\}$ and $\{\z_t\}$ are described by simple state evolutions in 
the asymptotic limit as $m, n \to \infty$ with $m/n \to \gamma \in (0,\infty)$.
When $\W$ has i.i.d.\ $\Normal(0,1/n)$ entries, this may be done by setting
$\bOmega_1=\gamma \cdot \E[U_1^2] \in \R^{1 \times 1}$ and iteratively defining
\begin{equation}\label{eq:whitenoisebSigma}
\bSigma_t=(\E[V_rV_s])_{r,s=1}^t \in \R^{t \times t}, \qquad
\bOmega_{t+1}=(\gamma \cdot \E[U_rU_s])_{r,s=1}^{t+1} \in \R^{(t+1) \times (t+1)},
\end{equation}
where we set $Z_{1:t} \sim \Normal(0,\bOmega_t)$ independent of $G_{1:\ell}$;
$V_s=v_s(Z_{1:s},G_{1:\ell})$ for each $s=1,\ldots,t$;
$Y_{1:t} \sim \Normal(0,\bSigma_t)$ independent of $(U_1,F_{1:k})$; and
$U_{s+1}=u_{s+1}(Y_{1:s},F_{1:k})$ for each $s=1,\ldots,t$. We then define
$a_{ts}$ and $b_{ts}$ as
\begin{equation}\label{eq:whitenoiseb}
a_{ts}=\E[\partial_s v_t(Z_{1:t},G_{1:\ell})],
\qquad b_{ts}=\gamma \cdot \E[\partial_s u_t(Y_{1:(t-1)},F_{1:k})].
\end{equation}
We call (\ref{eq:whitenoisebSigma}) and (\ref{eq:whitenoiseb}) the
\emph{white noise prescriptions} for $\bOmega_t,\bSigma_t$ and $a_{ts},b_{ts}$.
Results of \cite{bayati2011dynamics,javanmard2013state} imply the state
evolutions, for Lipschitz functions $v_t(\cdot)$ and $u_t(\cdot)$,
almost surely as $n \to \infty$ for any fixed $t \geq 1$,
\begin{align*}
(\u_1,\f_1,\ldots,\f_k,\y_1,\ldots,\y_t) &\toW
(U_1,F_1,\ldots,F_k,Y_1,\ldots,Y_t),\\
(\g_1,\ldots,\g_\ell,\z_1,\ldots,\z_t) &\toW
(G_1,\ldots,G_\ell,Z_1,\ldots,Z_t).
\end{align*}

For a bi-orthogonally invariant matrix $\W=\bO\bD\bQ^\top \in \R^{m
\times n}$, let $\bD=\diag(\bd) \in \R^{m \times n}$ be the matrix of singular
values where $\bd \in \R^{\min(m,n)}$, and define
\begin{equation}\label{eq:bard}
\bar{\bd}=\bd \in \R^m \text{ if } m \leq n,
\qquad \bar{\bd}=(\bd,0,\ldots,0) \in \R^m \text{ if } m>n
\end{equation}
Suppose that $\bar\bd \toW D$ as $m,n \to \infty$.
We refer to $D$ as the limit singular value distribution of $\bW$.
The above may then be extended as follows: Set $\bOmega_1=\gamma
\cdot \E[D^2] \cdot \E[U_1^2]$. For two continuous functions $\bSigma_t(\cdot)$
and $\bOmega_{t+1}(\cdot)$ whose forms depend only on $\gamma$ and the law of $D$,
define iteratively
\begin{subequations}\label{eq:rectorthobSigma}
\begin{align}
\bSigma_t&=\bSigma_t\Big(\{\E[V_rV_s],\E[U_rU_s]\}_{r,s \leq t},\nonumber\\
&\hspace{1in}\{\E[\partial_r v_s(Z_{1:s},G_{1:\ell})]\}_{r \leq s \leq t},
\{\E[\partial_r u_{s+1}(Y_{1:s},F_{1:k})]\}_{r \leq s<t}\Big),\\
\bOmega_{t+1}&=\bOmega_{t+1}\Big(\{\E[V_rV_s]\}_{r,s \leq t},
\{\E[U_rU_s]\}_{r,s \leq t+1},\nonumber\\
&\hspace{1in}\{\E[\partial_r v_s(Z_{1:s},G_{1:\ell})],
\E[\partial_r u_{s+1}(Y_{1:s},F_{1:k})]\}_{r \leq s \leq t}\Big).
\end{align}
\end{subequations}
Then for continuous functions $a_{ts}(\cdot)$ and $b_{ts}(\cdot)$ whose forms
also depend only on $\gamma$ and the law of $D$, define
\begin{subequations}\label{eq:rectorthob}
\begin{align}
a_{ts}&=a_{ts}\Big(\{\E[V_qV_r],\E[U_qU_r]\}_{q,r \leq t},\nonumber\\
&\hspace{1in}
\{\E[\partial_q v_r(Z_{1:r},G_{1:\ell})]\}_{q \leq r \leq t},
\{\E[\partial_q u_{r+1}(Y_{1:r},F_{1:k})]\}_{q \leq r<t}\Big),\\
b_{ts}&=b_{ts}\Big(\{\E[V_qV_r]\}_{q,r<t}, \{\E[U_qU_r]\}_{q,r \leq t},\nonumber\\
&\hspace{1in}\{\E[\partial_q v_r(Z_{1:r},G_{1:\ell})],
\E[\partial_q u_{r+1}(Y_{1:r},F_{1:k})]\}_{q \leq r<t}\Big).
\end{align}
\end{subequations}
We call (\ref{eq:rectorthobSigma}) and (\ref{eq:rectorthob}) the
\emph{bi-orthogonally invariant prescriptions} for $\bOmega_t,\bSigma_t$ and
$a_{ts},b_{ts}$, and we refer to \cite[Section 5]{fan2022approximate} for their
exact forms. For weakly differentiable functions
$u_t(\cdot),v_t(\cdot)$ with derivatives having at most polynomial growth,
it is shown in \cite{fan2022approximate}
that the iterates of (\ref{eq:AMPrect}) satisfy the state evolution,
almost surely as $m,n \to \infty$ for any fixed $t \geq 1$,
\begin{align*}
(\u_1,\f_1,\ldots,\f_k,\y_1,\ldots,\y_t) &\toW
(U_1,F_1,\ldots,F_k,Y_1,\ldots,Y_t),\\
(\g_1,\ldots,\g_\ell,\z_1,\ldots,\z_t) &\toW
(G_1,\ldots,G_\ell,Z_1,\ldots,Z_t).
\end{align*}

\subsection{Sufficient conditions for generalized
invariance}\label{subsec:rectinvariant}

\begin{proposition}\label{prop:rectinvariant}
Let $\bW \in \R^{m \times n}$ have singular values $\bd \in \R^{\min(m,n)}$,
such that $\bar\bd$ defined by (\ref{eq:bard}) satisfies
$\bar\bd \toW D$ almost surely as $m,n \to \infty$ with $m/n \to \gamma \in
(0,\infty)$. Suppose $D$ has finite moments of all orders.
\begin{enumerate}[(a)]
\item If $\bW$ is bi-orthogonally invariant in law, then $\bW$ is a
rectangular generalized invariant matrix in the sense of Definition
\ref{def:rectinvariant}, and its limit diagonal distribution $\cD_\tdiag$ is
determined uniquely by $\gamma$ and the law of $D$.
\item Suppose that either
\begin{enumerate}[1.]
\item $\W=\bO\bD\bQ^\top$ where $\bD=\diag(\bd) \in \R^{m \times n}$,
$\bO=\bPi_U\bH\bPi_E$, and $\bQ=\bPi_V\bK\bPi_F$. Here, $\bPi_U,\bPi_E \in
\R^{m \times m}$ and $\bPi_V,\bPi_F \in \R^{n \times n}$ are uniformly random
signed permutations independent of each other and of $(\bD,\bH,\bK)$, and
$\bH \in \R^{m \times m}$ and $\bK \in \R^{n \times n}$ are orthogonal matrices
whose entries satisfy (\ref{eq:symdelocalization})
for any fixed $\eps>0$, almost surely for all large $m,n$.
\item $\bW=\bPi_U\bM\bPi_V^\top$ such that $\bPi_U \in \R^{m \times m}$ and $\bPi_V
\in \R^{n \times n}$ are uniformly random signed permutations independent of
$\bM \in \R^{m \times n}$. For any integer $\nu \geq 0$ and any fixed $\eps>0$,
the matrices $(\bM\bM^\top)^\nu$, $(\bM^\top\bM)^\nu$,
and $(\bM\bM^\top)^\nu\bM$ satisfy
\begin{equation}\label{eq:Mcondition_rec}
\begin{aligned}
\max_{\alpha \in [m]} \big|(MM^\top)^\nu[\alpha,\alpha]-\tfrac{1}{m}\Tr
(\bM\bM^\top)^\nu\big|<n^{-1/2+\eps},&
\quad \max_{\alpha \neq \beta} \big|(MM^\top)^\nu[\alpha,\beta]\big|<n^{-1/2+\eps},\\
\max_{i \in [n]} \big|(M^\top M)^\nu[i,i]-\tfrac{1}{n}\Tr (\bM^\top\bM)^\nu\big|
<n^{-1/2+\eps},&
\quad \max_{i \neq j} \big|(M^\top M)^\nu[i,j]\big|<n^{-1/2+\eps},\\
\max_{\alpha \in [m],i \in [n]} \big|(MM^\top)^\nu
M[\alpha,i]\big|&<n^{-1/2+\eps}
\end{aligned}
\end{equation}
almost surely for all large $m,n$.
\end{enumerate}
Then $\bW$ is a rectangular generalized invariant matrix in the sense of
Definition \ref{def:rectinvariant}, and its limit diagonal distribution
$\cD_\tdiag$ coincides with that of the bi-orthogonally invariant matrix in part
(a).
\end{enumerate}
\end{proposition}

To prove Proposition~\ref{prop:rectinvariant}, we first slightly relax the 
assumptions in Lemma~\ref{lemma:syminvariantb2} to 
obtain the following lemma for the rectangular setting.

\begin{lemma}\label{lemma:recinvariantb2}
Let $\bM\in\RR^{m\times n}$ be a rectangular matrix having singular values 
$\bd\in\RR^{\min(m,n)}$.
Suppose $\bar\bd$ defined by (\ref{eq:bard}) satisfies
$\bar\bd \toW D$ almost surely as $m,n \to \infty$ with $m/n \to \gamma \in
(0,\infty)$, and $D$ has finite moments of all orders.
Suppose $\bM$ satisfies \eqref{eq:Mcondition_rec} almost surely for all large
$m,n$, and let $\widetilde{\bM}$ be its symmetric embedding in
(\ref{eq:Membedding}).
Then for any $p(\bx)\in\Delta\langle\bx,\bI_m,\bI_n\rangle$, 
\begin{enumerate}[(a)]
\item The limits
\[\lim_{m,n \to \infty} \frac{1}{m}\sum_{i=1}^m
p(\widetilde\bM)[i,i],
\qquad \lim_{m,n \to \infty} \frac{1}{n}\sum_{i=m+1}^{m+n}
p(\widetilde\bM)[i,i]\]
both exist almost surely, are finite, and depend only on $\gamma$
and the law of $D$.
\item For any $\eps>0$, almost surely for all large $m,n$,
\[\max_{1 \leq i \leq m}
\bigg|p(\widetilde\bM)[i,i]-\frac{1}{m}\sum_{j=1}^m
p(\widetilde\bM)[j,j]\bigg|<n^{-1/2+\eps},\]
\[\max_{m+1 \leq i \leq m+n} \bigg|p(\widetilde\bM)[i,i]-\frac{1}{n}
\sum_{j=m+1}^{m+n} p(\widetilde\bM)[j,j]\bigg|<n^{-1/2+\eps},\]
\[\max_{i \neq j \in [n+m]} |p(\widetilde\bM)[i,j]|<n^{-1/2+\eps}.\]
\end{enumerate}
\end{lemma}

\begin{proof}
The proof is similar to that of Lemma~\ref{lemma:syminvariantb2}. Analogously to
(\ref{eq:diag_monomial}), we may represent
\begin{align}\label{eq:diag_monomial_rec}
p(\x)=w_1(\x)\Delta(p_1(\x))w_2(\x)\Delta(p_2(\x))
\cdots w_L(\x)\Delta(p_L(\x))w_{L+1}(\x)
\end{align}
where each $w_\ell(\x)$ is a word in $\{\x,\bI_m,\bI_n\}$. We again define the depth
$\delta(p)$ as $\delta(p)=0$ if $L=0$ and $\delta(p)=1+\max_{\ell=1}^L
\delta(p_\ell)$ if $L \geq 1$, and we induct on the value of $\delta(p)$.
For the base case $\delta(p)=0$, i.e.\ $p(\x)=w_1(\x)$, we note that
$p(\widetilde\bM)$ must be either a block diagonal matrix or a block off-diagonal matrix, 
and in the first case, it must have its upper-left block equal to $(\bM\bM^\top)^\nu$ or 0,
lower-right block equal to $(\bM^\top\bM)^\nu$ or 0,
while in the second case, it must have upper-right block equal to $(\bM\bM^\top)^\nu\bM$ or 0, and
lower-left block equal to $\bM^\top(\bM\bM^\top)^\nu$ or 0 for some (possibly
different) integer values $\nu \geq 0$. Then claim (a) holds by the convergence
$\bar{\bd} \toW D$, and claim (b) holds by the
conditions in \eqref{eq:Mcondition_rec}.

Next, suppose inductively the lemma is true for all $p(\bx)\in\Delta\langle\bx,\bI_m,\bI_n\rangle$
such that $\delta(p)\leq K$, and consider any
$p(\bx)\in\Delta\langle\bx,\bI_m,\bI_n\rangle$ with $\delta(p)=K+1$,
of the form \eqref{eq:diag_monomial_rec}.
Fix any $i,j\in[m+n]$ and write $i_0\equiv i$ and $i_{L+1}\equiv j$. Then 
$p(\widetilde\bM)[i,j]$ may be expressed as
\begin{align*}
	p(\widetilde\bM)[i,j] &= \sum_{\bi\in[m+n]^L} 
	w_1(\widetilde \bM)[i_0,i_1] p_1(\widetilde\bM)[i_1,i_1] w_2(\widetilde
\bM)[i_1,i_2]
	\cdots p_{L}(\widetilde\bM)[i_{L},i_{L}] w_{L+1}(\widetilde \bM)[i_L,i_{L+1}].
\end{align*}
Let us abbreviate $\Tr_m,\Tr_n$ for the traces of the upper-left $m \times m$
and lower-right $n \times n$ submatrices, respectively. Then
by the induction hypothesis,
$p_\ell(\widetilde\bM)=\diag(\frac{1}{m}\Tr_m p_\ell(\widetilde\bM) \cdot \Id_m,
\frac{1}{n}\Tr_n p_\ell(\widetilde\bM) \cdot \Id_n)+\bE_\ell$,
where $\Id_m$ and 
$\Id_n$ are $m\times m$ and $n\times n$ identity matrices and $\bE_\ell$
satisfies $\max_{i,j\in[m+n]} |E_\ell[i,j]| < n^{-1/2+\eps}$
almost surely for all large $m,n$.

For $\cJ\subseteq[L]$, we write $[L]\setminus\cJ=\bar\cJ_1 \sqcup \bar\cJ_2$
where $\bar\cJ_1$ will contain indices of $\bi$ taking values in
$\{1,\ldots,m\}$ and 
$\bar\cJ_2$ will contain indices of $\bi$ taking values in $\{m+1,\ldots,m+n\}$.
Further define $\cI(\bar\cJ_1,\bar\cJ_2)=\{\bi\in[m+n]^L: i_\ell
\leq m \text{ for all } \ell \in \bar\cJ_1 \text{ and } i_\ell \geq
m+1 \text{ for all } \ell \in \bar\cJ_2\}$.
Then 
\begin{align*}
	&p(\widetilde\bM)[i,j]\\
&= \sum_{\bi\in[m+n]^L} 
	\prod_{\ell=1}^{L+1} w_\ell(\widetilde \bM)[i_{\ell-1},i_{\ell}]
	\prod_{\ell=1}^L \bigg(\frac{1}{m}\Tr_m p_\ell(\widetilde{\bM}) \cdot
\1_{1\leq i_\ell\leq m}
	+ \frac{1}{n}\Tr_n p_\ell(\widetilde{\bM}) \cdot
\1_{m+1\leq i_\ell\leq m+n}
	+ E_\ell[i_{\ell}, i_{\ell}]\bigg)\\
	&= \sum_{\cJ \sqcup \bar\cJ_1 \sqcup \bar\cJ_2=[L]}
	\bigg(\prod_{\ell\in\bar\cJ_1} \frac{1}{m}\Tr_m p_\ell(\widetilde\bM)\bigg)
	\bigg(\prod_{\ell\in\bar\cJ_2} \frac{1}{n} \Tr_n p_\ell(\widetilde\bM)\bigg)
	\sum_{\bi\in\cI(\bar\cJ_1,\bar\cJ_2)} \prod_{\ell=1}^{L+1}
	w_\ell(\widetilde \bM)[i_{\ell-1},i_\ell] \prod_{\ell \in \cJ} E_\ell[i_\ell,i_\ell].
\end{align*}
Similarly to \eqref{eq:MJdef}, by part (a) of the induction hypothesis
we have $\prod_{\ell\in\bar\cJ_1} \frac{1}{m}
\Tr_m p_\ell(\widetilde\bM) = M_{\bar\cJ_1} + \xi_{\bar\cJ_1}$ and 
$\prod_{\ell\in\bar\cJ_2} \frac{1}{n}
\Tr_n p_\ell(\widetilde\bM) = M_{\bar\cJ_2} + \xi_{\bar\cJ_2}$, where 
$M_{\bar\cJ_1}, M_{\bar\cJ_2}$ are finite limit values depending only on $\gamma$
and the law of $D$, and the convergence
$\xi_{\bar\cJ_1},\xi_{\bar\cJ_2} \to 0$ holds uniformly over
$i,j \in [m+n]$.
This gives
\begin{align}\label{eq:rec_pm}
	p(\widetilde\bM)[i,j] &= \sum_{\cJ \sqcup \bar\cJ_1\sqcup\bar\cJ_2=[L]}
	(M_{\bar\cJ_1} + \xi_{\bar\cJ_1}) (M_{\bar\cJ_2} + \xi_{\bar\cJ_2})
	\sum_{\bi\in\cI(\bar\cJ_1,\bar\cJ_2)} \prod_{\ell=1}^{L+1}
	w_\ell(\widetilde \bM)[i_{\ell-1},i_\ell] \prod_{\ell \in \cJ} E_\ell[i_\ell,i_\ell].
\end{align}

Next, let us write explicitly
$\cJ=\{\ell_1,\ldots,\ell_{|\cJ|}\}$ where $1 \leq \ell_1<\ldots<\ell_{|\cJ|}
\leq L$, and set $\ell_0=0$ and $\ell_{|\cJ|+1}=L+1$. 
We can contract the summation over indices 
$\{i_\ell:\ell\notin \cJ\}$, incorporating the constraints
$i_\ell \leq m$ for $\ell \in \bar\cJ_1$ and $i_\ell \geq m+1$ for $\ell \in
\bar\cJ_2$ by introducing copies of $\bI_m$ and $\bI_n$. For example, if
$\ell_{\rho-1},\ell_\rho \in \cJ$ where
$\ell_\rho=\ell_{\rho-1}+3$, and $\ell_{\rho-1}+1 \in \bar\cJ_1$
and $\ell_{\rho-1}+2 \in \bar\cJ_2$, then
\begin{align*}
&\sum_{i_{\ell_{\rho-1}+1}=1}^m \sum_{i_{\ell_{\rho-1}+2}=m+1}^{m+n}
w_{\ell_{\rho-1}+1}(\widetilde{\bM})[i_{\ell_{\rho-1}},i_{\ell_{\rho-1}+1}]
w_{\ell_{\rho-1}+2}(\widetilde{\bM})[i_{\ell_{\rho-1}+1},i_{\ell_{\rho-1}+2}]
w_{\ell_\rho}(\widetilde{\bM})[i_{\ell_{\rho-1}+2},i_{\ell_\rho}]\\
&=\big(w_{\ell_{\rho-1}+1}(\widetilde{\bM}) \cdot \bI_m \cdot
w_{\ell_{\rho-1}+2}(\widetilde{\bM}) \cdot \bI_n
\cdot w_{\ell_\rho}(\widetilde{\bM})\big)[i_{\ell_{\rho-1}},i_{\ell_\rho}].
\end{align*}
Applying this argument gives, for some new words
$w_1^{\bar\cJ_1,\bar\cJ_2}(\x),\ldots,w_{|\cJ|+1}^{\bar\cJ
_1,\bar\cJ_2}(\x)$ in $\{\x,\bI_m,\bI_n\}$
whose definitions depend on $\bar\cJ_1,\bar\cJ_2$,
\begin{align*}
	p(\widetilde\bM)[i,j] &= \sum_{\cJ \sqcup \bar\cJ_1 \sqcup \bar\cJ_2=[L]} \;
	(M_{\bar\cJ_1} + \xi_{\bar\cJ_1}) (M_{\bar\cJ_2} + \xi_{\bar\cJ_2})
	\sum_{\bi \in [m+n]^{\cJ}} \prod_{\rho=1}^{|\cJ|+1}
	w_{R_\rho}^{\bar\cJ_1,\bar\cJ_2}(\widetilde \bM)[i_{\ell_{\rho-1}},i_{\ell_\rho}] \prod_{\ell \in \cJ} 
	E_\ell[i_\ell,i_\ell].
\end{align*}
The same argument as in Lemma~\ref{lemma:syminvariantb2} then yields that this
is at most $Cn^{-1/2+\eps(2|\cJ|+1)}$ unless $i=j$ and $\cJ=\emptyset$. Thus,
\[\left|p(\widetilde{\bM})[i,j]-\1\{i=j\}
\sum_{\bar\cJ_1 \sqcup \bar\cJ_2=[L]} \;
(M_{\bar\cJ_1} + \xi_{\bar\cJ_1}) (M_{\bar\cJ_2} + \xi_{\bar\cJ_2})
w_{R_1}^{\bar\cJ_1,\bar\cJ_2}(\widetilde\bM)[i,i]\right| \leq
Cn^{-1/2+\eps(2L+1)}.\]
In the case $i=j \leq m$ or $i=j \geq m+1$, further approximating
$w_{R_1}^{\bar\cJ_1,\bar\cJ_2}(\widetilde\bM)[i,i]$
by $\frac{1}{m}\Tr_m w_{R_1}^{\bar\cJ_1,\bar\cJ_2}$
or $\frac{1}{n}\Tr_n w_{R_1}^{\bar\cJ_1,\bar\cJ_2}$ respectively
and applying the same argument as in Lemma~\ref{lemma:syminvariantb2},
we obtain the inductive claim (b), and averaging over $i=j \leq m$ and $i=j
\geq m+1$ then shows the inductive claim (a).
\end{proof}

\begin{proof}[Proof of Proposition~\ref{prop:rectinvariant}]
Lemma \ref{lemma:recinvariantb2} implies that 
Definition \ref{def:rectinvariant} holds for matrices satisfying the
description of part (b2), where the limit diagonal law $\cD_{\tdiag}$
depends only on $\gamma$ and the law of $D$.

In the setting of part (a) where $\bW=\bO\bD\bQ^\top$ is bi-orthogonally
invariant, we have the equalities in law $\bO\overset{L}{=}\bPi_U\bO\bPi_E$
and $\bQ\overset{L}{=}\bPi_V\bQ\bPi_F$ for uniformly
random signed permutations $\bPi_U,\bPi_E,\bPi_V,\bPi_F$ independent of each
other and of $\bO,\bQ$. The
entries of $\bO,\bQ$ satisfy (\ref{eq:symdelocalization}) almost
surely for all large $m,n$ by \cite[Theorem 1]{jiang2005maxima}, so $\bW$ is
an example of the matrix model in part (b1).

To conclude the proof,
it remains to verify that if $\bW$ is a matrix described by
part (b1), then it also is an example of the matrix model in part (b2).
For this, write $\bW=\bPi_U\bM \bPi_V^\top$ where
\[\bM=\bH \bPi_E \bD \bPi_F^\top \bK^\top.\]
Observe that $\bM^\top \bM=\bK \bPi_F \bD^\top\bD \bPi_F^\top
\bK^\top$ and $\bM\bM^\top=\bH\bPi_E \bD\bD^\top \bPi_E^\top \bH^\top$.
Then for any integer $\nu \geq 0$,
$(\bM^\top \bM)^\nu$ and $(\bM\bM^\top)^\nu$ satisfy the conditions of
(\ref{eq:Mcondition_rec}) by the proof of Proposition
\ref{prop:syminvariant} in the symmetric setting.

Now fix any integer $\nu \geq 0$ and consider
\[(\bM\bM^\top)^\nu \bM=\bH\bPi_E(\bD\bD^\top)^\nu \bD\bPi_F^\top \bK^\top.\]
We suppose for notational simplicity that $m \leq n$; the case $m \geq n$ is
analogous. Write $\bPi_E=\bXi_E\bP_E$ and $\bPi_F=\bXi_F\bP_F$ for
permutation and sign matrices defining $\bPi_E,\bPi_F$, and let
$\sigma_E,\sigma_F$ be the permutations of $[m],[n]$ such that
$P_E[\beta,\sigma_E(\beta)]=1$
and $P_F[j,\sigma_F(j)]=1$ for all $\beta \in [m]$ and $j \in [n]$.
Fixing $\alpha \in [m]$ and $i \in [n]$, let $\bh,\bk$ be the
$\alpha^\text{th}$ row of $\bH$ and $i^\text{th}$ row of $\bK$.
Let $\tilde{\bD}=(\bD\bD^\top)^\nu \bD \in \R^{m \times n}$. Then
\[\Big((\bM\bM^\top)^\nu \bM\Big)[\alpha,i]
=\sum_{\beta \in [m]} \sum_{j \in [n]} h[\beta] \Xi_E[\beta,\beta]
\tilde{D}[\sigma_E(\beta),\sigma_F(j)]\Xi_F[j,j]k[j],\]
so for any even power $p \geq 2$,
\begin{align*}
\E\left[\Big((\bM\bM^\top)^\nu \bM\Big)[\alpha,i]^p\right]
&=\sum_{\bbeta \in [m]^p} \sum_{\bj \in [n]^p}
h[\beta_1]\ldots h[\beta_p]k[j_1]\ldots k[j_p]
\E\Big[\Xi_E[\beta_1,\beta_1]\ldots \Xi_E[\beta_p,\beta_p]\Big]\\
&\quad\times \E\Big[\Xi_F[j_1,j_1]\ldots \Xi_F[j_p,j_p]\Big]
\E\Big[\tilde{D}[\sigma_E(\beta_1),\sigma_F(j_1)]
\ldots \tilde{D}[\sigma_E(\beta_p),\sigma_F(j_p)]\Big].
\end{align*}
Let $\cP$ be the lattice of partitions of $[p]$, and let $\pi(\bbeta),\pi(\bj)
\in \cP$ be those partitions induced by $\bbeta,\bj$.
Observe that $\E[\Xi_E[\beta_1,\beta_1]\ldots \Xi_E[\beta_p,\beta_p]]=1$ if
$\pi(\bbeta)$ is even (i.e.\ all blocks have even cardinality) and 0 otherwise,
and similarly for $\bXi_F$. Since
$\tilde \bD$ is diagonal, the product $\tilde
D[\sigma_E(\beta_1),\sigma_F(j_1)]\ldots \tilde D[\sigma_E(\beta_p),\sigma_F(j_p)]$
is zero unless $\sigma_E(\beta_a)=\sigma_F(j_a)$ for every $a=1,\ldots,p$,
which can occur only when $\pi(\bbeta)=\pi(\bj)$. Fixing $\bbeta,\bj$ for which
$\pi(\bbeta)=\pi(\bj)=\pi$ and supposing $m \leq n$, observe that
\[\E\Big[\tilde{D}[\sigma_E(\beta_1),\sigma_F(j_1)]
\ldots \tilde{D}[\sigma_E(\beta_p),\sigma_F(j_p)]\Big]
=\frac{(m-|\pi|)!}{m!} \cdot \frac{(n-|\pi|)!}{n!}
\sum_{\bgamma \in [m]^p:\pi(\bgamma)=\pi} \tilde d[\gamma_1]
\ldots \tilde d[\gamma_p]\]
where we set $\tilde d[\gamma]=\tilde D[\gamma,\gamma]$, because
\begin{itemize}
\item Given any $\sigma_E$,
the probability over $\sigma_F$ that
$\sigma_F(j_a)=\sigma_E(\beta_a)$ for every $a=1,\ldots,p$ is
$(n-|\pi|)!/n!$.
\item The expectation of $\tilde{D}[\sigma_E(\beta_1),\sigma_E(\beta_1)]
\ldots \tilde{D}[\sigma_E(\beta_p),\sigma_E(\beta_p)]$ over 
$\sigma_E$ is a uniform average over all $m!/(m-|\pi|)!$ relabelings of the
$|\pi|$ distinct indices of $\bbeta$.
\end{itemize}
Thus we have
\begin{align*}
\E\left[\Big((\bM\bM^\top)^\nu \bM\Big)[\alpha,i]^p\right]
&=\sum_{\text{even } \pi \in \cP}
\frac{(m-|\pi|)!}{m!} \cdot \frac{(n-|\pi|)!}{n!}
\sum_{\bbeta \in [m]^p:\pi(\bbeta)=\pi} h[\beta_1]\ldots h[\beta_p]\\
&\hspace{1in}\times \sum_{\bj \in [n]^p:\pi(\bj)=\pi} k[j_1]\ldots k[j_p]
\sum_{\bgamma \in [m]^p:\pi(\bgamma)=\pi} \tilde d[\gamma_1]\ldots \tilde
d[\gamma_p]
\end{align*}
Lemma \ref{lemma:distinctness} and the assumption $\bar \bd \toW D$ imply
$|\sum_{\bgamma \in [m]^p:\pi(\bgamma)=\pi} \tilde d[\gamma_1]\ldots \tilde
d[\gamma_p]| \leq Cn^{|\pi|} \leq Cn^{p/2}$, where $|\pi| \leq p/2$ because
$\pi$ is even. Applying (\ref{eq:symdelocalization}) for
$\bH$ and $\bK$, we have
\[n^{-|\pi|}
\sum_{\bbeta \in [m]^p:\pi(\bbeta)=\pi} |h[\beta_1]\ldots h[\beta_p]| \leq
Cn^{p(-1/2+\eps)}, \quad n^{-|\pi|}
\sum_{\bj \in [n]^p:\pi(\bj)=\pi} |k[j_1]\ldots k[j_p]| \leq
Cn^{p(-1/2+\eps)}.\]
Thus $\E[((\bM\bM^\top)^\nu \bM)[\alpha,i]^p] \leq Cn^{-p/2+2p\eps}$.
Choosing even $p \geq 2$ sufficiently large, this
implies by Markov's inequality and a union bound over all $\alpha \in [m]$
and $i \in [n]$ that
\[\max_{\alpha \in [m],i \in [n]} |((\bM\bM^\top)^\nu \bM)[\alpha,i]|
<n^{-1/2+3\eps}\]
almost surely for all large $m,n$. Then
$(\bM\bM^\top)^\nu \bM$ also satisfies \eqref{eq:Mcondition_rec}.
So any $\bW$ as described by part (b1) also satisfies the conditions of
part (b2), concluding the proof.
\end{proof}

\subsection{Reduction to tensor networks}\label{subsec:recttensor}

\begin{definition}
An \emph{alternating diagonal tensor network} $T=(\cU,\cV,\cE,\{p_u\}_{u \in \cU},
\{q_v\}_{v \in \cV})$ in $(k,\ell)$ variables is an undirected tree graph 
with vertices $\cU \sqcup \cV$ and edges $\cE \subset \cU \times \cV$ where
\begin{itemize}
\item Each edge connects a vertex $u \in \cU$ to a vertex $v \in \cV$.
\item Each $u \in \cU$ is labeled by a polynomial function $p_u:\R^k \to \R$.
\item Each $v \in \cV$ is labeled by a polynomial function $q_v:\R^\ell \to \R$.
\end{itemize}
The \emph{value} of $T$ on a rectangular matrix $\W \in \R^{m \times n}$
and vectors $\x_1,\ldots,\x_k \in \R^m$ and $\y_1,\ldots,\y_\ell \in \R^n$ is
\[\val_T(\W;\x_1,\ldots,\x_k;\y_1,\ldots,\y_\ell)
=\frac{1}{n}\sum_{\balpha \in [m]^{\cU}} \sum_{\bi \in [n]^{\cV}}
p_{\balpha|T} \cdot q_{\bi|T} \cdot W_{\balpha,\bi|T}\]
where, for each index tuple $\balpha=(\alpha_u:u \in \cU) \in [m]^{\cU}$ and
$\bi=(i_v:v \in \cV) \in [n]^{\cV}$,
\[p_{\balpha|T}=\prod_{u \in \cU} p_u(x_1[\a_u],\ldots,x_k[\a_u]),
\qquad q_{\bi|T}=\prod_{v \in \cV} q_v(y_1[i_v],\ldots,y_\ell[i_v]),
\qquad W_{\balpha,\bi|T}=\prod_{(u,v) \in \cE} W[\a_u,i_v].\]
\end{definition}

This value may be understood as:
\begin{enumerate}
\item Associating to vertices $u \in \cU$ and $v \in \cV$ the diagonal
tensors $\bT_u=\diag(p_u(\x_1,\ldots,\x_k)) \in \R^{m \times \ldots \times m}$
and $\bT_v=\diag(q_v(\y_1,\ldots,\y_\ell)) \in \R^{n \times \ldots \times n}$,
whose orders equal the degrees of $u$ and $v$ in the tree.
\item Associating to each edge the matrix $\W$.
\item Iteratively contracting all tensor-matrix-tensor products represented by
edges of the tree, where each product involves a tensor in dimension $m$, $\W
\in \R^{m \times n}$, and a tensor in dimension $n$.
\end{enumerate}
For example, if $\cU=\{1,3,\ldots,w-1,w+1\}$ and $\cV=\{2,4,\ldots,w\}$ for an
even integer $w$,
and $T$ is the line graph $1-2-\ldots-w-(w+1)$, then $\bT_1,\bT_{w+1} \in \R^m$
are vectors, $\bT_3,\bT_5,\ldots,\bT_{w-1} \in \R^{m \times m}$ and
$\bT_2,\bT_4,\ldots,\bT_w \in \R^{n \times n}$ are matrices of alternating
dimensions, and the value is
\[\val_T(\W;\x_1,\ldots,\x_k;\y_1,\ldots,\y_\ell)
=\frac{1}{n}\,
\bT_1^\top \W \bT_2 \W^\top \bT_3 \W \cdots \W \bT_w \W^\top \bT_{w+1}.\]

The following lemma is analogous to Lemma \ref{lemma:Wignercompare}, and reduces
the universality of AMP to the universality of values of alternating diagonal
tensor networks.

\begin{lemma}\label{lemma:rectcompare}
Let $\ub_1, \fb_1, \ldots, \fb_k \in \RR^m$ and $\gb_1, \ldots, \gb_\ell \in 
\RR^n$ satisfy Assumption~\ref{assump:ufgconvergence}.
Let $\bW, \bG \in \RR^{m \times n}$ be random matrices independent of 
$\ub_1, \fb_1, \ldots, \fb_k, \gb_1, \ldots, \gb_\ell$ such that
\begin{enumerate}
\item $\bG=\bO\bD\bQ^\top$ is a bi-orthogonally invariant matrix such that
$\bD=\diag(\bd)$ and $\bar \bd \toW D$, where $\bar \bd$ is defined by
(\ref{eq:bard}) and $D$ is a compactly supported limit law with $\E[D^2]>0$.
\item $\|\W\|_\op <C$ for a constant $C>0$, 
almost surely for all large $m, n$.

\item For every alternating diagonal tensor network $T$ in $(k+1,\ell)$
variables, almost surely as $m, n \to \infty$,
\[\val_T(\W;\u_1,\f_1,\ldots,\f_k;\g_1,\ldots,\g_\ell)
-\val_T(\G;\u_1,\f_1,\ldots,\f_k;\g_1,\ldots,\g_\ell) \to 0.\]
\end{enumerate}
Let $v_t:\R^{t+\ell} \to \R$ and $u_{t+1}:\R^{t+k} \to \R$ be continuous
functions satisfying the polynomial growth condition \eqref{eq:polygrowth} for some 
order $p \geq 1$, and are Lipschitz in their first $t$ arguments. Let $\{a_{ts}\}, 
\{b_{ts}\}, \{\bOmega_t\}$ and $\{\bSigma_t\}$ be defined by the bi-orthogonally 
invariant prescriptions in \eqref{eq:rectorthobSigma} and \eqref{eq:rectorthob}
for the limit law $D$, where each $\bOmega_t$ and $\bSigma_t$ is non-singular.
Then the iterates \eqref{eq:AMPrect} applied to $\bW$ satisfy, almost surely as 
$n \to \infty$ for any fixed $t \geq 1$,
\begin{align*}
(\u_1,\f_1,\ldots,\f_k,\y_1,\ldots,\y_t) &\toWtwo
(U_1,F_1,\ldots,F_k,Y_1,\ldots,Y_t)\\
(\g_1,\ldots,\g_\ell,\z_1,\ldots,\z_t) &\toWtwo
(G_1,\ldots,G_\ell,Z_1,\ldots,Z_t)
\end{align*}
where these limits have the same joint laws as described by the AMP state 
evolution for $\bG$.
\end{lemma}

Lemma \ref{lemma:rectcompare} may be proven in the same way as
Lemma \ref{lemma:Wignercompare}. For the same initialization
$\tilde\ub_1 = \ub_1$ and vectors of side 
information $\fb_1, \ldots, \fb_k, \gb_1, \ldots, \gb_\ell$ as in the given
Lipschitz AMP algorithm, we consider an auxiliary AMP algorithm with polynomial
non-linearities
\begin{align}\label{eq:aux_AMP_rect}
\begin{aligned}
\tilde\zb_t &= \bW^\top \tilde\ub_t - \sum_{s=1}^{t-1} \tilde b_{ts} \tilde\vb_s\\
\tilde\vb_t &= \tilde v_t(\tilde\zb_1, \ldots, \tilde\zb_t, \gb_1, \ldots, \gb_\ell)\\
\tilde\yb_t &= \bW \tilde\vb_t - \sum_{s=1}^t \tilde a_{ts} \tilde\ub_s\\
\tilde\ub_{t+1} &= \tilde{u}_{t+1}(\tilde\yb_1, \ldots, \tilde\yb_t, \fb_1,
\ldots, \fb_k)
\end{aligned}
\end{align}
Fixing $\eps>0$, this is defined such that
\begin{enumerate}
\item Each coefficient $\tilde a_{ts}$ and $\tilde b_{ts}$ is defined by the
polynomials $\{\tilde u_{t+1}(\cdot)\}$, $\{\tilde v_t(\cdot)\}$ and the 
bi-orthogonally invariant prescriptions (\ref{eq:rectorthob}).
\item Let $\widetilde \bOmega_t$, $\widetilde \bSigma_t$ be the bi-orthogonally
invariant prescriptions (\ref{eq:rectorthobSigma}), and let
$(U_1,F_{1:k},\widetilde Y_{1:t})$ and $(G_{1:\ell},\widetilde Z_{1:t})$ be the
corresponding state evolutions. Then each polynomial $\tilde u_{t+1}(\cdot)$ and
$\tilde v_t(\cdot)$ is chosen to satisfy
\begin{align*}
\E\big[\big(\tilde u_{t+1}(\widetilde Y_{1:t},F_{1:k})
-u_{t+1}(\widetilde Y_{1:t},F_{1:k})\big)^2\big]&<\eps\\
\E\big[\big(\tilde v_t(\widetilde Z_{1:t},G_{1:\ell})
- v_t(\widetilde Z_{1:t},G_{1:\ell})\big)^2\big]&<\eps
\end{align*}
\item For any fixed arguments $y_{1:(t-1)}$, $f_{1:k}$, $z_{1:(1-t)}$, and
$g_{1:\ell}$, the functions $y_t \mapsto \tilde u_{t+1}(y_{1:t},f_{1:k})$ and
$z_t \mapsto v_t(z_{1:t},g_{1:\ell})$ have non-linear dependences in $y_t$ and
$z_t$.
\end{enumerate}
Again we write the iterates as $\tilde\ub_t(\bW), \tilde\zb_t(\bW), 
\tilde\vb_t(\bW), \tilde\yb_t(\bW)$ if we want to specify that the algorithm 
is applied to $\bW$.
Assumption \ref{assump:ufgconvergence}, the given condition $\E[D^2]>0$,
and condition (3) above verify the conditions of \cite[Assumption
5.2]{fan2022approximate}. Then
\cite[Theorem 5.3]{fan2022approximate} applies to show
that, for each fixed $t \geq 1$, almost surely as $n \to \infty$,
\begin{equation}\label{eq:Grectconvergence}
(\u_1,\f_{1:k},\tilde \y_{1:t}(\G)) \toW (U_1,F_{1:k},\widetilde Y_{1:t}),
\qquad (\g_{1:\ell},\tilde \z_{1:t}(\G)) \toW (G_{1:\ell},\widetilde Z_{1:t})
\end{equation}
when this algorithm is applied to the bi-orthogonally invariant matrix $\G$.

The following lemma, analogous to
Lemma~\ref{lemma:polynomial_tensor_network_decomposition}, shows that empirical
averages of polynomial test functions evaluated on the iterates of this
polynomial AMP algorithm may be decomposed as a sum of values of alternating
diagonal tensor networks.

\begin{lemma}\label{lemma:rec:polyf_tensor_network_decomp}
Fix any $t\geq 1$ and let $\tilde \u_1,\tilde \z_1, \tilde \v_1,
\tilde \y_1, \ldots, \tilde \u_t,\tilde \z_t, \tilde \v_t,
\tilde \y_t$ be the iterates of any algorithm of the
form \eqref{eq:aux_AMP_rect}, where $\{\tilde{a}_{ts},\tilde{b}_{ts}\}$ are
scalar constants and $\tilde{u}_{t+1}:\R^{t+k} \to \R$ and
$\tilde{v}_{t+1}:\R^{t+\ell} \to \R$ are polynomial functions applied row-wise.
For any polynomials $p:\RR^{2t+k}\to \RR$ and $q:\RR^{2t+l}\to\RR$, and for 
two finite sets $\cF_1$ and $\cF_2$ of alternating diagonal tensor networks
in $(k+1,\ell)$ variables,
\begin{align}
\langle p(\tilde\bu_1, \ldots, \tilde\ub_t, \tilde\by_1, \ldots, \tilde\yb_t, 
\bff_1, \ldots, \fb_k) \rangle 
&= \sum_{T \in \cF_1} \val_{T}(\bW; \bu_1, \bff_1, \ldots, \bff_k;
\gb_1, \ldots, \gb_\ell),\label{eq:pureduction}\\
\langle q(\tilde\vb_1, \ldots, \tilde\vb_t, \tilde\zb_1, \ldots, \tilde\zb_t, 
\gb_1, \ldots, \gb_\ell)\rangle
&= \sum_{T \in \cF_2} \val_T(\bW; \bu_1, \bff_1, \ldots, \bff_k;
\gb_1, \ldots, \gb_\ell).\label{eq:pvreduction}
\end{align}
\end{lemma}
\begin{proof}[Proof of Lemma~\ref{lemma:rec:polyf_tensor_network_decomp}]
We have
\begin{align*}
\langle p(\tilde \bu_{1:t},\tilde \by_{1:t}, \bff_{1:k}) \rangle 
=\val_{T}(\bW;\tilde \bu_{1:t},\tilde \by_{1:t},\bff_{1:k})
\end{align*}
where $T$ is a tensor network with only one vertex $\{u\} = \mathcal{U}$
and polynomial label $p_u = p$. 

We claim that for any tensor network $T$ in variables
$(\tilde y_{1:t},\tilde u_{1:{t}},f_{1:k};\tilde z_{1:t},\tilde v_{1:t},
g_{1:\ell})$, there exists a set of tensor networks $\cF$ such that
\begin{align}\label{eq:rec:reducey}
\val_{T}(\bW;\tilde\by_{1:t},\tilde\bu_{1:t},\tilde\bff_{1:k};
\tilde\bz_{1:t},\tilde\bv_{1:t},\tilde\bg_{1:\ell})
= \sum_{T'\in\cF}
\val_{T'}(\bW;\tilde\by_{1:(t-1)},\tilde\bu_{1:{t}},\tilde\bff_{1:k};
\tilde\bz_{1:t},\tilde\bv_{1:t},\tilde\bg_{1:\ell}).
\end{align}
To show this claim, recall $\tilde\by_t = \bW \tilde\bv_t - \sum_{s=1}^t \tilde a_{ts}
\tilde\bu_s$. Expanding each polynomial $p_u$ for $u \in \cU$
in terms of $(\tilde\by_{1:(t-1)},\tilde\bu_{1:t}, \bff_{1:k})$ and 
$\bW\tilde\bv_t$, we have 
\begin{align*}
p_u(\tilde y_{1:t}[\alpha],\tilde u_{1:t}[\alpha], f_{1:k}[\alpha]) 
= \sum_{\theta=0}^{\Theta_u} p_{u,\theta}(\tilde y_{1:(t-1)}[\alpha], 
\tilde u_{1:t}[\alpha], f_{1:k}[\alpha]) \cdot
\bigg(\sum_{i=1}^n W[\alpha,i] \tilde v_t[i]\bigg)^\theta
\end{align*}
for some polynomials $p_{u,0},\ldots,p_{u,\Theta_u}$, where $\Theta_u$ is the 
maximum degree of $p_u$ in $\tilde y_t$. Then we may write
\begin{align*}
&\val_{T}(\bW;\tilde \by_{1:t},\tilde \bu_{1:{t}},\bff_{1:k};\tilde
\bz_{1:t},\tilde \bv_{1:t},\bg_{1:\ell})\\
&\qquad = \frac{1}{n} \sum_{\balpha\in[m]^\cU} \sum_{\bi\in[n]^\cV} 
p_{\balpha|T} \cdot q_{\ib|T} \cdot W_{\balpha,\ib|T}\\
&\qquad =  \frac{1}{n} \sum_{\btheta \in \prod_{u \in \cU} \{0,\ldots,\Theta_u\}}
\sum_{\balpha\in[m]^{\cU}} \sum_{\bi\in[n]^{\cV}} 
\bigg(\prod_{u\in\cU} p_{u,\theta}(\tilde y_{1:(t-1)}[\alpha_u],
\tilde u_{1:t}[\alpha_u],f_{1:k}[\alpha_u])\bigg)\\
&\qquad\qquad \cdot q_{\bi|T} \cdot
\bigg(\sum_{i=1}^n W[\alpha_u,i] \tilde v_t[i]\bigg)^{\theta_u}
\cdot W_{{\balpha},\bi|T}.
\end{align*}
For each $\btheta \in \prod_{u\in\cU}\{0,\ldots,\Theta_u\}$,
we define a new tensor network $T_{\btheta}$ from $T$ as follows: For
each $u\in\cU$, replace the associated polynomial $p_u$ by $p_{u,\theta_u}$. 
Then, to each $u \in \cU$, add $\theta_u$ new edges connecting to $\theta_u$ new 
vertices in $\cV$, where each new vertex $v$ has the label
$q_v(\tilde z_{1:t}, \tilde v_{1:t}, g_{1:\ell}) = \tilde v_t$. 
Then the above is exactly
\begin{align*}
\val_T(\bW;\tilde\by_{1:t},\tilde\bu_{1:{t}},\bff_{1:k};
\tilde\bz_{1:t},\tilde\bv_{1:t},\bg_{1:\ell})
 = \sum_{\btheta\in\prod_{u\in\cU}\{0,\ldots,\Theta_u\}}
\val_{T_{\btheta}}(\bW; \tilde\by_{1:(t-1)}, \tilde\bu_{1:t}, \bff_{1:k};
\tilde\bz_{1:t},\tilde\bv_{1:t},\bg_{1:\ell})
\end{align*}
which verifies~\eqref{eq:rec:reducey}.

Next, for any tensor network $T$ in the variables $(\tilde
y_{1:(t-1)},\tilde u_{1:t},f_{1:k};\tilde z_{1:t},\tilde v_{1:t},g_{1:\ell})$,
applying $\tilde v_t=\tilde v_t(\tilde z_{1:t},g_{1:\ell})$ in the
polynomial label $p_v$ for each vertex $v \in \cV$, there exists a
tensor network $T'$ such that
\[\val_T(\bW;\tilde \by_{1:(t-1)},\tilde
\bu_{1:t},\bff_{1:k};\tilde\bz_{1:t},\tilde\bv_{1:t},\bg_{1:\ell})
=\val_{T'}(\bW;\tilde\by_{1:(t-1)},\tilde\bu_{1:t},\bff_{1:k};
\tilde\bz_{1:t},\tilde\bv_{1:(t-1)},\bg_{1:\ell}).\]
Since $\tilde\zb_t = \bW^\top \tilde\ub_t - \sum_{s=1}^{t-1} \tilde b_{ts}
\tilde \vb_s$, by the same argument as in the above for $\tilde\yb_t$,
for any tensor network $T$ in variables
$(\tilde y_{1:(t-1)},\tilde u_{1:t},f_{1:k};\tilde z_{1:t},\tilde v_{1:(t-1)},
g_{1:\ell})$, there exists a set of tensor networks $\cF$ such that
\[\val_T(\bW;\tilde\by_{1:(t-1)},\tilde\bu_{1:t},\tilde\bff_{1:k};
\tilde\bz_{1:t},\tilde\bv_{1:(t-1)},\tilde\bg_{1:\ell})
= \sum_{T'\in\cF}
\val_{T'}(\bW;\tilde\by_{1:(t-1)},\tilde\bu_{1:t},\tilde\bff_{1:k};
\tilde\bz_{1:(t-1)},\tilde\bv_{1:(t-1)},\tilde\bg_{1:\ell}),\]
and applying the fact that $\tilde u_t(\cdot)$ is a polynomial function, for any 
tensor network $T$ in the variables $(\tilde y_{1:(t-1)},\tilde u_{1:t},f_{1:k}; 
\tilde z_{1:(t-1)},\tilde v_{1:(t-1)},g_{1:\ell})$, there exists a tensor 
network $T'$ such that
\[\val_T(\bW;\tilde\by_{1:(t-1)},\tilde\bu_{1:t},\bff_{1:k};
\tilde\bz_{1:(t-1)},\tilde\bv_{1:(t-1)},\bg_{1:\ell})
=\val_{T'}(\bW;\tilde\by_{1:(t-1)},\tilde\bu_{1:(t-1)},\bff_{1:k};
\tilde\bz_{1:(t-1)},\tilde\bv_{1:(t-1)},\bg_{1:\ell}).\]
Iteratively applying these four reductions shows (\ref{eq:pureduction}), and the
proof of (\ref{eq:pvreduction}) is analogous.
\end{proof}

The remainder of the proof of Lemma \ref{lemma:rectcompare} parallels that of
Lemma \ref{lemma:Wignercompare}: As in Lemma \ref{lem:polyLawLimit},
the above result together with (\ref{eq:Grectconvergence})
and the given condition for universality of tensor
network values between $\W$ and $\G$ implies the almost-sure Wasserstein 
convergence
\[(\u_1,\f_{1:k},\tilde \y_{1:t}(\W)) \toW (U_1,F_{1:k},\widetilde Y_{1:t}),
\qquad (\g_{1:\ell},\tilde \z_{1:t}(\W)) \toW (G_{1:\ell},\widetilde Z_{1:t})\]
also for the polynomial AMP algorithm applied to $\W$. A similar inductive
polynomial approximation argument as in Lemma \ref{lemma:AMPpolyAPprox},
using the conditions that $\|\W\|_\op<C$ almost surely for all large $m,n$ and
$u_{t+1}(\cdot)$ and $v_t(\cdot)$ are Lipschitz, establishes
\[\max_{s=1}^t \frac{1}{\sqrt{n}} \|\zb_t(\W) - \tilde\zb_t(\W)\|_2<\iota(\eps),
\qquad
\max_{s=1}^t \frac{1}{\sqrt{n}} \|\yb_t(\W) - \tilde\yb_t(\W)\|_2<\iota(\eps),\]
\[\|\bOmega_t - \widetilde{\bOmega}_t\|_\op<\iota(\eps),
\qquad \|\bSigma_t - \widetilde{\bSigma}_t\|_\op<\iota(\eps)\]
for a constant $\iota(\eps) \to 0$ as $\eps \to 0$. This then implies the
desired $W_2$-convergence in Lemma \ref{lemma:rectcompare}.
We omit the details of the argument for brevity.

\subsection{Universality for generalized white noise
matrices}\label{subsec:whitenoise}

In this section, we prove the following result showing that the value of an
alternating diagonal tensor network is universal for the class of generalized
white noise matrices in Definition \ref{def:rect}.

\begin{lemma}\label{lemma:rectmoments}
Let $\x_1,\ldots,\x_k \in \R^m$ and $\y_1,\ldots,\y_\ell \in \R^n$ be (random or
deterministic) vectors and let $(X_1,\ldots,X_k)$ and $(Y_1,\ldots,Y_\ell)$ have
finite moments of all orders, such that almost surely as $m,n \to \infty$
with $m/n \to \gamma \in (0,\infty)$,
\[(\x_1,\ldots,\x_k) \toW (X_1,\ldots,X_k) \qquad \text{ and } 
\qquad (\y_1,\ldots,\y_\ell) \toW (Y_1,\ldots,Y_\ell).\]
Let $\W \in \R^{m \times n}$ be a generalized white noise matrix,
independent of $\x_1,\ldots,\x_k$ and $\y_1,\ldots,\y_\ell$, with variance
profile $\S$. Let $\s_\a$ be the $\a^\text{th}$ row of $\S$, let
$\s^i$ be the $i^\text{th}$ column of $\S$, and suppose for any fixed polynomial
functions $p:\R^k \to \R$ and $q:\R^\ell \to \R$ that
\begin{align}
\max_{i=1}^n \Big| \langle p(\x_1,\ldots,\x_k) \odot \s^i \rangle
-\langle p(\x_1,\ldots,\x_k) \rangle \cdot \langle \s^i \rangle \Big| &\to 0,
\label{eq:rectpsavg}\\
\max_{\a=1}^m \Big| \langle q(\y_1,\ldots,\y_\ell) \odot \s_\a \rangle
-\langle q(\y_1,\ldots,\y_\ell) \rangle \cdot \langle \s_\a \rangle \Big| &\to
0.\label{eq:rectqsavg}
\end{align}
Then for any alternating diagonal tensor network $T$ in $(k,\ell)$ variables,
there is a deterministic limit value
\begin{align*}
    \limval_T(X_1,\ldots,X_k,Y_1,\ldots,Y_\ell)
\end{align*}
depending only on $T$, $\gamma$, and the joint laws of $(X_1,\ldots,X_k)$ and
$(Y_1,\ldots,Y_\ell)$ such that almost surely,
\[\lim_{m,n \to \infty} \val_T(\W;\x_1,\ldots,\x_k;\y_1,\ldots,\y_\ell)
=\limval_T(\gamma,X_1,\ldots,X_k,Y_1,\ldots,Y_\ell).\]
In particular, this limit value is the same for $\W$ as for a Gaussian white
noise matrix $\G$.
\end{lemma}

The proof of Lemma \ref{lemma:rectmoments} is quite similar to that of Lemma
\ref{lemma:Wignermoments}, and we highlight only the differences: Fix the tensor
network $T=(\cU,\cV,\cE,\{p_u\}_{u \in \cU},\{q_v\}_{v \in \cV})$.
Let $\cP=\cP_\cU \times \cP_\cV$ be the set of pairs of partitions
$(\pi_U,\pi_V)$ of $\cU$ and $\cV$. For index
tuples $\balpha \in [m]^\cU$ and $\bi \in [n]^{\cV}$, consider their
induced partitions $\pi(\balpha) \in \cP_\cU$ and $\pi(\bi) \in \cP_\cV$. Then
\[\val_T(\W;\x_{1:k};\y_{1:\ell})
=\frac{1}{n}\sum_{(\pi_U,\pi_V) \in \cP}
\sum_{\balpha \in [m]^\cU:\pi(\balpha)=\pi_U}
\sum_{\bi \in [n]^{\cV}:\pi(\bi)=\pi_V} p_{\balpha|T} \cdot q_{\bi|T}
\cdot W_{\balpha,\bi|T}.\]
Corresponding to $(\pi_U,\pi_V) \in \cP$,
let $G_{\pi_U,\pi_V}=(\cK_{\pi_U} \sqcup \cK_{\pi_V},\cF_{\pi_U,\pi_V})$
be the image of $(\cU,\cV,\cE)$ under $(\pi_U,\pi_V)$: This is the bipartite
multi-graph with vertices $\cK_{\pi_U} \sqcup
\cK_{\pi_V}$ being the blocks of $\pi_U$ and $\pi_V$, and having one
edge $(U,V) \in \cF_{\pi_U,\pi_V}$ for each edge $(u,v) \in \cE$ where $U,V$
are the blocks containing $u,v$. Denote
$P_U=\prod_{u \in U} p_u$ and $Q_V=\prod_{v \in V} q_v$. Then this gives
\begin{equation}\label{eq:rectval1}
\val_T(\W;\x_{1:k};\y_{1:\ell})
=\frac{1}{n}\sum_{(\pi_U,\pi_V) \in \cP}
\sum_{\balpha \in [m]^{\cK_{\pi_U}}}^* \sum_{\bi \in [n]^{\cK_{\pi_V}}}^*
P_{\balpha|G_{\pi_U,\pi_V}} \cdot Q_{\bi|G_{\pi_U,\pi_V}} \cdot
W_{\balpha,\bi|G_{\pi_U,\pi_V}}^e
\end{equation}
where
\[P_{\balpha|G_{\pi_U,\pi_V}}=\prod_{U \in \cK_{\pi_U}} P_U(x_{1:k}[\a_U]),
\quad Q_{\bi|G_{\pi_U,\pi_V}}=\prod_{V \in \cK_{\pi_V}} Q_V(y_{1:\ell}[i_V]),\]
\[W_{\balpha,\bi|G_{\pi_U,\pi_V}}^e=\prod_{\text{unique edges } (U,V) \text{ of
} \cG_{\pi_U,\pi_V}} W[\a_U,i_V]^{e(U,V)}\]
and $e(U,V)$ is the number of times the unique
edge $(U,V)$ appears in $\cF_{\pi_U,\pi_V}$.

\begin{lemma}\label{lemma:Erectmoments}
Let $\E$ denote the expectation over $\W$ conditional on
$\x_1,\ldots,\x_k,\y_1,\ldots,\y_\ell$. Then Lemma~\ref{lemma:rectmoments} holds for
$\E[\val_T(\W;\x_1,\ldots,\x_k;\y_1,\ldots,\y_\ell)]$ in place of
$\val_T(\W;\x_1,\ldots,\x_k;\y_1,\ldots,\y_\ell)$.
\end{lemma}
\begin{proof}
Taking expectation on both sides of (\ref{eq:rectval1}), observe that any
summand has 0 expectation unless each unique edge of $G_{\pi_U,\pi_V}$ appears
at least twice. For the partitions $(\pi_U,\pi_V)$ corresponding to
summands with non-zero expectation, the numbers
of vertices and unique edges of $G_{\pi_U,\pi_V}$ must satisfy
\[|\cK_{\pi_U}|+|\cK_{\pi_V}| \leq |\cF_{\pi_U,\pi_V}|_*+1 \leq |\cE|/2+1,\]
the first inequality holding because the graph $G_{\pi_U,\pi_V}$ is connected.
By the same argument as in Lemma \ref{lem:Wignermoments_expected_value},
the summands where
$|\cK_{\pi_U}|+|\cK_{\pi_V}| \leq |\cE|/2$ have vanishing
contributions in the limit as $m,n \to \infty$. Thus the only non-vanishing
contributions come from partitions $(\pi_U,\pi_V)$ where
\begin{equation}\label{eq:rectnonvanishing}
|\cK_{\pi_U}|+|\cK_{\pi_V}|=|\cF_{\pi_U,\pi_V}|_*+1=|\cE|/2+1.
\end{equation}

For such classes, each unique edge of $G_{\pi_U,\pi_V}$ appears exactly twice by
the second equality of (\ref{eq:rectnonvanishing}),
these edges form a tree by the first equality of (\ref{eq:rectnonvanishing}),
and we have
\begin{align*}
&\frac{1}{n} \sum_{\balpha \in [m]^{\cK_{\pi_U}}}^* \sum_{\bi \in
[n]^{\cK_{\pi_V}}}^* P_{\balpha|G_{\pi_U,\pi_V}} \cdot Q_{\bi|G_{\pi_U,\pi_V}}
\cdot \E[W_{\balpha,\bi|G_{\pi_U,\pi_V}}^e]\\
&\qquad=\frac{1}{n} \sum_{\balpha \in [m]^{\cK_{\pi_U}}}^* \sum_{\bi \in
[n]^{\cK_{\pi_V}}}^* P_{\balpha|G_{\pi_U,\pi_V}} \cdot Q_{\bi|G_{\pi_U,\pi_V}}
\cdot \prod_{\text{unique edges } (U,V) \text{ of } G_{\pi_U,\pi_V}}
\frac{S[\alpha_U,i_V]}{n}.
\end{align*}
Applying Lemma \ref{lemma:distinctness} and the first equality of
(\ref{eq:rectnonvanishing}), this has the same asymptotic limit as 
\[\frac{1}{n} \sum_{\balpha \in [m]^{\cK_{\pi_U}}} \sum_{\bi \in
[n]^{\cK_{\pi_V}}} P_{\balpha|G_{\pi_U,\pi_V}} \cdot Q_{\bi|G_{\pi_U,\pi_V}}
\cdot \prod_{\text{unique edges } (U,V) \text{ of } G_{\pi_U,\pi_V}}
\frac{S[\alpha_U,i_V]}{n}\]
which removes the distinctness requirement for the indices of $\balpha$ and
$\bi$. Applying the conditions
(\ref{eq:rectpsavg}--\ref{eq:rectqsavg}) and the same argument as in
Lemma \ref{lem:VarianceProfile_universality}, this quantity has
the limit, as $m,n \to \infty$ with $m/n \to \gamma$,
\[\gamma^{|\cK_{\pi_U}|} \prod_{U \in \cK_{\pi_U}} \E[P_U(X_{1:k})]
\prod_{V \in \cK_{\pi_V}} \E[Q_V(Y_{1:\ell})]\]
(A factor of $\gamma$ arises when summing over each vertex of
$\cK_{\pi_U}$ because $S[\alpha_U,i_V]$ is normalized by $n$ but the
summation is over $[m]$.) Thus
\begin{align*}
&\lim_{m,n \to \infty} \E[\val_T(\W;\x_{1:k};\y_{1,\ell})]\\
&=\mathop{\sum_{(\pi_U,\pi_V) \in \cP}}_{|\cK_{\pi_U}|+|\cK_{\pi_V}|
=|\cF_{\pi_U,\pi_V}|_*+1=|\cE|/2+1}
\gamma^{|\cK_{\pi_U}|} \prod_{U \in \cK_{\pi_U}} \E[P_U(X_{1:k})]
\prod_{V \in \cK_{\pi_V}} \E[Q_V(Y_{1:\ell})]\\
&=:\limval_T(\gamma,X_1,\ldots,X_k,Y_1,\ldots,Y_\ell)
\end{align*}
This limit depends only on $T$, $\gamma$, and the joint laws of $X_{1:k}$ and
$Y_{1:\ell}$, concluding the proof.
\end{proof}

\begin{lemma}\label{lemma:rectfourthmoment}
Let $\E$ denote the expectation over $\W$, conditional on $\x_1,\ldots,\x_k,\y_1,\ldots,\y_\ell$. Under the setting of Lemma \ref{lemma:rectmoments}, almost
surely as $m,n \to \infty$,
\[\val_T(\W;\x_1,\ldots,\x_k;\y_1,\ldots,\y_\ell)
-\E[\val_T(\W;x_1,\ldots,\x_k;\y_1,\ldots,\y_\ell)] \to 0.\]
\end{lemma}
\begin{proof}
The proof is similar to Lemma \ref{lem:Wignermoments_concentration}.
Let $\W^{(1)},\W^{(2)},\W^{(3)},\W^{(4)}$ be four independent copies of $\W$,
define index tuples $\balpha_1,\balpha_2,\balpha_3,\balpha_4 \in [m]^\cU$
and $\bi_1,\bi_2,\bi_3,\bi_4 \in [n]^\cV$, and set
\[p_{\balpha_{1:4}}=p_{\balpha_1|T}p_{\balpha_2|T}p_{\balpha_3|T}p_{\balpha_4|T},
\quad q_{\bi_{1:4}}=q_{\bi_1|T}q_{\bi_2|T}q_{\bi_3|T}q_{\bi_4|T},\]
\[W_{\balpha_{1:4},\bi_{1:4}}^{(a_1,a_2,a_3,a_4)}
=W_{\balpha_1,\bi_1|T}^{(a_1)}W_{\balpha_2,\bi_2|T}^{(a_2)}
W_{\balpha_3,\bi_3|T}^{(a_3)}W_{\balpha_4,\bi_4|T}^{(a_4)}.\]
For index tuples $\balpha_{1:4}$ and $\bi_{1:4}$, consider a bi-partitite
multigraph $G(\balpha_{1:4},\bi_{1:4})=(\cU_G,\cV_G,\cE_G)$ whose vertices
$\cU_G$ are the unique index values in $\balpha_{1:4}$ and vertices $\cV_G$
are the unique index values in $\bi_{1:4}$, and having one edge
$(\a_{a,u},i_{a,v})$ for every combination of $a=1,2,3,4$ and edge $(u,v) \in
\cE$. Define
\begin{align*}
\cI_2&=\{(\balpha_{1:4},\bi_{1:4}):G(\balpha_{1:4},\bi_{1:4}) \text{ has } 1
\text{ or } 2 \text{ connected components\}},\\
\cI_3&=\{(\balpha_{1:4},\bi_{1:4}):G(\balpha_{1:4},\bi_{1:4}) \text{ has } 3
\text{ connected components\}},\\
\cI_4&=\{(\balpha_{1:4},\bi_{1:4}):G(\balpha_{1:4},\bi_{1:4}) \text{ has } 4
\text{ connected components\}},
\end{align*}
and define correspondingly for each $j=2, 3, 4$
\[A_j=\frac{1}{n^4}\sum_{(\balpha_{1:4},\bi_{1:4}) \in \cI_j}
p_{\balpha_{1:4}}q_{\bi_{1:4}}\Big(\E[W_{\balpha_{1:4},\bi_{1:4}}^{(1,1,1,1)}]
-4 \cdot \E[W_{\balpha_{1:4},\bi_{1:4}}^{(1,1,1,2)}]
+6 \cdot \E[W_{\balpha_{1:4},\bi_{1:4}}^{(1,1,2,3)}]
-3 \cdot \E[W_{\balpha_{1:4},\bi_{1:4}}^{(1,2,3,4)}]\Big).\]
Then $\E[(\val(\W)-\val(\W))^4]=A_2+A_3+A_4$.
By the same argument as in Lemma \ref{lem:Wignermoments_concentration},
we have $A_3=A_4=0$, while $|A_2| \leq C/n^2$ for a constant $C>0$ and all large
$n$. Then the result follows from Markov's inequality and the Borel-Cantelli
lemma.
\end{proof}

Lemmas \ref{lemma:Erectmoments} and \ref{lemma:rectfourthmoment} show Lemma
\ref{lemma:rectmoments}. Finally, combining Lemmas \ref{lemma:rectcompare} 
and \ref{lemma:rectmoments} concludes the proof of Theorem \ref{thm:rect}.

\subsection{Universality for rectangular generalized invariant
matrices}\label{subsec:rectinv}

In this section, we prove the following result showing that the value of an
alternating diagonal tensor network is universal across the class of rectangular
generalized invariant matrices.

\begin{lemma}\label{lemma:rectinvmoments}
Let $\x_1,\ldots,\x_k \in \R^m$ and $\y_1,\ldots,\y_\ell \in \R^n$ be (random or
deterministic) vectors and let $(X_1,\ldots,X_k)$ and $(Y_1,\ldots,Y_\ell)$ have
finite moments of all orders, such that almost surely as
$m,n \to \infty$ with $m/n \to \gamma \in (0,\infty)$,
\[(\x_1,\ldots,\x_k) \toW (X_1,\ldots,X_k) \qquad
\text{ and } \qquad (\y_1,\ldots,\y_\ell) \toW (Y_1,\ldots,Y_\ell).\]
Let $\W \in \R^{m \times n}$ be a rectangular generalized invariant
matrix independent of $\x_1,\ldots,\x_k$ and $\y_1,\ldots,\y_\ell$,
with limit diagonal distribution $\cD_\tdiag$.
Then for any alternating
diagonal tensor network $T$ in $(k,\ell)$ variables, there is a deterministic
limit value
\[\limval_T(\gamma,X_1,\ldots,X_k,Y_1,\ldots,Y_\ell,\cD_\tdiag)\]
depending only on $T$, $\gamma$, the joint laws of $(X_1,\ldots,X_k)$ and
$(Y_1,\ldots,Y_\ell)$, and $\cD_\tdiag$, such that almost surely
\[\lim_{m,n \to \infty} \val_T(\W;\x_1,\ldots,\x_k;\y_1,\ldots,\y_\ell)
=\limval_T(\gamma, X_1,\ldots,X_k,Y_1,\ldots,Y_\ell,\cD_\tdiag).\]
In particular, if there exists a bi-orthogonally invariant matrix $\bG$
having the same limit diagonal distribution $\cD_\tdiag$, then this 
limit value is the same for $\bW$ as for $\bG$.
\end{lemma}

The proof of Lemma~\ref{lemma:rectinvmoments} is similar to 
Lemma \ref{lemma:Esyminvmoments}. Fix the tensor network
$T=(\cU,\cV,\cE,\{p_u\}_{u \in \cU},\{q_v\}_{v \in \cV})$.
Expanding $\bW=\bPi_U\bM\bPi_V^\top$, the tensor network value is given by
\begin{align*}
&\val_T(\W;\x_{1:k};\y_{1:\ell})\\
&\qquad=\frac{1}{n}\sum_{\balpha \in [m]^\cU}
\sum_{\bi \in [n]^\cV}
\sum_{\bbeta \in [m]^\cE}
\sum_{\bj \in [n]^\cE}
\prod_{u \in \cU} p_u(x_{1:k}[\a_u])
\prod_{v \in \cV} q_v(y_{1:\ell}[i_v])
\prod_{e=(u,v) \in \cE} \Pi_U[\a_u,\b_e]M[\b_e,j_e] \Pi_V[i_v,j_e].
\end{align*}
Write $\bPi_U = \bXi_U\bP_U$ and $\bPi_V=\bXi_V\bP_V$ for the random
sign and permutation matrices defining $\bPi_U,\bPi_V$, and
let $\sigma_U,\sigma_V$ be the permutations of $[m],[n]$ for which
$P_U[\alpha,\sigma_U(\alpha)]=1$ and $P_V[i,\sigma_V(i)]=1$.
Let $\cP=\cP_U \times \cP_V$ be the set of partitions $(\pi_U,\pi_V)$
of $\cU$ and $\cV$, and let
$G_{\pi_U,\pi_V} = (\cK_{\pi_U} \sqcup \cK_{\pi_V}, \cF_{\pi_U,\pi_V})$
be the image of $(\cU,\cV,\cE)$ under $(\pi_U,\pi_V)$ as defined in
Appendix \ref{subsec:whitenoise}. Then
we may obtain analogously to (\ref{eq:orthovalasO})
\begin{align}\label{eq:rectinvval1}
	\val_T(\bW;\bx_{1:k};\by_{1:\ell})
	&=\sum_{(\pi_U,\pi_V)\in\cP} \frac{1}{n}
	\sum_{\balpha\in[m]^{\cK_{\pi_U}}}^*
	\sum_{\bi\in[n]^{\cK_{\pi_V}}}^*
	\prod_{R\in\cK_{\pi_U}} P_R(x_{1:k}[\alpha_R])
\Xi_U[\alpha_R]^{\degext(R)}\notag\\
&\hspace{0.5in}\times
\prod_{S\in\cK_{\pi_V}} Q_S(y_{1:\ell}[i_S]) \Xi_V[i_S]^{\degext(S)}
\prod_{(R,S)\in\cF_{\pi_U,\pi_V}} M[\sigma_U(\alpha_R),\sigma_V(i_S)].
\end{align}
Here $P_R=\prod_{u \in R} p_u$, $Q_S=\prod_{v \in S} q_v$, and
$\degext(R)$ is the number of non-self-loop edges in $\cF_{\pi_U,\pi_V}$
(counting multiplicity) that contain $R$.

\begin{lemma}\label{lemma:Erectinvmoments}
Let $\E$ be the expectation over $\bPi_U,\bPi_V$ conditional on
$\bM,\x_1,\ldots,\x_k,\y_1,\ldots,\y_\ell$. Then
Lemma \ref{lemma:rectinvmoments} holds with
$\E[\val_T(\W;\x_1,\ldots,\x_k;\y_1,\ldots,\y_\ell)]$ in place of
$\val_T(\W;\x_1,\ldots,\x_k;\y_1,\ldots,\y_\ell)$.
\end{lemma}
\begin{proof}
Define analogously to (\ref{eq:An}--\ref{eq:Mn})
\begin{align}
B_n(\pi_U,\pi_V)&=m^{|\cK_{\pi_U}|} \cdot \frac{(m-|\cK_{\pi_U}|)!}{m!} \cdot
n^{|\cK_{\pi_V}|} \cdot \frac{(n-|\cK_{\pi_V}|)!}{n!}\label{eq:Anrect}\\
Q_n(\pi_U,\pi_V)&=\frac{1}{m^{|\cK_{\pi_U}|}} \sum_{\balpha\in[m]^{\cK_{\pi_U}}}^*
\prod_{R \in \cK_{\pi_U}} P_R(x_{1:k}[\alpha_R]) \cdot
\frac{1}{n^{|\cK_{\pi_V}|}} \sum_{\bi \in [n]^{\cK_{\pi_V}}}^*
\prod_{S \in \cK_{\pi_V}} Q_S(y_{1:\ell}[i_S])\label{eq:Qnrect}\\
M_n(\pi_U,\pi_V)&=\frac{1}{n} \sum_{\balpha\in[m]^{\cK_{\pi_U}}}^*
	\sum_{\bi\in[n]^{\cK_{\pi_V}}}^*
	\prod_{(R,S)\in\cF_{\pi_U,\pi_V}} M[\alpha_R,i_S]\label{eq:Mnrect}
\end{align}
Then, taking the expectation conditional on $\bM$ using
the identities \eqref{eq:sign} and \eqref{eq:permutation} yields
\begin{align*}
	\E[\val_T(\W;\x_{1:k};\y_{1:\ell})]
	&=\sum_{(\text{even } \pi_U,\text{ even } \pi_V) \in \cP}
	B_n(\pi_U,\pi_V) \cdot Q_n(\pi_U,\pi_V) \cdot M_n(\pi_U,\pi_V)
\end{align*}

We note that in terms of the symmetric embedding $\widetilde\bM$ from
(\ref{eq:Membedding}), the above quantity $M_n(\pi_U,\pi_V)$ has the equivalent
form
\[M_n(\pi_U,\pi_V) = \frac{1}{n} \sum_{\balpha\in[m]^{\cK_{\pi_U}}}^*
	\sum_{\bi\in[n]^{\cK_{\pi_V}}}^*
	\prod_{(R,S)\in\cF_{\pi_U,\pi_V}} 
	\widetilde M[\alpha_R,m+i_S].\]
We further add edges to $G_{\pi_U,\pi_V}$ by attaching to every vertex in 
$\cK_{\pi_U}$ a self-loop labeled with $p_e(\x)=\bI_m$,
attaching to every vertex in $\cK_{\pi_V}$ a self-loop labeled with
$p_e(\x)=\bI_n$, and labeling each original edge $e=(R,S) \in
\cF_{\pi_U,\pi_V}$ with $p_e(\x)=\x$.
Denote the resulting graph as $\widetilde G_{\pi_U,\pi_V} = 
(\cK_{\pi_U}\sqcup \cK_{\pi_V},\widetilde\cF_{\pi_U,\pi_V})$. Then 
\begin{align*}
	M_n(\pi_U,\pi_V) = \frac{1}{n} \sum_{\bj\in[m+n]^{\cK_{\pi_U}\sqcup\cK_{\pi_V}}}^*
	\prod_{(R,S)\in\widetilde\cF_{\pi_U,\pi_V}} 
	p_e(\widetilde\bM)[j_R, j_{S}].
\end{align*}
By the conditions on $\bM$ in Definition~\ref{def:rectinvariant}, 
$\widetilde\bM$ satisfies the assumptions for Lemma~\ref{lemma:Hlemma_rec}
below. Note that $\widetilde G_{\pi_U,\pi_V}$ is connected, and all external 
degrees are even when
$\pi_U,\pi_V$ are both even. Then, applying Lemma~\ref{lemma:Hlemma_rec},
there exists a value $\widetilde{M}(G_{\pi_U,\pi_V},\cD_\tdiag)$ depending only on 
$G_{\pi_U,\pi_V}$ and $\cD_\tdiag$ such that 
\begin{align*}
\lim_{m,n \to \infty} \frac{n}{n+m} M_n(\pi_U,\pi_V)=
	\lim_{m,n\to\infty} \frac{1}{n+m}
\sum_{\bj\in[m+n]^{\cK_{\pi_U}\sqcup\cK_{\pi_V}}}^*
	\prod_{(R,S)\in\widetilde\cF_{\pi_U,\pi_V}} p_e(\widetilde \bM)[j_R,j_S]
	=\widetilde{M}(G_{\pi_U,\pi_V},\cD_\tdiag)
\end{align*}
Combining with the limit value for $Q_n$, which exists by Lemma
\ref{lemma:distinctness} and the assumptions $\x_{1:k} \toW X_{1:k}$ and
$\y_{1:\ell} \toW Y_{1:\ell}$, we conclude that
$\lim_{m,n\to\infty} \E[\val_T(\W;\x_{1:k};\y_{1:\ell})]$ exists and is given by
\begin{align*}
&\limval_T(\gamma,X_{1:k},Y_{1:\ell},\cD_{\tdiag})\\
&:=\sum_{(\text{even } \pi_U,\text{ even } \pi_V) \in \cP}\;
\prod_{R \in \cK_{\pi_U}} \E[P_R(X_{1:k})]
\prod_{S \in \cK_{\pi_V}} \E[Q_S(Y_{1:\ell})]
\cdot (1+\gamma)\,\widetilde{M}(G_{\pi_U,\pi_V},\cD_\tdiag).
\end{align*}
This depends only on $\gamma$, $T$, the joint laws of $X_{1:k}$ and
$Y_{1:\ell}$, and $\cD_\tdiag$, concluding the proof.
\end{proof}

\begin{lemma}\label{lemma:Hlemma_rec}
Let $\widetilde\Mb \in \R^{(m+n) \times (m+n)}$ be a deterministic symmetric
matrix, such that for any diagonal monomial
$p(\x) \in \Delta\langle \x,\bI_m,\bI_n \rangle$, 
$\lim_{m,n \to \infty} \frac{1}{m+n} \Tr p(\widetilde{\bM})$ exists (and is
finite) almost surely, and for any fixed $\eps>0$ and all large $m,n$,
\[\max_{i\neq j} |p(\widetilde\bM)[i,j]|<n^{-1/2+\eps}.\]
Let $G=(\cK,\cF)$ be a connected multi-graph such that 
the external degree $\degext(R)$ is even for every vertex $R\in\cK$.
For every edge $e\in\cF$, let $p_e(\x)\in\Delta\langle\bx,\bI_m,\bI_n\rangle$ 
be a diagonal monomial labeling this edge.
Then there exists a value $M(G,\cD_\tdiag)$ depending only on
$G$ and $\cD_\tdiag$ such that
\[\lim_{m,n \to \infty} \frac{1}{m+n}
\sum_{\bj\in[m+n]^{\cK}}^*
\prod_{e=(R,R')\in\cF} p_e(\widetilde\bM)[j_R, j_{R'}]
= M(G,\cD_\tdiag).\]
\end{lemma}
\begin{proof}
The proof is identical to that of Lemma~\ref{lemma:Hlemma}.
\end{proof}

\begin{lemma}\label{lemma:rectinvfourthmoment}
Let $\E$ be the expectation over $\bPi_U,\bPi_V$ conditional on
$\bM,\x_1,\ldots,\x_k,\y_1,\ldots,\y_\ell$.
Under the setting of
Lemma \ref{lemma:rectinvmoments}, almost surely as $m,n \to \infty$,
\[\val_T(\W;\x_1,\ldots,\x_k;\y_1,\ldots,\y_\ell)
-\E[\val_T(\W;\x_1,\ldots,\x_k;\y_1,\ldots,\y_\ell)] \to 0.\]
\end{lemma}
\begin{proof}
The proof is similar to that of Lemma \ref{lemma:syminvfourthmoment}.
We write $\val(\W)$ for $\val_T(\W;\x_{1:k};\y_{1:\ell})$,
and $\val(\bar\W)$ for the value corresponding to $\bar\W$ defined by
independent copies $\bar\bPi_U,\bar\bPi_V$ of $\bPi_U,\bPi_V$.

Let $(\cU^{(1)},\cV^{(1)},\cE^{(1)}),\ldots,(\cU^{(4)},\cV^{(4)},\cE^{(4)})$
denote four copies of $(\cU,\cV,\cE)$. For subsets $A \subseteq \{1,\ldots,4\}$,
let $(\cU_A,\cV_A,\cE_A)=\bigsqcup_{a \in A} (\cU^{(a)},\cV^{(a)},\cE^{(a)})$
be the graph obtained as the disjoint union of those copies in $A$.
Let $\cP_A=\cP_{\cU_A} \times \cP_{\cV_A}$ be the set of pairs of
partitions of $\cU_A$ and $\cV_A$, and let $\bar{A}=\{1,2,3,4\} \setminus A$.
Then, we have analogously to (\ref{eq:orthoEvalfourth}),
\begin{align}
\E[(\val(\W)-\val(\bar\W))^4]&=\sum_{A \subseteq \{1,2,3,4\}} (-1)^{|A|}
\mathop{\sum_{(\text{even } \pi_U,\text{ even } \pi_V) \in \cP_A}}_{(\text{even }
\bar\pi_U, \text{ even } \bar\pi_V) \in \cP_{\bar A}}
\frac{n^{|\cC(\pi_U,\pi_V)|+|\cC(\bar\pi_U,\bar\pi_V)|}}{n^4}\notag\\
&\times B_n(\pi_U,\pi_V)B_n(\bar\pi_U,\bar\pi_V)
\cdot Q_n(\pi_U,\pi_V)Q_n(\bar\pi_U,\bar\pi_V)
\cdot M_n(\pi_U,\pi_V)M_n(\bar\pi_U,\bar\pi_V)\label{eq:orthoEvalfourthrect}
\end{align}
Here, $B_n,Q_n$ are as defined in (\ref{eq:Anrect}--\ref{eq:Qnrect}),
$|\cC(\pi_U,\pi_V)|$ denotes the number of connected components of the bipartite
graph $G_{\pi_U,\pi_V}=(\cK_{\pi_U} \sqcup \cK_{\pi_V},\cF_{\pi_U,\pi_V})$,
and we extend the definition (\ref{eq:Mnrect}) to
\[M_n(\pi_U,\pi_V)=\frac{1}{n^{|\cC(\pi_U,\pi_V)|}}
\sum_{\balpha\in[m]^{\cK_{\pi_U}}}^*
\sum_{\bi\in[n]^{\cK_{\pi_V}}}^*\prod_{(R,S)\in\cF_{\pi_U,\pi_V}} M[\alpha_R,i_S]\]
when $G_{\pi_U,\pi_V}$ has more than 1 connected component.

We define $(\tau_U,\tau_V)=(\pi_U \oplus \bar\pi_U,\pi_V \oplus \bar\pi_V)
\in \cP_{\{1,2,3,4\}}$ as the partitions of $\cU_{\{1,2,3,4\}}$ and
$\cV_{\{1,2,3,4\}}$ obtained by combining $\pi_U$ with $\bar\pi_U$
and $\pi_V$ with $\bar\pi_V$. Then the same
arguments as in the proof of Lemma \ref{lemma:syminvfourthmoment} show
\begin{align*}
B_n(\pi_U,\pi_V)B_n(\bar\pi_U,\bar\pi_V)&=B_n(\tau_U,\tau_V)+O(n^{-1})\\
Q_n(\pi_U,\pi_V)Q_n(\bar\pi_U,\bar\pi_V)&=Q_n(\tau_U,\tau_V)+O(n^{-1})\\
M_n(\pi_U,\pi_V)M_n(\bar\pi_U,\bar\pi_V)&=M_n(\tau_U,\tau_V)+O(n^{-1})
\end{align*}
When $G_{\tau_U,\tau_V}$ has 4 connected components, which we denote by
$G_{\tau_U,\tau_V}(a)=(\cK_{\tau_U}(a) \sqcup
\cK_{\tau_V}(a),\cF_{\tau_U,\tau_V}(a))$ for $a=1,2,3,4$,
the arguments of Lemma \ref{lemma:syminvfourthmoment} show also
\begin{align*}
B_n(\pi_U,\pi_V)B_n(\bar\pi_U,\bar\pi_V)&=B_n(\tau_U,\tau_V)
+\sum_{a \in A,b \notin A} B_n(\tau_U,\tau_V,a,b)+O(n^{-2})\\
Q_n(\pi_U,\pi_V)Q_n(\bar\pi_U,\bar\pi_V)&=Q_n(\tau_U,\tau_V)
+\sum_{a \in A,b \notin A} Q_n(\tau_U,\tau_V,a,b)+O(n^{-2})\\
M_n(\pi_U,\pi_V)M_n(\bar\pi_U,\bar\pi_V)&=M_n(\tau_U,\tau_V)
+\sum_{a \in A,b \notin A} M_n(\tau_U,\tau_V,a,b)+O(n^{-2})
\end{align*}
for the quantities
\begin{align*}
B_n(\tau_U,\tau_V,a,b)&=-m^{-1}|\cK_{\tau_U}(a)| \cdot |\cK_{\tau_U}(b)|
-n^{-1}|\cK_{\tau_V}(a)| \cdot |\cK_{\tau_V}(b)|\\
Q_n(\tau_U,\tau_V,a,b)&=\frac{1}{m^{|\cK_{\tau_U}|}}
\frac{1}{n^{|\cK_{\tau_V}|}} \sum_{\balpha \in [m]^{\cK_{\tau_U}}}
\sum_{\bi \in [n]^{\cK_{\tau_V}}} \1\{\text{indices of } \bi \text{ are
distinct and exactly 2 indices of } \balpha\\
&\qquad \text{ from } \cK_{\tau_U}(a),\cK_{\tau_U}(b) \text{ coincide, or
indices of } \balpha \text{ are distinct and exactly 2 indices of } \bi\\
&\qquad \text{ from } \cK_{\tau_V}(a),\cK_{\tau_V}(b) \text{ coincide}\}
\times \prod_{R \in \cK_{\tau_U}} P_R(x_{1:k}[\alpha_R])
\prod_{S \in \cK_{\tau_V}} Q_S(y_{1:\ell}[i_S])\\
M_n(\tau_U,\tau_V,a,b)&=\frac{1}{n}\sum_{(\tau_U',\tau_V') \in
\cP(\tau_U,\tau_V,a,b)} M_n^{**}(\tau_U',\tau_V')
\end{align*}
where $\cP(\tau_U,\tau_V,a,b)$ is the set of partitions $(\tau_U',\tau_V')$
obtained by merging one or more pairs of blocks
$R \in \cK_{\tau_U}(a)$ with $R' \in \cK_{\tau_U}(b)$
or $S \in \cK_{\tau_V}(a)$ with $S' \in \cK_{\tau_V}(b)$,
and $M_n^{**}(\tau_U,\tau_V)$ is the product across all
connected components $G_{\sigma_U,\sigma_V} \subseteq G_{\tau_U,\tau_V}$ of
$M_n(\sigma_U,\sigma_V)$.

Then the same argument as in Lemma \ref{lemma:syminvfourthmoment} 
shows $\E[(\val(\W)-\E \val(\W))^4] \leq
\E[(\val(\W)-\val(\bar\W))^4] \leq C/n^2$ for a constant $C>0$ and all
large $n$, so that the result follows by
Markov's inequality and the Borel-Cantelli lemma.
\end{proof}

Lemmas \ref{lemma:Erectinvmoments} and \ref{lemma:rectinvfourthmoment} show Lemma
\ref{lemma:rectinvmoments}, and combining Lemmas \ref{lemma:rectcompare} 
and \ref{lemma:rectinvmoments} concludes the proof of Theorem~\ref{thm:rectinvariant}.

\subsection*{Acknowledgments}

We'd like to thank Mark Sellke for asking us an interesting question about
asymptotic freeness for random tensors, and Zhigang Bao and Yuxin Chen for
asking about AMP algorithms for heteroskedastic variances, which motivated parts
of this work. XZ was supported in part by funding from Two Sigma Investments,
LP. ZF was supported in part by NSF DMS-1916198 and DMS-2142476.

\bibliography{reference}

\newcommand{\etalchar}[1]{$^{#1}$}
\begin{thebibliography}{BDM{\etalchar{+}}16}

\bibitem[AB73]{arabie1973multidimensional}
Phipps Arabie and Scott~A Boorman.
\newblock Multidimensional scaling of measures of distance between partitions.
\newblock {\em Journal of Mathematical Psychology}, 10(2):148--203, 1973.

\bibitem[ACD{\etalchar{+}}21]{au2021freeness}
Benson Au, Guillaume C{\'e}bron, Antoine Dahlqvist, Franck Gabriel, and Camille Male.
\newblock Freeness over the diagonal for large random matrices.
\newblock {\em Annals of Probability}, 49(1):157--179, 2021.

\bibitem[AF14]{anderson2014asymptotically}
Greg~W Anderson and Brendan Farrell.
\newblock Asymptotically liberating sequences of random unitary matrices.
\newblock {\em Advances in Mathematics}, 255:381--413, 2014.

\bibitem[AZ06]{anderson2006clt}
Greg~W Anderson and Ofer Zeitouni.
\newblock A {CLT} for a band matrix model.
\newblock {\em Probability Theory and Related Fields}, 134(2):283--338, 2006.

\bibitem[BDM{\etalchar{+}}16]{dia2016mutual}
Jean Barbier, Mohamad Dia, Nicolas Macris, Florent Krzakala, Thibault Lesieur, and Lenka Zdeborov\'{a}.
\newblock Mutual information for symmetric rank-one matrix estimation: A proof of the replica formula.
\newblock In {\em Proceedings of the 30th International Conference on Neural Information Processing Systems}, NIPS'16, page 424–432, Red Hook, NY, USA, 2016. Curran Associates Inc.

\bibitem[BGBK20]{benaych2020spectral}
Florent Benaych-Georges, Charles Bordenave, and Antti Knowles.
\newblock Spectral radii of sparse random matrices.
\newblock {\em Annales de l'Institut Henri Poincar{\'e}, Probabilit{\'e}s et Statistiques}, 56(3):2141--2161, 2020.

\bibitem[Bil95]{billingsley1995probability}
Patrick Billingsley.
\newblock {\em Probability and measure}.
\newblock John Wiley \& Sons, 3 edition, 1995.

\bibitem[BKRS19]{bu2019algorithmic}
Zhiqi Bu, Jason Klusowski, Cynthia Rush, and Weijie Su.
\newblock Algorithmic analysis and statistical estimation of slope via approximate message passing.
\newblock {\em Advances in Neural Information Processing Systems}, 32, 2019.

\bibitem[BLM15]{bayati2015universality}
Mohsen Bayati, Marc Lelarge, and Andrea Montanari.
\newblock Universality in polytope phase transitions and message passing algorithms.
\newblock {\em The Annals of Applied Probability}, 25(2):753--822, 2015.

\bibitem[BM11]{bayati2011dynamics}
Mohsen Bayati and Andrea Montanari.
\newblock The dynamics of message passing on dense graphs, with applications to compressed sensing.
\newblock {\em IEEE Transactions on Information Theory}, 57(2):764--785, 2011.

\bibitem[BMN20]{berthier2020state}
Raphael Berthier, Andrea Montanari, and Phan-Minh Nguyen.
\newblock State evolution for approximate message passing with non-separable functions.
\newblock {\em Information and Inference: A Journal of the IMA}, 9(1):33--79, 2020.

\bibitem[BNSX21]{bolthausen2021gardner}
Erwin Bolthausen, Shuta Nakajima, Nike Sun, and Changji Xu.
\newblock Gardner formula for {I}sing perceptron models at small densities.
\newblock {\em arXiv preprint arXiv:2111.02855}, 2021.

\bibitem[BO73]{boorman1973metrics}
Scott~A Boorman and Donald~C Olivier.
\newblock Metrics on spaces of finite trees.
\newblock {\em Journal of Mathematical Psychology}, 10(1):26--59, 1973.

\bibitem[Bol14]{bolthausen2014iterative}
Erwin Bolthausen.
\newblock An iterative construction of solutions of the {TAP} equations for the {S}herrington--{K}irkpatrick model.
\newblock {\em Communications in Mathematical Physics}, 325(1):333--366, 2014.

\bibitem[Bol18]{bolthausen2018morita}
Erwin Bolthausen.
\newblock A {M}orita type proof of the replica-symmetric formula for {SK}.
\newblock In {\em International Conference on Statistical Mechanics of Classical and Disordered Systems}, pages 63--93. Springer, 2018.

\bibitem[CCM21]{celentano2021high}
Michael Celentano, Chen Cheng, and Andrea Montanari.
\newblock The high-dimensional asymptotics of first order methods with random data.
\newblock {\em arXiv preprint arXiv:2112.07572}, 2021.

\bibitem[CL21]{chen2021universality}
Wei-Kuo Chen and Wai-Kit Lam.
\newblock Universality of approximate message passing algorithms.
\newblock {\em Electronic Journal of Probability}, 26:1--44, 2021.

\bibitem[CMW20]{celentano2020estimation}
Michael Celentano, Andrea Montanari, and Yuchen Wu.
\newblock The estimation error of general first order methods.
\newblock In {\em Conference on Learning Theory}, pages 1078--1141. PMLR, 2020.

\bibitem[CO19]{ccakmak2019memory}
Burak Cakmak and Manfred Opper.
\newblock {M}emory-free dynamics for the {T}houless-{A}nderson-{P}almer equations of ising models with arbitrary rotation-invariant ensembles of random coupling matrices.
\newblock {\em Physical Review E}, 99(6):062140, 2019.

\bibitem[CO20]{ccakmak2020dynamical}
Burak Cakmak and Manfred Opper.
\newblock A dynamical mean-field theory for learning in restricted boltzmann machines.
\newblock {\em Journal of Statistical Mechanics: Theory and Experiment}, 2020(10):103303, 2020.

\bibitem[CR23]{cademartori2023non}
Collin Cademartori and Cynthia Rush.
\newblock A non-asymptotic analysis of generalized approximate message passing algorithms with right rotationally invariant designs.
\newblock {\em arXiv preprint arXiv:2302.00088}, 2023.

\bibitem[C{\'S}06]{collins2006integration}
Beno{\^\i}t Collins and Piotr {\'S}niady.
\newblock Integration with respect to the haar measure on unitary, orthogonal and symplectic group.
\newblock {\em Communications in Mathematical Physics}, 264(3):773--795, 2006.

\bibitem[CWF14]{cakmak2014s}
Burak Cakmak, Ole Winther, and Bernard~H Fleury.
\newblock {S}-{AMP}: {A}pproximate message passing for general matrix ensembles.
\newblock In {\em 2014 IEEE Information Theory Workshop (ITW 2014)}, pages 192--196. IEEE, 2014.

\bibitem[DAM17]{deshpande2017asymptotic}
Yash Deshpande, Emmanuel Abbe, and Andrea Montanari.
\newblock Asymptotic mutual information for the balanced binary stochastic block model.
\newblock {\em Information and Inference: A Journal of the IMA}, 6(2):125--170, 2017.

\bibitem[DB20]{dudeja2020universality}
Rishabh Dudeja and Milad Bakhshizadeh.
\newblock Universality of linearized message passing for phase retrieval with structured sensing matrices.
\newblock {\em arXiv preprint arXiv:2008.10503}, 2020.

\bibitem[DJM13]{donoho2013information}
David~L Donoho, Adel Javanmard, and Andrea Montanari.
\newblock Information-theoretically optimal compressed sensing via spatial coupling and approximate message passing.
\newblock {\em IEEE transactions on information theory}, 59(11):7434--7464, 2013.

\bibitem[DLS22]{dudeja2022universality}
Rishabh Dudeja, Yue~M Lu, and Subhabrata Sen.
\newblock Universality of approximate message passing with semi-random matrices.
\newblock {\em arXiv preprint arXiv:2204.04281}, 2022.

\bibitem[DM16]{donoho2016high}
David Donoho and Andrea Montanari.
\newblock High dimensional robust m-estimation: {A}symptotic variance via approximate message passing.
\newblock {\em Probability Theory and Related Fields}, 166(3):935--969, 2016.

\bibitem[DMM09]{donoho2009message}
David~L Donoho, Arian Maleki, and Andrea Montanari.
\newblock Message-passing algorithms for compressed sensing.
\newblock {\em Proceedings of the National Academy of Sciences}, 106(45):18914--18919, 2009.

\bibitem[DMM10a]{donoho2010messagea}
David~L Donoho, Arian Maleki, and Andrea Montanari.
\newblock Message passing algorithms for compressed sensing: {I}. {M}otivation and construction.
\newblock In {\em 2010 IEEE information theory workshop on information theory (ITW 2010, Cairo)}, pages 1--5. IEEE, 2010.

\bibitem[DMM10b]{donoho2010messageb}
David~L Donoho, Arian Maleki, and Andrea Montanari.
\newblock Message passing algorithms for compressed sensing: {II}. {A}nalysis and validation.
\newblock In {\em 2010 IEEE Information Theory Workshop on Information Theory (ITW 2010, Cairo)}, pages 1--5. IEEE, 2010.

\bibitem[DS19]{ding2019capacity}
Jian Ding and Nike Sun.
\newblock Capacity lower bound for the {I}sing perceptron.
\newblock In {\em Proceedings of the 51st Annual ACM SIGACT Symposium on Theory of Computing}, pages 816--827, 2019.

\bibitem[DSL22]{dudeja2022spectral}
Rishabh Dudeja, Subhabrata Sen, and Yue~M Lu.
\newblock Spectral universality of regularized linear regression with nearly deterministic sensing matrices.
\newblock {\em arXiv preprint arXiv:2208.02753}, 2022.

\bibitem[DT05]{donoho2005neighborliness}
David~L Donoho and Jared Tanner.
\newblock Neighborliness of randomly projected simplices in high dimensions.
\newblock {\em Proceedings of the National Academy of Sciences}, 102(27):9452--9457, 2005.

\bibitem[DT09]{donoho2009observed}
David Donoho and Jared Tanner.
\newblock Observed universality of phase transitions in high-dimensional geometry, with implications for modern data analysis and signal processing.
\newblock {\em Philosophical Transactions of the Royal Society A: Mathematical, Physical and Engineering Sciences}, 367(1906):4273--4293, 2009.

\bibitem[EYY12a]{erdHos2012bulk}
L{\'a}szl{\'o} Erd{\H{o}}s, Horng-Tzer Yau, and Jun Yin.
\newblock Bulk universality for generalized {W}igner matrices.
\newblock {\em Probability Theory and Related Fields}, 154(1):341--407, 2012.

\bibitem[EYY12b]{erdHos2012rigidity}
L{\'a}szl{\'o} Erd{\H{o}}s, Horng-Tzer Yau, and Jun Yin.
\newblock Rigidity of eigenvalues of generalized {W}igner matrices.
\newblock {\em Advances in Mathematics}, 229(3):1435--1515, 2012.

\bibitem[Fan22]{fan2022approximate}
Zhou Fan.
\newblock Approximate message passing algorithms for rotationally invariant matrices.
\newblock {\em The Annals of Statistics}, 50(1):197--224, 2022.

\bibitem[FLS22]{fan2022tap}
Zhou Fan, Yufan Li, and Subhabrata Sen.
\newblock Tap equations for orthogonally invariant spin glasses at high temperature.
\newblock {\em arXiv preprint arXiv:2202.09325}, 2022.

\bibitem[FVRS22]{feng2022unifying}
Oliver~Y. Feng, Ramji Venkataramanan, Cynthia Rush, and Richard~J. Samworth.
\newblock A unifying tutorial on approximate message passing.
\newblock {\em Foundations and Trends® in Machine Learning}, 15(4):335--536, 2022.

\bibitem[FW21]{fan2021replica}
Zhou Fan and Yihong Wu.
\newblock The replica-symmetric free energy for ising spin glasses with orthogonally invariant couplings.
\newblock {\em arXiv preprint arXiv:2105.02797}, 2021.

\bibitem[GB21]{gerbelot2021graph}
C{\'e}dric Gerbelot and Rapha{\"e}l Berthier.
\newblock Graph-based approximate message passing iterations.
\newblock {\em arXiv preprint arXiv:2109.11905}, 2021.

\bibitem[HS19]{hafemeister2019normalization}
Christoph Hafemeister and Rahul Satija.
\newblock Normalization and variance stabilization of single-cell {RNA}-seq data using regularized negative binomial regression.
\newblock {\em Genome biology}, 20(1):1--15, 2019.

\bibitem[Jia05]{jiang2005maxima}
Tiefeng Jiang.
\newblock Maxima of entries of haar distributed matrices.
\newblock {\em Probability Theory and Related Fields}, 131:121--144, 2005.

\bibitem[JM13]{javanmard2013state}
Adel Javanmard and Andrea Montanari.
\newblock State evolution for general approximate message passing algorithms, with applications to spatial coupling.
\newblock {\em Information and Inference: A Journal of the IMA}, 2(2):115--144, 2013.

\bibitem[Kab03]{kabashima2003cdma}
Yoshiyuki Kabashima.
\newblock A {CDMA} multiuser detection algorithm on the basis of belief propagation.
\newblock {\em Journal of Physics A: Mathematical and General}, 36(43):11111, 2003.

\bibitem[LFW23]{li2023approximate}
Gen Li, Wei Fan, and Yuting Wei.
\newblock Approximate message passing from random initialization with applications to $\mathbb{Z}_2$ synchronization.
\newblock {\em arXiv preprint arXiv:2302.03682}, 2023.

\bibitem[LHK21]{liu2021memory}
Lei Liu, Shunqi Huang, and Brian~M Kurkoski.
\newblock Memory approximate message passing.
\newblock In {\em 2021 IEEE International Symposium on Information Theory (ISIT)}, pages 1379--1384. IEEE, 2021.

\bibitem[LW21]{li2021minimum}
Yue Li and Yuting Wei.
\newblock Minimum $\ell_1$-norm interpolators: Precise asymptotics and multiple descent.
\newblock {\em arXiv preprint arXiv:2110.09502}, 2021.

\bibitem[LW22]{li2022non}
Gen Li and Yuting Wei.
\newblock A non-asymptotic framework for approximate message passing in spiked models.
\newblock {\em arXiv preprint arXiv:2208.03313}, 2022.

\bibitem[LZK21]{landa2021biwhitening}
Boris Landa, Thomas~TCK Zhang, and Yuval Kluger.
\newblock Biwhitening reveals the rank of a count matrix.
\newblock {\em arXiv preprint arXiv:2103.13840}, 2021.

\bibitem[Mal20]{male2020traffic}
Camille Male.
\newblock {\em Traffic distributions and independence: permutation invariant random matrices and the three notions of independence}, volume 267.
\newblock American mathematical society, 2020.

\bibitem[MJG{\etalchar{+}}13]{monajemi2013deterministic}
Hatef Monajemi, Sina Jafarpour, Matan Gavish, et~al.
\newblock Deterministic matrices matching the compressed sensing phase transitions of gaussian random matrices.
\newblock {\em Proceedings of the National Academy of Sciences}, 110(4):1181--1186, 2013.

\bibitem[Mon21]{montanari2021optimization}
Andrea Montanari.
\newblock Optimization of the {S}herrington--{K}irkpatrick {H}amiltonian.
\newblock {\em SIAM Journal on Computing}, 2021.

\bibitem[MP17]{ma2017orthogonal}
Junjie Ma and Li~Ping.
\newblock Orthogonal amp.
\newblock {\em IEEE Access}, 5:2020--2033, 2017.

\bibitem[MS12]{mingo2012sharp}
James~A Mingo and Roland Speicher.
\newblock Sharp bounds for sums associated to graphs of matrices.
\newblock {\em Journal of Functional Analysis}, 262(5):2272--2288, 2012.

\bibitem[MS17]{mingo2017free}
James~A Mingo and Roland Speicher.
\newblock {\em Free probability and random matrices}, volume~35.
\newblock Springer, 2017.

\bibitem[MV21a]{mondelli2021pca}
Marco Mondelli and Ramji Venkataramanan.
\newblock {PCA} initialization for approximate message passing in rotationally invariant models.
\newblock {\em Advances in Neural Information Processing Systems}, 34:29616--29629, 2021.

\bibitem[MV21b]{montanari2021estimation}
Andrea Montanari and Ramji Venkataramanan.
\newblock Estimation of low-rank matrices via approximate message passing.
\newblock {\em The Annals of Statistics}, 49(1):321--345, 2021.

\bibitem[NS06]{nica2006lectures}
Alexandru Nica and Roland Speicher.
\newblock {\em Lectures on the combinatorics of free probability}, volume~13.
\newblock Cambridge University Press, 2006.

\bibitem[OCW16]{opper2016theory}
Manfred Opper, Burak Cakmak, and Ole Winther.
\newblock A theory of solving tap equations for ising models with general invariant random matrices.
\newblock {\em Journal of Physics A: Mathematical and Theoretical}, 49(11):114002, 2016.

\bibitem[Ran11]{rangan2011generalized}
Sundeep Rangan.
\newblock Generalized approximate message passing for estimation with random linear mixing.
\newblock In {\em 2011 IEEE International Symposium on Information Theory Proceedings}, pages 2168--2172. IEEE, 2011.

\bibitem[RF12]{rangan2012iterative}
Sundeep Rangan and Alyson~K Fletcher.
\newblock Iterative estimation of constrained rank-one matrices in noise.
\newblock In {\em 2012 IEEE International Symposium on Information Theory Proceedings}, pages 1246--1250. IEEE, 2012.

\bibitem[RSF19]{rangan2019vector}
Sundeep Rangan, Philip Schniter, and Alyson~K Fletcher.
\newblock Vector approximate message passing.
\newblock {\em IEEE Transactions on Information Theory}, 65(10):6664--6684, 2019.

\bibitem[RV18]{rush2018finite}
Cynthia Rush and Ramji Venkataramanan.
\newblock Finite sample analysis of approximate message passing algorithms.
\newblock {\em IEEE Transactions on Information Theory}, 64(11):7264--7286, 2018.

\bibitem[SCC19]{sur2019likelihood}
Pragya Sur, Yuxin Chen, and Emmanuel~J Cand{\`e}s.
\newblock The likelihood ratio test in high-dimensional logistic regression is asymptotically a rescaled chi-square.
\newblock {\em Probability theory and related fields}, 175(1):487--558, 2019.

\bibitem[Sch17]{schmudgen2017moment}
Konrad Schm{\"u}dgen.
\newblock {\em The moment problem}, volume~9.
\newblock Springer, 2017.

\bibitem[SRF16]{schniter2016vector}
Philip Schniter, Sundeep Rangan, and Alyson~K Fletcher.
\newblock Vector approximate message passing for the generalized linear model.
\newblock In {\em 2016 50th Asilomar Conference on Signals, Systems and Computers}, pages 1525--1529. IEEE, 2016.

\bibitem[SS21]{sarkar2021separating}
Abhishek Sarkar and Matthew Stephens.
\newblock Separating measurement and expression models clarifies confusion in single-cell {RNA} sequencing analysis.
\newblock {\em Nature genetics}, 53(6):770--777, 2021.

\bibitem[Tak17]{takeuchi2017rigorous}
Keigo Takeuchi.
\newblock Rigorous dynamics of expectation-propagation-based signal recovery from unitarily invariant measurements.
\newblock In {\em 2017 IEEE International Symposium on Information Theory (ISIT)}, pages 501--505. IEEE, 2017.

\bibitem[Tak20]{takeuchi2020convolutional}
Keigo Takeuchi.
\newblock Convolutional approximate message-passing.
\newblock {\em IEEE Signal Processing Letters}, 27:416--420, 2020.

\bibitem[Tak21]{takeuchi2021bayes}
Keigo Takeuchi.
\newblock Bayes-optimal convolutional {AMP}.
\newblock In {\em 2021 IEEE International Symposium on Information Theory (ISIT)}, pages 1385--1390. IEEE, 2021.

\bibitem[THAI19]{townes2019feature}
F~William Townes, Stephanie~C Hicks, Martin~J Aryee, and Rafael~A Irizarry.
\newblock Feature selection and dimension reduction for single-cell {RNA}-seq based on a multinomial model.
\newblock {\em Genome biology}, 20(1):1--16, 2019.

\bibitem[VDN92]{voiculescu1992free}
Dan~V Voiculescu, Ken~J Dykema, and Alexandru Nica.
\newblock {\em Free random variables}.
\newblock American Mathematical Soc., 1992.

\bibitem[Vil09]{villani2009optimal}
C{\'e}dric Villani.
\newblock {\em Optimal transport: old and new}, volume 338.
\newblock Springer, 2009.

\bibitem[ZWF21]{zhong2021approximate}
Xinyi Zhong, Tianhao Wang, and Zhou Fan.
\newblock Approximate {M}essage {P}assing for orthogonally invariant ensembles: {M}ultivariate non-linearities and spectral initialization.
\newblock {\em arXiv preprint arXiv:2110.02318}, 2021.

\end{thebibliography}
\bibliographystyle{alpha}

\end{document}